\documentclass[numbers=enddot,12pt,final,onecolumn,notitlepage]{scrartcl}%
\usepackage[headsepline,footsepline,manualmark]{scrlayer-scrpage}
\usepackage[all,cmtip]{xy}
\usepackage{amssymb}
\usepackage{amsmath}
\usepackage{amsthm}
\usepackage{framed}
\usepackage{comment}
\usepackage{color}
\usepackage[breaklinks=True]{hyperref}
\usepackage[sc]{mathpazo}
\usepackage[T1]{fontenc}
\usepackage{needspace}
\usepackage{tabls}
\usepackage{ytableau}
\providecommand{\U}[1]{\protect\rule{.1in}{.1in}}
\newcounter{exer}

\numberwithin{exer}{section}
\theoremstyle{definition}
\newtheorem{theo}{Theorem}[section]
\newenvironment{theorem}[1][]
{\begin{theo}[#1]\begin{leftbar}}
{\end{leftbar}\end{theo}}
\newtheorem{lem}[theo]{Lemma}
\newenvironment{lemma}[1][]
{\begin{lem}[#1]\begin{leftbar}}
{\end{leftbar}\end{lem}}
\newtheorem{prop}[theo]{Proposition}
\newenvironment{proposition}[1][]
{\begin{prop}[#1]\begin{leftbar}}
{\end{leftbar}\end{prop}}
\newtheorem{defi}[theo]{Definition}

\newtheorem{remk}[theo]{Remark}
\newenvironment{remark}[1][]
{\begin{remk}[#1]\begin{leftbar}}
{\end{leftbar}\end{remk}}
\newtheorem{coro}[theo]{Corollary}

\newtheorem{conv}[theo]{Convention}
\newenvironment{convention}[1][]
{\begin{conv}[#1]\begin{leftbar}}
{\end{leftbar}\end{conv}}
\newtheorem{quest}[theo]{Question}
\newenvironment{question}[1][]
{\begin{quest}[#1]\begin{leftbar}}
{\end{leftbar}\end{quest}}
\newtheorem{warn}[theo]{Warning}
\newenvironment{warning}[1][]
{\begin{warn}[#1]\begin{leftbar}}
{\end{leftbar}\end{warn}}
\newtheorem{conj}[theo]{Conjecture}

\newtheorem{exam}[theo]{Example}

\newtheorem{exmp}[exer]{Exercise}

\newenvironment{statement}{\begin{quote}}{\end{quote}}

\let\sumnonlimits\sum
\let\prodnonlimits\prod
\let\cupnonlimits\bigcup
\let\capnonlimits\bigcap
\renewcommand{\sum}{\sumnonlimits\limits}
\renewcommand{\prod}{\prodnonlimits\limits}
\renewcommand{\bigcup}{\cupnonlimits\limits}
\renewcommand{\bigcap}{\capnonlimits\limits}
\excludecomment{verlong}
\includecomment{vershort}
\excludecomment{noncompile}

\newcommand{\arinj}{\ar@{_{(}->}}
\newcommand{\arinjrev}{\ar@{^{(}->}}
\newcommand{\arsurj}{\ar@{->>}}
\newcommand{\arelem}{\ar@{|->}}
\newcommand{\arback}{\ar@{<-}}
\newcommand{\symd}{\mathbin{\bigtriangleup}}

\newcommand{\sslash}{\mathbin{/\mkern-5mu/}}
\definecolor{darkgreen}{rgb}{0,.5,0}
\newtheoremstyle{plainsl}
{8pt plus 2pt minus 4pt}
{8pt plus 2pt minus 4pt}
{\slshape}
{0pt}
{\bfseries}
{.}
{5pt plus 1pt minus 1pt}
{}
\theoremstyle{plainsl}
\ihead{The $V/L$ recursion for Macdonald-7s, version June 1, 2026}
\ohead{page \thepage}
\begin{document}

\title{The $V/L$ recursion for Macdonald's 7th Variation Schur polynomials}
\author{Darij Grinberg}
\date{June 1, 2026}
\maketitle

\begin{abstract}
\textbf{Abstract.} We generalize and prove the recursive relation
\[
S_{\lambda}\left(  V\right)  =\sum_{L\subseteq V\text{ line}}S_{\lambda
}\left(  V \mathbin{/\mkern-5mu/} L\right)
\]
conjectured by I. G. Macdonald for his \textquotedblleft7th
variation\textquotedblright\ of the Schur functions. This variation is a
family of polynomials over a finite field that mimic the (straight and skew)
Schur polynomials using powers of the Frobenius.

\end{abstract}

\section*{***}

In \cite{Macdon92}, Ian Macdonald discusses nine variations on the Schur
functions. One of them -- the 7th Variation in his numbering -- can be
described as a \textquotedblleft function field version\textquotedblright, in
that it is defined over a finite field $\mathbb{F}_{q}$ and exhibits
$\operatorname*{GL}\left(  V\right)  $-symmetry rather than merely $S_{n}%
$-symmetry. Macdonald proves several properties of this family, including
Jacobi--Trudi formulas, a sum-over-tableaux formula, and a dual-Cauchy-like identity.

In this note, we shall prove a further property, which Macdonald left unproved
\cite[(7.25 ?)]{Macdon92}: a recursion for the 7th Variation Schur polynomials
$S_{\lambda}\left(  V\right)  $. We will in fact generalize it to the skew
version $S_{\lambda/\mu}\left(  V\right)  $ of these polynomials, showing that
any finite-dimensional vector subspace $V$ of a commutative $F$-algebra and
any two partitions $\lambda$ and $\mu$ of length $<\dim V$ satisfy%
\[
S_{\lambda/\mu}\left(  V\right)  =\sum_{\substack{L\subseteq V\text{ line
(i.e.,}\\1\text{-dimensional subspace)}}}S_{\lambda/\mu}\left(  V\sslash
L\right)  ,
\]
where we assume that $\mathbf{A}$ is an integral domain with an invertible
Frobenius morphism (so that $S_{\lambda/\mu}$ is well-defined) (Theorem
\ref{thm.skew-V/L}). Here, $V\sslash L$ denotes the internal quotient of $V$
by $L$, consisting of the products of elements of all cosets of $L$ in $V$.
Along the way, we will show a Pieri-like recursion (Lemma \ref{lem.SlamVL-ie}%
). After proving the above recursion, we will use it to show the explicit
formula%
\[
S_{\lambda}\left(  V\right)  =\sum_{\substack{V=V_{0}>V_{1}>\cdots
>V_{n}=0\\\text{is a complete flag in }V}}\ \ \prod_{i=1}^{n}H_{\lambda_{i}%
}\left(  V_{i-1}\sslash V_{i}\right)
\]
(Theorem \ref{thm.7.24}), spelling out the argument left to the reader by
Macdonald in \cite{Macdon92}. (In the appendix, we will also spell out some
other implicit arguments from \cite{Macdon92}.)

\begin{statement}
\textbf{Update:} After posting the first version of this note, I have learned
that a slight generalization of Theorem \ref{thm.skew-V/L} was proved by Hoang
in 2025 \cite[Theorem 1.24]{Hoang25}. He also gave a proof of (an equivalent
form of) Theorem \ref{thm.7.24} in \cite[Theorem 1.26]{Hoang25}. I was unaware
of his work when I wrote this note, and would like to thank him for informing
me of it and making it available. The proofs I give below are largely the
same, although more detailed and more explicit about the algebraic
infrastructure built upon.
\end{statement}

\section{Definitions and results}

We will study the 7th Variation of Schur functions \cite[7th Variation]%
{Macdon92}. First, we shall recall the relevant definitions (following
\cite[errata]{Macdon92} for the right level of generality).

\subsection{Macdonald's 7th Variation}

We let $\mathbb{N}:=\left\{  0,1,2,\ldots\right\}  $, and we set $\left[
n\right]  :=\left\{  1,2,\ldots,n\right\}  $ for each $n\in\mathbb{N}$.

Fix a finite field $F:=\mathbb{F}_{q}$ and any commutative $F$-algebra
$\mathbf{A}$. If $x_{1},x_{2},\ldots,x_{n}\in\mathbf{A}$ are elements, and if
$\alpha=\left(  \alpha_{1},\alpha_{2},\ldots,\alpha_{n}\right)  \in
\mathbb{N}^{n}$ is an $n$-tuple of nonnegative integers, then we define the
$\alpha$\emph{-alternant}%
\begin{equation}
A_{\alpha}:=\det\left(  x_{i}^{q^{\alpha_{j}}}\right)  _{i,j\in\left[
n\right]  } \label{eq.Aalpha=}%
\end{equation}
in $\mathbf{A}$. We also define the special $n$-tuple%
\[
\delta:=\left(  n-1,n-2,\ldots,1,0\right)  \in\mathbb{N}^{n}.
\]
We define the addition of $n$-tuples in $\mathbb{N}^{n}$ entrywise (i.e., if
$\alpha,\beta\in\mathbb{N}^{n}$ are two $n$-tuples, then $\alpha+\beta$
denotes the $n$-tuple whose $i$-th entry is the sum of the $i$-th entries of
$\alpha$ and $\beta$). Recall that a partition is a weakly decreasing finite
tuple of positive integers. Any partition of length $\leq n$ is identified
with an $n$-tuple by inserting trailing zeroes at its end. For any partition
$\lambda=\left(  \lambda_{1},\lambda_{2},\ldots,\lambda_{n}\right)
\in\mathbb{N}^{n}$ and any $n$ elements $x_{1},x_{2},\ldots,x_{n}\in
\mathbf{A}$, we set%
\[
S_{\lambda}\left(  x_{1},x_{2},\ldots,x_{n}\right)  :=A_{\lambda+\delta
}/A_{\delta},
\]
where the quotient is taken \textquotedblleft universally\textquotedblright%
\ (i.e., we treat the $x_{1},x_{2},\ldots,x_{n}$ as independent
indeterminates, then compute the quotient $A_{\lambda+\delta}/A_{\delta}$ in
the polynomial ring $F\left[  x_{1},x_{2},\ldots,x_{n}\right]  $, and only
then substitute the original values of $x_{1},x_{2},\ldots,x_{n}$ back). This
is a polynomial in $x_{1},x_{2},\ldots,x_{n}$, and can be easily seen to be
$\operatorname*{GL}\nolimits_{n}$-invariant, i.e., we have%
\begin{equation}
S_{\lambda}\left(  x_{1},x_{2},\ldots,x_{n}\right)  =S_{\lambda}\left(
y_{1},y_{2},\ldots,y_{n}\right)  \label{eq.Slam.GL-inv}%
\end{equation}
whenever $\left(  x_{1},x_{2},\ldots,x_{n}\right)  $ and $\left(  y_{1}%
,y_{2},\ldots,y_{n}\right)  $ are two bases of the same $F$-vector subspace of
$\mathbf{A}$ (see \cite[first paragraph on page 24]{Macdon92} for the proof).
Thus, for any finite-dimensional $F$-vector subspace $V$ of $\mathbf{A}$ and
any partition $\lambda$ of length $\leq\dim V$, we can set%
\begin{equation}
S_{\lambda}\left(  V\right)  :=S_{\lambda}\left(  x_{1},x_{2},\ldots
,x_{n}\right)  , \label{eq.Slam.SlamV1}%
\end{equation}
where $\left(  x_{1},x_{2},\ldots,x_{n}\right)  $ is an arbitrary basis of $V$
(the result will not depend on the choice of this basis). We also set
\begin{equation}
S_{\lambda}\left(  V\right)  :=0\qquad\text{if }\lambda\text{ is a partition
of length }>\dim V. \label{eq.Slam.SlamV0}%
\end{equation}

These $S_{\lambda}\left(  V\right)  $ are Macdonald's \emph{7th variation} of
the Schur functions. He introduced them in \cite[7th Variation]{Macdon92} and
proved several of their properties, in particular defining a skew version
$S_{\lambda/\mu}\left(  V\right)  $, to which we shall come later. The problem
we are interested in involves subspaces and \textquotedblleft internal
quotients\textquotedblright, which we shall introduce now.

\subsection{Internal quotients}

If $U$ is a finite-dimensional $F$-vector subspace of $\mathbf{A}$, then the
polynomial%
\[
f_{U}\left(  t\right)  :=\prod_{u\in U}\left(  t+u\right)  \in\mathbf{A}%
\left[  t\right]
\]
is a $q$-polynomial, i.e., it is an $F$-linear combination of the monomials
$t^{q^{i}}$ for $i\in\mathbb{N}$ (see \cite[(7.7) and (7.8)]{Macdon92} or
\cite[Theorem 1.6]{ppoly}). In particular, it satisfies $f_{U}\left(
x+y\right)  =f_{U}\left(  x\right)  +f_{U}\left(  y\right)  $ in the
polynomial ring $\mathbf{A}\left[  x,y\right]  $ as well as $f_{U}\left(
\lambda t\right)  =\lambda f_{U}\left(  t\right)  $ for each $\lambda\in F$
(since each $\lambda\in F$ satisfies $\lambda^{q}=\lambda$). Thus, this
polynomial $f_{U}$ defines an $F$-linear map%
\begin{align*}
\widetilde{f}_{U}:\mathbf{A}  &  \rightarrow\mathbf{A},\\
v  &  \mapsto f_{U}\left(  v\right)  ,
\end{align*}
which contains $U$ in its kernel (because for each $v\in U$, we have
$\widetilde{f}_{U}\left(  v\right)  =f_{U}\left(  v\right)  =\prod_{u\in
U}\underbrace{\left(  v+u\right)  }_{=0\text{ for }u=-v}=0$). When
$\mathbf{A}$ is an integral domain, the kernel of this map $\widetilde{f}_{U}$
is precisely $U$.

Now, if $U\subseteq V$ are two finite-dimensional $F$-vector subspaces of
$\mathbf{A}$, then we define the \emph{internal quotient} $V\sslash U$ to be
the image $\widetilde{f}_{U}\left(  V\right)  $ of the subspace $V$ under the
$F$-linear map $\widetilde{f}_{U}:\mathbf{A}\rightarrow\mathbf{A}$. This is an
$F$-vector subspace of $\mathbf{A}$ again. If $\mathbf{A}$ is an integral
domain, then this internal quotient $V\sslash U$ is isomorphic to the actual
quotient $V/U$ (by the first isomorphism theorem, since the $F$-linear map
$\left.  \widetilde{f}_{U}\mid_{V}\right.  :V\rightarrow\mathbf{A}$ has kernel
$U$ and image $V\sslash U$); thus, Macdonald simply denotes it by $V/U$. We
shall stick with the safer notation $V\sslash U$, however. Either way, the
design behind $V\sslash U$ is for $V\sslash U$ to act as a \textquotedblleft
canonical internal copy\textquotedblright\ of the quotient space $V/U$ inside
$\mathbf{A}$, in which each coset of $U$ in $V$ is replaced by the product of
all vectors in this coset. The magic of the Frobenius endomorphism $v\mapsto
v^{q}$ ensures that this replacement does not disturb the linear structure.

\subsection{Macdonald's conjectural recursion}

A \emph{line} in an $F$-vector space $W$ means a $1$-dimensional vector
subspace of $W$. The set of all lines in $W$ is known as the \emph{projective
space} $\mathbb{P}\left(  W\right)  $.

Now Macdonald's conjecture \cite[(7.25 ?)]{Macdon92} says the following:

\begin{theorem}
\label{thm.7.25}Assume that $\mathbf{A}$ is an integral domain. Let $V$ be a
finite-dimensional $F$-vector subspace of $\mathbf{A}$. Let $\lambda$ be a
partition of length $<\dim V$. Then,%
\[
S_{\lambda}\left(  V\right)  =\sum_{L\subseteq V\text{ line}}S_{\lambda
}\left(  V\sslash L\right)  .
\]
(The sum ranges over all lines $L$ in $V$.)
\end{theorem}

Macdonald proved this for $\lambda=\left(  1^{r}\right)  $ in \cite[\S I.2,
Example 26 (d)]{Macdon95}.

We shall prove Theorem \ref{thm.7.25} in the general case, along with a
generalization to skew partitions. Before we do any of this, however, we need
to introduce more notations.

\subsection{On the Frobenius morphism}

The skew version of Macdonald's 7th Variation requires a certain technical
condition on $\mathbf{A}$, defined in terms of the so-called Frobenius
morphism. Let us briefly recall it.

The \emph{Frobenius morphism} $\varphi$ of the commutative $F$-algebra
$\mathbf{A}$ is defined to be the map%
\begin{align*}
\varphi:\mathbf{A}  &  \rightarrow\mathbf{A},\\
a  &  \mapsto a^{q}.
\end{align*}
It is well-known that $\varphi$ is an $F$-algebra endomorphism (since
$F=\mathbb{F}_{q}$). Note that%
\begin{equation}
\varphi^{i}\left(  a\right)  =a^{q^{i}}\ \ \ \ \ \ \ \ \ \ \text{for each
}a\in\mathbf{A}\text{ and }i\in\mathbb{N}. \label{eq.phiia}%
\end{equation}

If $\mathbf{A}$ is an integral domain, then the Frobenius morphism $\varphi$
is injective (since $a^{q}=0$ entails $a=0$ in this case). However, $\varphi$
is often not surjective (since not every element of $\mathbf{A}$ is a $q$-th
power). The next remark shows some ways to find commutative $F$-algebras whose
$\varphi$ is bijective:

\begin{remark}
\label{rmk.phi-bij.1}\ \ 

\begin{enumerate}
\item[\textbf{(a)}] Let $\mathbb{N}\left[  1/q\right]  $ denote the set of all
nonnegative rational numbers of the form $i/q^{j}$ with $i,j\in\mathbb{N}$. We
call these numbers \emph{nonnegative }$q$\emph{-local integers}. Their set
$\mathbb{N}\left[  1/q\right]  $ is a monoid under addition (and even a
commutative semiring).

If $\mathbf{A}$ is a polynomial ring $F\left[  x_{1},x_{2},\ldots
,x_{n}\right]  $ over $F$, then the Frobenius morphism $\varphi:\mathbf{A}%
\rightarrow\mathbf{A}$ is not surjective (unless $n=0$), but $\mathbf{A}$ can
be embedded into a larger commutative $F$-algebra $\widehat{\mathbf{A}}$ whose
Frobenius morphism $\varphi:\widehat{\mathbf{A}}\rightarrow\widehat{\mathbf{A}%
}$ is invertible. This larger algebra $\widehat{\mathbf{A}}$ can be defined
informally as the ring of all \textquotedblleft polynomials\textquotedblright%
\ in $x_{1},x_{2},\ldots,x_{n}$ over $F$, where the exponents are not
restricted to nonnegative integers but can be any nonnegative $q$-local
integers. Formally speaking, this is the monoid algebra (over $F$) of the
additive monoid $\left(  \mathbb{N}\left[  1/q\right]  \right)  ^{n}$ (where
we rename each element $\left(  a_{1},a_{2},\ldots,a_{n}\right)  \in\left(
\mathbb{N}\left[  1/q\right]  \right)  ^{n}$ of this monoid as the monomial
$x_{1}^{a_{1}}x_{2}^{a_{2}}\cdots x_{n}^{a_{n}}$). The embedding of
$\mathbf{A}$ into $\widehat{\mathbf{A}}$ is the obvious one (sending each
$x_{i}$ to $x_{i}$). Note that $\widehat{\mathbf{A}}$ is an integral domain
(this can be proved in the same way as for usual polynomial rings over a field).

\item[\textbf{(b)}] More generally, if $\mathbf{A}$ is any commutative
$F$-algebra whose Frobenius morphism $\varphi:\mathbf{A}\rightarrow\mathbf{A}$
is injective, then we can construct a commutative $F$-algebra
$\widehat{\mathbf{A}}$ that contains $\mathbf{A}$ as a subalgebra and has an
invertible Frobenius morphism $\varphi:\widehat{\mathbf{A}}\rightarrow
\widehat{\mathbf{A}}$. Namely, we define $\widehat{\mathbf{A}}$ as a direct
limit of an infinite sequence
\[
\mathbf{A}\overset{\varphi}{\longrightarrow}\mathbf{A}\overset{\varphi
}{\longrightarrow}\mathbf{A}\overset{\varphi}{\longrightarrow}\cdots
\]
of copies of $\mathbf{A}$, where each copy embeds into the next via the
$F$-algebra morphism $\varphi$. We call this $\widehat{\mathbf{A}}$ the
\emph{perfect closure} of $\mathbf{A}$ (see, e.g., \cite[Proposition
1.4]{Leptie22}). If $\mathbf{A}$ is a polynomial ring $F\left[  x_{1}%
,x_{2},\ldots,x_{n}\right]  $, then this $\widehat{\mathbf{A}}$ is isomorphic
to the monoid algebra $\widehat{\mathbf{A}}$ from Remark \ref{rmk.phi-bij.1}
\textbf{(a)}. In general, if $\mathbf{A}$ is an integral domain, then
$\widehat{\mathbf{A}}$ is an integral domain as well.
\end{enumerate}
\end{remark}

\subsection{The skew 7th Variation}

We need some more notations before we can state our result.

For every $r\in\mathbb{N}$ and any finite-dimensional $F$-vector subspace $V$
of $\mathbf{A}$, we set%
\[
H_{r}\left(  V\right)  :=S_{\left(  r\right)  }\left(  V\right)
\]
and%
\[
E_{r}\left(  V\right)  :=S_{\left(  1^{r}\right)  }\left(  V\right)
,\ \ \ \ \ \ \ \ \ \ \text{where }\left(  1^{r}\right)  :=\left(
\underbrace{1,1,\ldots,1}_{r\text{ times}}\right)  .
\]
Note that if $r>\dim V$, then $E_{r}\left(  V\right)  =S_{\left(
1^{r}\right)  }\left(  V\right)  =0$ by (\ref{eq.Slam.SlamV0}), since $\left(
1^{r}\right)  $ is a partition of length $r$. These polynomials $H_{r}\left(
V\right)  $ and $E_{r}\left(  V\right)  $ are analogues of the complete
homogeneous and elementary symmetric polynomials, respectively. We also set
$H_{r}\left(  V\right)  :=0$ and $E_{r}\left(  V\right)  :=0$ for all negative
$r$.

Macdonald makes the following definition: If $\mathbf{A}$ is a commutative
$F$-algebra whose Frobenius morphism $\varphi:\mathbf{A}\rightarrow\mathbf{A}$
is invertible, and if $\lambda=\left(  \lambda_{1},\lambda_{2},\ldots
,\lambda_{k}\right)  $ and $\mu=\left(  \mu_{1},\mu_{2},\ldots,\mu_{k}\right)
$ are two partitions, and if $V$ is a finite-dimensional $F$-vector subspace
of $\mathbf{A}$, then the \emph{skew 7th Variation Schur polynomial} of $V$
corresponding to $\lambda/\mu$ is defined by%
\begin{equation}
S_{\lambda/\mu}\left(  V\right)  :=\det\left(  \left(  \varphi^{\mu_{j}%
-j+1}H_{\lambda_{i}-\mu_{j}-i+j}\left(  V\right)  \right)  _{i,j\in\left[
k\right]  }\right)  . \label{eq.Slammu}%
\end{equation}
Here, the power $\varphi^{\mu_{j}-j+1}$ of $\varphi$ is well-defined even if
the exponent is negative, since $\varphi$ is invertible. It is easy to see
(see \cite[errata, \textquotedblleft page 26, (7.11)\textquotedblright%
]{Macdon92}) that the right hand side of (\ref{eq.Slammu}) does not depend on
$k$ (as long as $\lambda$ and $\mu$ have length $\leq k$ each). It is
furthermore easy to see that%
\begin{equation}
S_{\lambda/\mu}\left(  V\right)  =0\ \ \ \ \ \ \ \ \ \ \text{unless }%
\mu\subseteq\lambda\label{eq.Slammu=0ifnotsub}%
\end{equation}
(a particular case of \cite[(7.12)]{Macdon92}); this means that the
\textquotedblleft interesting\textquotedblright\ skew 7th Variation Schur
polynomials are indexed by skew partitions. If the Frobenius morphism
$\varphi:\mathbf{A}\rightarrow\mathbf{A}$ is not invertible, then we can still
define $S_{\lambda/\mu}\left(  V\right)  $ in the perfect closure
$\widehat{\mathbf{A}}$ of $\mathbf{A}$, provided that $\varphi$ is injective.

As usual, we let $\varnothing$ denote the empty partition $\left(  {}\right)
$. If $\lambda$ is any partition and $V$ is any finite-dimensional $F$-vector
subspace of $\mathbf{A}$, then%
\begin{equation}
S_{\lambda/\varnothing}\left(  V\right)  =S_{\lambda}\left(  V\right)  .
\label{eq.Slam0}%
\end{equation}
(This follows by comparing (\ref{eq.Slammu}) with \cite[(7.10)]{Macdon92} if
the length of $\lambda$ is $\leq\dim V$. Otherwise, both sides of this
equality are $0$, since \cite[(7.12)]{Macdon92} shows that $S_{\lambda
/\varnothing}\left(  V\right)  =0$ (because $\lambda_{1}^{\prime}%
-\varnothing_{1}^{\prime}=\lambda_{1}^{\prime}=\ell\left(  \lambda\right)
>\dim V$) whereas (\ref{eq.Slam.SlamV0}) yields $S_{\lambda}\left(  V\right)
=0$.)

We are now able to state the skew generalization of Theorem \ref{thm.7.25}
that constitutes the main result of this paper:

\begin{theorem}
\label{thm.skew-V/L}Assume that $\mathbf{A}$ is an integral domain such that
the Frobenius morphism $\varphi:\mathbf{A}\rightarrow\mathbf{A}$ is
invertible. Let $V$ be a finite-dimensional $F$-vector subspace of
$\mathbf{A}$. Let $\lambda$ and $\mu$ be two partitions of length $<\dim V$
each. Then,%
\[
S_{\lambda/\mu}\left(  V\right)  =\sum_{L\subseteq V\text{ line}}%
S_{\lambda/\mu}\left(  V\sslash L\right)  .
\]

\end{theorem}

Applying Theorem \ref{thm.skew-V/L} to $\mu=\varnothing$ (the empty
partition), and simplifying using (\ref{eq.Slam0}), we obtain Theorem
\ref{thm.7.25} in the case when $\mathbf{A}$ is an integral domain whose
Frobenius morphism $\varphi$ is invertible. From this case, we can easily
deduce the general case (see Section \ref{sec.apx-orig-7.25} in the Appendix).
Our main quest is thus to prove Theorem \ref{thm.skew-V/L}.

\section{Lemmas and the proof}

\subsection{A few elementary lemmas}

In preparation for the proof, we first establish some lemmas about finite
fields and permutations. The first one helps us simplify (or complicate,
depending on one's viewpoint) sums over lines in $V$:

\begin{lemma}
\label{lem.linesum}Let $V$ be a finite-dimensional $F$-vector space. Let $A$
be an $F$-vector space. Let $b_{L}$ be an element of $A$ for each line
$L\subseteq V$. Then,%
\[
\sum_{L\subseteq V\text{ line}}b_{L}=-\sum_{w\in V\setminus\left\{  0\right\}
}b_{\operatorname*{span}\left(  w\right)  }.
\]

\end{lemma}

\begin{proof}
We have $q=0$ in $F$ (since $F=\mathbb{F}_{q}$), and thus $q-1=-1$ in $F$.
Moreover, for each line $L\subseteq V$, we have $\left\vert L\right\vert =q$
(since a line is a $1$-dimensional $F$-vector space) and $0\in L$, and
therefore%
\begin{equation}
\left\vert L\setminus\left\{  0\right\}  \right\vert =\left\vert L\right\vert
-1=q-1 \label{pf.lem.linesum.line}%
\end{equation}
(since $\left\vert L\right\vert =q$).

For each vector $w\in V\setminus\left\{  0\right\}  $, the span
$\operatorname*{span}\left(  w\right)  $ is a line in $V$. Thus,%
\begin{align*}
\sum_{w\in V\setminus\left\{  0\right\}  }b_{\operatorname*{span}\left(
w\right)  }  &  =\sum_{L\subseteq V\text{ line}}\ \ \sum_{\substack{w\in
V\setminus\left\{  0\right\}  ;\\\operatorname*{span}\left(  w\right)
=L}}\ \underbrace{b_{\operatorname*{span}\left(  w\right)  }}%
_{\substack{=b_{L}\\\text{(since }\operatorname*{span}\left(  w\right)
=L\text{)}}}=\sum_{L\subseteq V\text{ line}}\ \ \underbrace{\sum
_{\substack{w\in V\setminus\left\{  0\right\}  ;\\\operatorname*{span}\left(
w\right)  =L}}b_{L}}_{\substack{=\sum_{w\in L\setminus\left\{  0\right\}
}b_{L}\\=\left\vert L\setminus\left\{  0\right\}  \right\vert b_{L}}}\\
&  =\sum_{L\subseteq V\text{ line}}\underbrace{\left\vert L\setminus\left\{
0\right\}  \right\vert }_{\substack{=q-1\\\text{(by (\ref{pf.lem.linesum.line}%
))}}}b_{L}=\underbrace{\left(  q-1\right)  }_{=-1\text{ in }F}\ \ \sum
_{L\subseteq V\text{ line}}b_{L}=-\sum_{L\subseteq V\text{ line}}b_{L}.
\end{align*}
This yields the lemma.
\end{proof}

Our next lemma is so trivial it is barely worth the name. Following Macdonald,
we use the notation $\pi\left(  U\right)  $ for the product $\prod_{u\in
U\setminus\left\{  0\right\}  }u$ of all nonzero vectors in a
finite-dimensional $F$-vector subspace $U$ of $\mathbf{A}$. The lemma gives an
explicit formula for this product when $U$ is a line:

\begin{lemma}
\label{lem.pispanv}Let $v\in\mathbf{A}$ be nonzero. Then,%
\begin{equation}
\pi\left(  \operatorname*{span}\left(  v\right)  \right)  =-v^{q-1}.
\label{eq.pispanV}%
\end{equation}

\end{lemma}

\begin{proof}
By its definition, $\pi\left(  \operatorname*{span}\left(  v\right)  \right)
$ is the product of all nonzero vectors in $\operatorname*{span}\left(
v\right)  $.

However, the elements of $\operatorname*{span}\left(  v\right)  $ are
precisely the vectors $\alpha v$ for $\alpha\in F$, and furthermore these
vectors are all distinct (since $v$ is nonzero). Hence, the nonzero vectors in
$\operatorname*{span}\left(  v\right)  $ are precisely the vectors $\alpha v$
for $\alpha\in F\setminus\left\{  0\right\}  $, and furthermore these vectors
are all distinct. Therefore, their product is%
\[
\prod_{\alpha\in F\setminus\left\{  0\right\}  }\left(  \alpha v\right)
=\left(  \prod_{\alpha\in F\setminus\left\{  0\right\}  }\alpha\right)
v^{\left\vert F\setminus\left\{  0\right\}  \right\vert }.
\]

However, Wilson's theorem for finite fields says that $\prod_{\alpha\in
F\setminus\left\{  0\right\}  }\alpha=-1$\ \ \ \ \footnote{We recall the
proof: In the product $\prod_{\alpha\in F\setminus\left\{  0\right\}  }\alpha
$, each factor $\alpha$ can be paired with its inverse $\alpha^{-1}$ unless it
equals this inverse (i.e., unless $\alpha=\alpha^{-1}$). Thus, all the factors
of this product cancel out except for those that satisfy $\alpha=\alpha^{-1}$.
But the latter factors are precisely $1$ and $-1$ (since $\alpha=\alpha^{-1}$
is equivalent to $\alpha^{2}=1$, that is, $\alpha^{2}-1=0$, that is, $\left(
\alpha-1\right)  \left(  \alpha+1\right)  =0$, that is, $\alpha\in\left\{
1,-1\right\}  $), and multiply to $-1$ (this holds even if these two factors
are equal, which happens when $\operatorname*{char}F=2$). Hence, the entire
product $\prod_{\alpha\in F\setminus\left\{  0\right\}  }\alpha$ simplifies to
$-1$.}. Moreover, $\left\vert F\setminus\left\{  0\right\}  \right\vert =q-1$
(since $\left\vert F\right\vert =q$), so that $v^{\left\vert F\setminus
\left\{  0\right\}  \right\vert }=v^{q-1}$. Hence,%
\[
\prod_{\alpha\in F\setminus\left\{  0\right\}  }\left(  \alpha v\right)
=\underbrace{\left(  \prod_{\alpha\in F\setminus\left\{  0\right\}  }%
\alpha\right)  }_{=-1}\underbrace{v^{\left\vert F\setminus\left\{  0\right\}
\right\vert }}_{=v^{q-1}}=-v^{q-1}.
\]
Since we previously saw that the product of the nonzero vectors in
$\operatorname*{span}\left(  v\right)  $ is $\prod_{\alpha\in F\setminus
\left\{  0\right\}  }\left(  \alpha v\right)  $, we thus conclude that this
product is $-v^{q-1}$. In other words, $\pi\left(  \operatorname*{span}\left(
v\right)  \right)  =-v^{q-1}$. This proves Lemma \ref{lem.pispanv}.
\end{proof}

The next three lemmas are staples in combinatorial number theory:

\begin{lemma}
\label{lem.zerosum-F}Let $i\in\mathbb{N}$ be such that $i<q-1$. Then,
$\sum_{\alpha\in F}\alpha^{i}=0$.
\end{lemma}

\begin{proof}
Well-known, but let us give a proof nevertheless. If $i=0$, then $\sum
_{\alpha\in F}\underbrace{\alpha^{i}}_{=\alpha^{0}=1}=\sum_{\alpha\in
F}1=\left\vert F\right\vert =q=0$ in $F$. Thus, we WLOG assume that $i>0$.
Hence, the polynomial $x^{i+1}-x\in F\left[  x\right]  $ is a nonzero
polynomial of degree $i+1$. Therefore, this polynomial has at most $i+1$ roots
in $F$. Since $i+1<q$ (because $i<q-1$), this shows that this polynomial has
fewer roots in $F$ than $F$ has elements. Thus, there exists some $u\in F$
that is not a root of this polynomial. Consider this $u$. Then, $u^{i+1}%
-u\neq0$, so that $u\left(  u^{i}-1\right)  =u^{i+1}-u\neq0$. Thus, $u\neq0$
and $u^{i}-1\neq0$.

Now, $u\neq0$ shows that $u$ is invertible in the field $F$, and thus the map
$F\rightarrow F,\ \alpha\mapsto u\alpha$ is a bijection. Hence, substituting
$u\alpha$ for $\alpha$ in the sum $\sum_{\alpha\in F}\alpha^{i}$, we obtain
$\sum_{\alpha\in F}\alpha^{i}=\sum_{\alpha\in F}\underbrace{\left(
u\alpha\right)  ^{i}}_{=u^{i}\alpha^{i}}=u^{i}\sum_{\alpha\in F}\alpha^{i}$.
Therefore,%
\[
\left(  u^{i}-1\right)  \sum_{\alpha\in F}\alpha^{i}=u^{i}\sum_{\alpha\in
F}\alpha^{i}-\sum_{\alpha\in F}\alpha^{i}=0
\]
(since $\sum_{\alpha\in F}\alpha^{i}=u^{i}\sum_{\alpha\in F}\alpha^{i}$). We
can divide this by $u^{i}-1$ (since $u^{i}-1\neq0$), and obtain $\sum
_{\alpha\in F}\alpha^{i}=0$. This proves Lemma \ref{lem.zerosum-F}.
\end{proof}

\begin{lemma}
\label{lem.cw-lema}Let $P\in\mathbf{A}\left[  t_{1},t_{2},\ldots,t_{n}\right]
$ be a polynomial of total degree $<n\left(  q-1\right)  $ in $n$ variables
$t_{1},t_{2},\ldots,t_{n}$ over $\mathbf{A}$. Then,%
\[
\sum_{\alpha_{1},\alpha_{2},\ldots,\alpha_{n}\in F}P\left(  \alpha_{1}%
,\alpha_{2},\ldots,\alpha_{n}\right)  =0.
\]

\end{lemma}

\begin{proof}
The claim depends $\mathbf{A}$-linearly on $P$. Thus, we can WLOG assume that
$P$ is just a monomial $t_{1}^{i_{1}}t_{2}^{i_{2}}\cdots t_{n}^{i_{n}}$ with
$i_{1}+i_{2}+\cdots+i_{n}<n\left(  q-1\right)  $ (since any polynomial of
total degree $<n\left(  q-1\right)  $ in $t_{1},t_{2},\ldots,t_{n}$ is an
$\mathbf{A}$-linear combination of such monomials). Assume this, and consider
these exponents $i_{1},i_{2},\ldots,i_{n}$. Then, at least one $k\in\left[
n\right]  $ satisfies $i_{k}<q-1$ (since $i_{1}+i_{2}+\cdots+i_{n}<n\left(
q-1\right)  $) and therefore $\sum_{\alpha_{k}\in F}\alpha_{k}^{i_{k}}%
=\sum_{\alpha\in F}\alpha^{i_{k}}=0$ (by Lemma \ref{lem.zerosum-F}, applied to
$i=i_{k}$). But $P=t_{1}^{i_{1}}t_{2}^{i_{2}}\cdots t_{n}^{i_{n}}$, and
therefore%
\[
\sum_{\alpha_{1},\alpha_{2},\ldots,\alpha_{n}\in F}P\left(  \alpha_{1}%
,\alpha_{2},\ldots,\alpha_{n}\right)  =\sum_{\alpha_{1},\alpha_{2}%
,\ldots,\alpha_{n}\in F}\alpha_{1}^{i_{1}}\alpha_{2}^{i_{2}}\cdots\alpha
_{n}^{i_{n}}=\prod_{j=1}^{n}\ \ \underbrace{\sum_{\alpha_{j}\in F}\alpha
_{j}^{i_{j}}}_{\substack{=0\text{ for }j=k\\\text{(since }\sum_{\alpha_{k}\in
F}\alpha_{k}^{i_{k}}=0\text{)}}}=0
\]
(since a product is $0$ if at least one of its factors is $0$). This proves
the lemma.
\end{proof}

\begin{lemma}
\label{lem.zerosum}Let $a_{1},a_{2},\ldots,a_{k}$ be $k$ nonnegative integers,
where $k>0$. Let $V$ be an $n$-dimensional $F$-vector subspace of $\mathbf{A}%
$, where $n>k$. Then,%
\[
\sum_{w\in V\setminus\left\{  0\right\}  }w^{\left(  q-1\right)  \left(
q^{a_{1}}+q^{a_{2}}+\cdots+q^{a_{k}}\right)  }=0.
\]

\end{lemma}

\begin{proof}
Since $k>0$, we have $\left(  q-1\right)  \left(  q^{a_{1}}+q^{a_{2}}%
+\cdots+q^{a_{k}}\right)  >0$ and thus%
\[
0^{\left(  q-1\right)  \left(  q^{a_{1}}+q^{a_{2}}+\cdots+q^{a_{k}}\right)
}=0.
\]
Hence, we can replace the $\sum_{w\in V\setminus\left\{  0\right\}  }$ sign in
Lemma \ref{lem.zerosum} by a $\sum_{w\in V}$ sign without changing the sum. It
thus remains to prove
\begin{equation}
\sum_{w\in V}w^{\left(  q-1\right)  \left(  q^{a_{1}}+q^{a_{2}}+\cdots
+q^{a_{k}}\right)  }=0. \label{pf.lem.zerosum.1}%
\end{equation}

For each $a\in\mathbb{N}$, the map
\begin{align*}
\mathbf{A}  &  \rightarrow\mathbf{A},\\
v  &  \mapsto v^{q^{a}}%
\end{align*}
is an $F$-algebra endomorphism of $\mathbf{A}$ (indeed, it is the $a$-th power
of the Frobenius morphism $\varphi$ of $\mathbf{A}$). Thus, for each
$a\in\mathbb{N}$ and any $\alpha_{1},\alpha_{2},\ldots,\alpha_{n}\in F$ and
$x_{1},x_{2},\ldots,x_{n}\in\mathbf{A}$, we have%
\begin{align}
&  \left(  \alpha_{1}x_{1}+\alpha_{2}x_{2}+\cdots+\alpha_{n}x_{n}\right)
^{q^{a}}\nonumber\\
&  =\alpha_{1}x_{1}^{q^{a}}+\alpha_{2}x_{2}^{q^{a}}+\cdots+\alpha_{n}%
x_{n}^{q^{a}}. \label{pf.lem.zerosum.frobeq}%
\end{align}

Now, fix a basis $\left(  x_{1},x_{2},\ldots,x_{n}\right)  $ of the $F$-vector
space $V$. Then, each $w\in V$ can be uniquely written as $\alpha_{1}%
x_{1}+\alpha_{2}x_{2}+\cdots+\alpha_{n}x_{n}$ with $\alpha_{1},\alpha
_{2},\ldots,\alpha_{n}\in F$. Hence,%
\begin{align*}
&  \sum_{w\in V}w^{\left(  q-1\right)  \left(  q^{a_{1}}+q^{a_{2}}%
+\cdots+q^{a_{k}}\right)  }\\
&  =\sum_{\alpha_{1},\alpha_{2},\ldots,\alpha_{n}\in F}\underbrace{\left(
\alpha_{1}x_{1}+\alpha_{2}x_{2}+\cdots+\alpha_{n}x_{n}\right)  ^{\left(
q-1\right)  \left(  q^{a_{1}}+q^{a_{2}}+\cdots+q^{a_{k}}\right)  }}%
_{=\prod_{j=1}^{k}\left(  \alpha_{1}x_{1}+\alpha_{2}x_{2}+\cdots+\alpha
_{n}x_{n}\right)  ^{\left(  q-1\right)  q^{a_{j}}}}\\
&  =\sum_{\alpha_{1},\alpha_{2},\ldots,\alpha_{n}\in F}\ \ \prod_{j=1}%
^{k}\underbrace{\left(  \alpha_{1}x_{1}+\alpha_{2}x_{2}+\cdots+\alpha_{n}%
x_{n}\right)  ^{\left(  q-1\right)  q^{a_{j}}}}_{=\left(  \left(  \alpha
_{1}x_{1}+\alpha_{2}x_{2}+\cdots+\alpha_{n}x_{n}\right)  ^{q^{a_{j}}}\right)
^{q-1}}\\
&  =\sum_{\alpha_{1},\alpha_{2},\ldots,\alpha_{n}\in F}\ \ \prod_{j=1}%
^{k}\left(  \underbrace{\left(  \alpha_{1}x_{1}+\alpha_{2}x_{2}+\cdots
+\alpha_{n}x_{n}\right)  ^{q^{a_{j}}}}_{\substack{=\alpha_{1}x_{1}^{q^{a_{j}}%
}+\alpha_{2}x_{2}^{q^{a_{j}}}+\cdots+\alpha_{n}x_{n}^{q^{a_{j}}}\\\text{(by
(\ref{pf.lem.zerosum.frobeq}), applied to }a=a_{j}\text{)}}}\right)  ^{q-1}\\
&  =\sum_{\alpha_{1},\alpha_{2},\ldots,\alpha_{n}\in F}\ \ \prod_{j=1}%
^{k}\left(  \alpha_{1}x_{1}^{q^{a_{j}}}+\alpha_{2}x_{2}^{q^{a_{j}}}%
+\cdots+\alpha_{n}x_{n}^{q^{a_{j}}}\right)  ^{q-1}=0
\end{align*}
(by Lemma \ref{lem.cw-lema}, applied to the polynomial $P=\prod_{j=1}%
^{k}\left(  t_{1}x_{1}^{q^{a_{j}}}+t_{2}x_{2}^{q^{a_{j}}}+\cdots+t_{n}%
x_{n}^{q^{a_{j}}}\right)  ^{q-1}\in\mathbf{A}\left[  t_{1},t_{2},\ldots
,t_{n}\right]  $, which has total degree $\underbrace{k}_{<n}\left(
q-1\right)  <n\left(  q-1\right)  $). This proves (\ref{pf.lem.zerosum.1}) and
thus the lemma.
\end{proof}

The following purely combinatorial lemma will help us simplify a determinant:

\begin{lemma}
\label{lem.perm}Let $n\in\mathbb{N}$. Let $\alpha_{1}>\alpha_{2}>\cdots
>\alpha_{n}$ be $n$ integers, and let $\beta_{1}>\beta_{2}>\cdots>\beta_{n}$
be $n$ further integers. Assume that%
\begin{equation}
\alpha_{i}-\beta_{i}\in\left\{  0,1\right\}  \ \ \ \ \ \ \ \ \ \ \text{for all
}i\in\left[  n\right]  . \label{eq.lem.perm.ass}%
\end{equation}
Let $\sigma\in S_{n}$ be a permutation such that $\sigma\neq\operatorname*{id}%
$. Then, there exists an $i\in\left[  n\right]  $ such that $\alpha_{i}%
-\beta_{\sigma\left(  i\right)  }\notin\left\{  0,1\right\}  $.
\end{lemma}

\begin{proof}
We have assumed that $\sigma\neq\operatorname*{id}$. Thus, there exists some
$k\in\left[  n\right]  $ such that $\sigma\left(  k\right)  \neq k$. Consider
the \textbf{smallest} such $k$. Then,%
\begin{equation}
\sigma\left(  j\right)  =j\ \ \ \ \ \ \ \ \ \ \text{for each }j<k
\label{pf.lem.perm.min}%
\end{equation}
(since our $k$ is the \textbf{smallest} $k$). Hence, if we had $\sigma\left(
k\right)  <k$, then we would have $\sigma\left(  \sigma\left(  k\right)
\right)  =\sigma\left(  k\right)  $ (by (\ref{pf.lem.perm.min}), applied to
$j=\sigma\left(  k\right)  $), therefore $\sigma\left(  k\right)  =k$ (because
$\sigma$ is injective); but this would contradict $\sigma\left(  k\right)
\neq k$. Hence, we cannot have $\sigma\left(  k\right)  <k$. Thus, we have
$\sigma\left(  k\right)  \geq k$ and therefore $\sigma\left(  k\right)  >k$
(since $\sigma\left(  k\right)  \neq k$). That is, $\sigma\left(  k\right)
\geq k+1$. Hence, $k+1\leq\sigma\left(  k\right)  \leq n$, so that
$k+1\in\left[  n\right]  $. Thus, from $\beta_{1}>\beta_{2}>\cdots>\beta_{n}$,
we obtain $\beta_{\sigma\left(  k\right)  }\leq\beta_{k+1}$ (since
$k+1\leq\sigma\left(  k\right)  $) and $\beta_{k+1}<\beta_{k}$. Also,
$\alpha_{k}>\alpha_{k+1}$ (which follows from $\alpha_{1}>\alpha_{2}%
>\cdots>\alpha_{n}$).

We want to find an $i\in\left[  n\right]  $ such that $\alpha_{i}%
-\beta_{\sigma\left(  i\right)  }\notin\left\{  0,1\right\}  $. If $\alpha
_{k}-\beta_{\sigma\left(  k\right)  }\notin\left\{  0,1\right\}  $, then we
can just take $i=k$ and be done. Thus, we WLOG assume that $\alpha_{k}%
-\beta_{\sigma\left(  k\right)  }\in\left\{  0,1\right\}  $. Thus,
$0\leq\alpha_{k}-\beta_{\sigma\left(  k\right)  }\leq1$. In other words,
\begin{equation}
\beta_{\sigma\left(  k\right)  }\leq\alpha_{k}\leq\beta_{\sigma\left(
k\right)  }+1. \label{pf.lem.perm.4}%
\end{equation}

But (\ref{eq.lem.perm.ass}) yields $\alpha_{k}-\beta_{k}\in\left\{
0,1\right\}  $. Hence, $0\leq\alpha_{k}-\beta_{k}\leq1$. In other words,
\[
\beta_{k}\leq\alpha_{k}\leq\beta_{k}+1.
\]
The same argument (applied to $k+1$ instead of $k$) yields%
\[
\beta_{k+1}\leq\alpha_{k+1}\leq\beta_{k+1}+1.
\]

If we had $\sigma\left(  k\right)  >k+1$, then we would have $\beta
_{\sigma\left(  k\right)  }<\beta_{k+1}$ (since $\beta_{1}>\beta_{2}%
>\cdots>\beta_{n}$) and therefore $\beta_{\sigma\left(  k\right)  }+1\leq
\beta_{k+1}$, which would entail $\alpha_{k}\leq\beta_{\sigma\left(  k\right)
}+1\leq\beta_{k+1}\leq\alpha_{k+1}$; but this would contradict $\alpha
_{k}>\alpha_{k+1}$. Hence, we cannot have $\sigma\left(  k\right)  >k+1$.
Thus, we have $\sigma\left(  k\right)  \leq k+1$. Combined with $\sigma\left(
k\right)  \geq k+1$, this leads to $\sigma\left(  k\right)  =k+1$. Therefore,
we can rewrite the chain of inequalities (\ref{pf.lem.perm.4}) as%
\[
\beta_{k+1}\leq\alpha_{k}\leq\beta_{k+1}+1.
\]
Thus, $\alpha_{k}\leq\beta_{k+1}+1\leq\beta_{k}$ (since $\beta_{k+1}<\beta
_{k}$). Combining this with $\beta_{k}\leq\alpha_{k}$, we obtain
\[
\alpha_{k}=\beta_{k}.
\]

Let $j:=\sigma^{-1}\left(  k\right)  $. Then, $j\in\left[  n\right]  $ and
$\sigma\left(  j\right)  =k$. If we had $j<k$, then we would have
$\sigma\left(  j\right)  =j$ (by (\ref{pf.lem.perm.min})) and therefore
$k=\sigma\left(  j\right)  =j<k$, which is a contradiction. Thus, we must have
$j\geq k$. Since $\sigma\left(  j\right)  =k\neq\sigma\left(  k\right)  $, we
furthermore have $j\neq k$, and thus $j>k$ (since $j\geq k$). Hence,
$\alpha_{j}<\alpha_{k}$ (since $\alpha_{1}>\alpha_{2}>\cdots>\alpha_{n}$).
Thus, $\alpha_{j}<\alpha_{k}=\beta_{k}=\beta_{\sigma\left(  j\right)  }$
(since $k=\sigma\left(  j\right)  $). Therefore, $\alpha_{j}-\beta
_{\sigma\left(  j\right)  }<0$, so that $\alpha_{j}-\beta_{\sigma\left(
j\right)  }\notin\left\{  0,1\right\}  $. Thus, there exists an $i\in\left[
n\right]  $ such that $\alpha_{i}-\beta_{\sigma\left(  i\right)  }%
\notin\left\{  0,1\right\}  $ (namely, $i=j$). This completes the proof of
Lemma \ref{lem.perm}.
\end{proof}

\subsection{The $\protect\widetilde{S}_{\lambda/\mu}$}

We now come back to the properties of the 7th Variation. More precisely, we
define a variant of it (a 7.5th Variation perhaps) in the skew case.

\begin{convention}
For this whole section, we assume that the Frobenius morphism $\varphi
:\mathbf{A}\rightarrow\mathbf{A}$ is invertible.
\end{convention}

For any two partitions $\lambda$ and $\mu$, we define%
\begin{equation}
\widetilde{S}_{\lambda/\mu}\left(  U\right)  :=\det\left(  \varphi
^{\lambda_{i}-i}E_{\lambda_{i}-\mu_{j}-i+j}\left(  U\right)  \right)
_{i,j\in\left[  r\right]  }\in\mathbf{A}, \label{eq.Stil=}%
\end{equation}
where $r\in\mathbb{N}$ is chosen such that both $\ell\left(  \lambda\right)  $
and $\ell\left(  \mu\right)  $ are $\leq r$ (the exact choice of $r$ is
immaterial, because if $\lambda_{r}=\mu_{r}=0$, then the $r$-th row of the
matrix $\left(  \varphi^{\lambda_{i}-i}E_{\lambda_{i}-\mu_{j}-i+j}\left(
U\right)  \right)  _{i,j\in\left[  r\right]  }$ is $\left(  0,0,\ldots
,0,1\right)  $, and thus Laplace expansion along this row shows that
\[
\det\left(  \varphi^{\lambda_{i}-i}E_{\lambda_{i}-\mu_{j}-i+j}\left(
U\right)  \right)  _{i,j\in\left[  r\right]  }=\det\left(  \varphi
^{\lambda_{i}-i}E_{\lambda_{i}-\mu_{j}-i+j}\left(  U\right)  \right)
_{i,j\in\left[  r-1\right]  }%
\]
in this case). This is similar to the formula \cite[(7.11')]{Macdon92} for
$S_{\lambda^{\prime}/\mu^{\prime}}\left(  U\right)  $, but probably not
directly related. However, it has some properties in common. First, we have an
analogue of \cite[(7.12)]{Macdon92}:

\begin{lemma}
\label{lem.Stil.0}Let $\lambda$ and $\mu$ be two partitions. Let $U$ be an
$n$-dimensional $F$-vector subspace of $\mathbf{A}$. Then,
\begin{equation}
\widetilde{S}_{\lambda/\mu}\left(  U\right)  =0\text{ unless }0\leq\lambda
_{i}-\mu_{i}\leq n\text{ for all }i\geq1. \label{eq.Stilde.0}%
\end{equation}

\end{lemma}

\begin{proof}
This is entirely analogous to \cite[(7.12)]{Macdon92}. (See \cite[errata,
proof of (6.10)]{Macdon92} for a proof template that can be used almost
verbatim here -- just replace $\lambda^{\prime}$ and $\mu^{\prime}$ by
$\lambda$ and $\mu$, and replace $e_{\lambda_{i}^{\prime}-\mu_{j}^{\prime
}-i+j}\left(  x\mid\tau^{-\mu_{j}^{\prime}+j-1}a\right)  $ by $\varphi
^{\lambda_{i}-i}E_{\lambda_{i}-\mu_{j}-i+j}\left(  U\right)  $.)
\end{proof}

As a particular case of Lemma \ref{lem.Stil.0}, we see that
\begin{equation}
\widetilde{S}_{\lambda/\mu}\left(  U\right)  =0\text{ unless }\mu
\subseteq\lambda. \label{eq.Stilde.0ifnotsub}%
\end{equation}

Next, we have an analogue of \cite[(7.22)]{Macdon92}:\footnote{Unlike
Macdonald, we use $\lambda/\mu$ (not $\lambda-\mu$) to denote the skew
partition consisting of two partitions $\mu$ and $\lambda$.
\par
We also use the standard notation $\lambda_{i}$ for the $i$-th entry of a
partition $\lambda$. (If $i$ surpasses the length of $\lambda$, then this
entry is $0$ by definition.)
\par
Recall that a skew partition $\lambda/\mu$ is said to be a \emph{vertical
strip} if its Young diagram has no two distinct cells in the same row. This is
equivalent to requiring that each $i\geq1$ satisfies $\mu_{i}\leq\lambda
_{i}\leq\mu_{i}+1$.}

\begin{lemma}
\label{lem.Stil.1}Let $\lambda$ and $\mu$ be two partitions. Let $U$ be a line
(i.e., a $1$-dimensional vector subspace) in $\mathbf{A}$. Recall the notation
$\pi\left(  U\right)  $ for the product of all nonzero vectors in $U$. Then:

\begin{enumerate}
\item[\textbf{(a)}] If $\lambda/\mu$ is a vertical strip, then we have%
\begin{equation}
\widetilde{S}_{\lambda/\mu}\left(  U\right)  =\prod_{\substack{i\geq
1;\\\lambda_{i}>\mu_{i}}}\varphi^{\lambda_{i}-i}\left(  -\pi\left(  U\right)
\right)  . \label{eq.Stilde.1-dim}%
\end{equation}

\item[\textbf{(b)}] Otherwise, we have%
\begin{equation}
\widetilde{S}_{\lambda/\mu}\left(  U\right)  =0. \label{eq.Stilde.1-dim.0}%
\end{equation}

\end{enumerate}
\end{lemma}

\begin{proof}
\textbf{(b)} Assume that $\lambda/\mu$ is not a vertical strip. Then, some
$i\geq1$ satisfies $\lambda_{i}<\mu_{i}$ or $\lambda_{i}>\mu_{i}+1$. In other
words, some $i\geq1$ does \textbf{not} satisfy $0\leq\lambda_{i}-\mu_{i}\leq
1$. Hence, Lemma \ref{lem.Stil.0} (applied to $n=1$) shows that $\widetilde{S}%
_{\lambda/\mu}\left(  U\right)  =0$. This proves Lemma \ref{lem.Stil.1}
\textbf{(b)}. \medskip

\textbf{(a)} Let us first compute all the $E_{r}\left(  U\right)  $ for
$r\in\mathbb{Z}$:

\begin{itemize}
\item We have $E_{0}\left(  U\right)  =1$ (since each finite-dimensional
vector subspace $V$ of $\mathbf{A}$ satisfies $E_{0}\left(  V\right)
=S_{\left(  1^{0}\right)  }\left(  V\right)  =S_{\varnothing}\left(  V\right)
=1$).

\item From \cite[(7.20)]{Macdon92}, we know that $\pi\left(  U\right)
=\left(  -1\right)  ^{1}E_{1}\left(  U\right)  $ (since $\dim U=1$), and thus%
\[
E_{1}\left(  U\right)  =\left(  -1\right)  ^{1}\pi\left(  U\right)
=-\pi\left(  U\right)  .
\]

\item Recall that if $V$ is a finite-dimensional vector subspace of
$\mathbf{A}$, then $E_{r}\left(  V\right)  =0$ for all $r>\dim V$ and also for
all $r<0$. In other words, if $V$ is a finite-dimensional vector subspace of
$\mathbf{A}$, then%
\[
E_{r}\left(  V\right)  =0\ \ \ \ \ \ \ \ \ \ \text{for all }r\notin\left\{
0,1,\ldots,\dim V\right\}  .
\]
Applying this to $V=U$ (which has dimension $\dim U=1$), we conclude that
\begin{equation}
E_{r}\left(  U\right)  =0\ \ \ \ \ \ \ \ \ \ \text{for all }r\notin\left\{
0,1\right\}  . \label{eq.ErU=0}%
\end{equation}

\end{itemize}

Now assume that $\lambda/\mu$ is a vertical strip. Thus,%
\begin{equation}
\lambda_{i}-\mu_{i}\in\left\{  0,1\right\}  \ \ \ \ \ \ \ \ \ \ \text{for each
}i\geq1. \label{pf.eq.Stilde.1-dim.vs}%
\end{equation}

Let $n\in\mathbb{N}$ be such that both $\ell\left(  \lambda\right)  $ and
$\ell\left(  \mu\right)  $ are $\leq n$. Then, the definition of
$\widetilde{S}_{\lambda/\mu}\left(  U\right)  $ yields%
\begin{align}
\widetilde{S}_{\lambda/\mu}\left(  U\right)   &  =\det\left(  \varphi
^{\lambda_{i}-i}E_{\lambda_{i}-\mu_{j}-i+j}\left(  U\right)  \right)
_{i,j\in\left[  n\right]  }\nonumber\\
&  =\sum_{\sigma\in S_{n}}\left(  -1\right)  ^{\sigma}\prod_{i=1}^{n}%
\varphi^{\lambda_{i}-i}E_{\lambda_{i}-\mu_{\sigma\left(  i\right)  }%
-i+\sigma\left(  i\right)  }\left(  U\right)  \label{pf.eq.Stilde.1-dim.3}%
\end{align}
(by the definition of a determinant).

We shall now show that the only nonzero addend in the sum on the right hand
side is the addend for $\sigma=\operatorname*{id}$. Indeed, let $\sigma\in
S_{n}$ be a permutation such that $\sigma\neq\operatorname*{id}$. We have
$\lambda_{1}-1>\lambda_{2}-2>\cdots>\lambda_{n}-n$ (since $\lambda_{1}%
\geq\lambda_{2}\geq\cdots\geq\lambda_{n}$) and $\mu_{1}-1>\mu_{2}-2>\cdots
>\mu_{n}-n$ (similarly) and $\left(  \lambda_{i}-i\right)  -\left(  \mu
_{i}-i\right)  =\lambda_{i}-\mu_{i}\in\left\{  0,1\right\}  $ for each
$i\in\left[  n\right]  $ (by (\ref{pf.eq.Stilde.1-dim.vs})). Hence, Lemma
\ref{lem.perm} (applied to $\alpha_{i}=\lambda_{i}-i$ and $\beta_{i}=\mu
_{i}-i$) shows that there exists an $i\in\left[  n\right]  $ such that
$\left(  \lambda_{i}-i\right)  -\left(  \mu_{\sigma\left(  i\right)  }%
-\sigma\left(  i\right)  \right)  \notin\left\{  0,1\right\}  $. Consider this
$i$. Then,
\[
\lambda_{i}-\mu_{\sigma\left(  i\right)  }-i+\sigma\left(  i\right)  =\left(
\lambda_{i}-i\right)  -\left(  \mu_{\sigma\left(  i\right)  }-\sigma\left(
i\right)  \right)  \notin\left\{  0,1\right\}  .
\]
Thus, (\ref{eq.ErU=0}) yields $E_{\lambda_{i}-\mu_{\sigma\left(  i\right)
}-i+\sigma\left(  i\right)  }\left(  U\right)  =0$, so that $\varphi
^{\lambda_{i}-i}E_{\lambda_{i}-\mu_{\sigma\left(  i\right)  }-i+\sigma\left(
i\right)  }\left(  U\right)  =\varphi^{\lambda_{i}-i}0=0$ as well.

Thus, we have found an $i\in\left[  n\right]  $ such that $\varphi
^{\lambda_{i}-i}E_{\lambda_{i}-\mu_{\sigma\left(  i\right)  }-i+\sigma\left(
i\right)  }\left(  U\right)  =0$. In other words, the product $\prod_{i=1}%
^{n}\varphi^{\lambda_{i}-i}E_{\lambda_{i}-\mu_{\sigma\left(  i\right)
}-i+\sigma\left(  i\right)  }\left(  U\right)  $ has at least one factor equal
to $0$. Thus, this whole product is $0$.

Forget that we fixed $\sigma$. We thus have shown that if $\sigma\in S_{n}$ is
a permutation such that $\sigma\neq\operatorname*{id}$, then the product
$\prod_{i=1}^{n}\varphi^{\lambda_{i}-i}E_{\lambda_{i}-\mu_{\sigma\left(
i\right)  }-i+\sigma\left(  i\right)  }\left(  U\right)  $ is $0$.
Consequently, in the sum on the right hand side of (\ref{pf.eq.Stilde.1-dim.3}%
), all addends with $\sigma\neq\operatorname*{id}$ are $0$. Thus,
(\ref{pf.eq.Stilde.1-dim.3}) simplifies to%
\begin{align}
\widetilde{S}_{\lambda/\mu}\left(  U\right)   &  =\underbrace{\left(
-1\right)  ^{\operatorname*{id}}}_{=1}\prod_{i=1}^{n}\varphi^{\lambda_{i}%
-i}\underbrace{E_{\lambda_{i}-\mu_{\operatorname*{id}\left(  i\right)
}-i+\operatorname*{id}\left(  i\right)  }\left(  U\right)  }%
_{\substack{=E_{\lambda_{i}-\mu_{i}-i+i}\left(  U\right)  \\=E_{\lambda
_{i}-\mu_{i}}\left(  U\right)  }}=\prod_{i=1}^{n}\varphi^{\lambda_{i}%
-i}E_{\lambda_{i}-\mu_{i}}\left(  U\right) \nonumber\\
&  =\prod_{\substack{i\in\left[  n\right]  ;\\\lambda_{i}\neq\mu_{i}}%
}\varphi^{\lambda_{i}-i}E_{\lambda_{i}-\mu_{i}}\left(  U\right)
\label{pf.eq.Stilde.1-dim.7}%
\end{align}
(here, we have removed all the factors with $\lambda_{i}=\mu_{i}$ from the
product, since they all equal $\varphi^{\mu_{i}-i}\underbrace{E_{\mu_{i}%
-\mu_{i}}\left(  U\right)  }_{=E_{0}\left(  U\right)  =1}=\varphi^{\mu_{i}%
-i}1=1$).

But (\ref{pf.eq.Stilde.1-dim.vs}) shows that each $i\in\left[  n\right]  $
satisfies $\lambda_{i}-\mu_{i}\in\left\{  0,1\right\}  $, so that $\lambda
_{i}\geq\mu_{i}$. Hence, the condition $\lambda_{i}\neq\mu_{i}$ under the
product sign in (\ref{pf.eq.Stilde.1-dim.7}) is equivalent to $\lambda_{i}%
>\mu_{i}$. That is, we can rewrite (\ref{pf.eq.Stilde.1-dim.7}) as%
\begin{equation}
\widetilde{S}_{\lambda/\mu}\left(  U\right)  =\prod_{\substack{i\in\left[
n\right]  ;\\\lambda_{i}>\mu_{i}}}\varphi^{\lambda_{i}-i}E_{\lambda_{i}%
-\mu_{i}}\left(  U\right)  . \label{pf.eq.Stilde.1-dim.8}%
\end{equation}

Moreover, if $i\in\left[  n\right]  $ satisfies $\lambda_{i}>\mu_{i}$, then it
satisfies $\lambda_{i}-\mu_{i}=1$ (since (\ref{pf.eq.Stilde.1-dim.vs}) yields
$\lambda_{i}-\mu_{i}\in\left\{  0,1\right\}  $, but $\lambda_{i}-\mu_{i}$
cannot be $0$ because of $\lambda_{i}>\mu_{i}$) and thus $E_{\lambda_{i}%
-\mu_{i}}\left(  U\right)  =E_{1}\left(  U\right)  =-\pi\left(  U\right)  $.
Hence, we can rewrite (\ref{pf.eq.Stilde.1-dim.8}) as%
\[
\widetilde{S}_{\lambda/\mu}\left(  U\right)  =\prod_{\substack{i\in\left[
n\right]  ;\\\lambda_{i}>\mu_{i}}}\varphi^{\lambda_{i}-i}\left(  -\pi\left(
U\right)  \right)  .
\]

Comparing this with%
\begin{align*}
&  \prod_{\substack{i\geq1;\\\lambda_{i}>\mu_{i}}}\varphi^{\lambda_{i}%
-i}\left(  -\pi\left(  U\right)  \right) \\
&  =\left(  \prod_{\substack{i\in\left[  n\right]  ;\\\lambda_{i}>\mu_{i}%
}}\varphi^{\lambda_{i}-i}\left(  -\pi\left(  U\right)  \right)  \right)
\underbrace{\left(  \prod_{\substack{i>n;\\\lambda_{i}>\mu_{i}}}\varphi
^{\lambda_{i}-i}\left(  -\pi\left(  U\right)  \right)  \right)  }%
_{\substack{=1\\\text{(indeed, this is an empty product,}\\\text{since each
}i>n\text{ satisfies }\lambda_{i}=0\\\text{(because }\ell\left(
\lambda\right)  \leq n\text{) and }\mu_{i}=0\text{ (since }\ell\left(
\mu\right)  \leq n\text{)}\\\text{and thus cannot satisfy }\lambda_{i}>\mu
_{i}\text{)}}}\\
&  =\prod_{\substack{i\in\left[  n\right]  ;\\\lambda_{i}>\mu_{i}}%
}\varphi^{\lambda_{i}-i}\left(  -\pi\left(  U\right)  \right)  ,
\end{align*}
we obtain%
\[
\widetilde{S}_{\lambda/\mu}\left(  U\right)  =\prod_{\substack{i\geq
1;\\\lambda_{i}>\mu_{i}}}\varphi^{\lambda_{i}-i}\left(  -\pi\left(  U\right)
\right)  .
\]
This proves Lemma \ref{lem.Stil.1} \textbf{(a)}.
\end{proof}

Most importantly, we have a formula akin to \cite[(7.18)]{Macdon92}:

\begin{lemma}
\label{lem.Stilde.coproduct}Assume that $\mathbf{A}$ is an integral domain
such that the Frobenius morphism $\varphi:\mathbf{A}\rightarrow\mathbf{A}$ is
invertible. Let $\lambda$ and $\mu$ be two partitions. Let $U\subseteq V$ be
two finite-dimensional $F$-vector subspaces of $\mathbf{A}$. Then,%
\begin{equation}
S_{\lambda/\mu}\left(  V\sslash U\right)  =\sum_{\nu}\left(  -1\right)
^{\left\vert \lambda\right\vert -\left\vert \nu\right\vert }S_{\nu/\mu}\left(
V\right)  \cdot\varphi^{\dim\left(  V/U\right)  }\widetilde{S}_{\lambda/\nu
}\left(  U\right)  ,\ \ \ \ \ \ \ \ \ \ \label{eq.Stilde.coproduct}%
\end{equation}
where the sum ranges over all partitions $\nu$.
\end{lemma}

\begin{proof}
For any finite-dimensional $F$-vector subspace $W$ of $\mathbf{A}$, we define
the two $\mathbb{Z}\times\mathbb{Z}$-matrices%
\begin{align*}
\mathbf{H}\left(  W\right)   &  :=\left(  \varphi^{i+1}H_{j-i}\left(
W\right)  \right)  _{i,j\in\mathbb{Z}}\ \ \ \ \ \ \ \ \ \ \text{and}\\
\mathbf{E}\left(  W\right)   &  :=\left(  \left(  -1\right)  ^{j-i}\varphi
^{j}E_{j-i}\left(  W\right)  \right)  _{i,j\in\mathbb{Z}}%
\end{align*}
over $\mathbf{A}$. From \cite[(7.9)]{Macdon92}, we know that these two
matrices are upper-triangular and mutually inverse. In particular, $\left(
\mathbf{H}\left(  U\right)  \right)  ^{-1}=\mathbf{E}\left(  U\right)  $.

Moreover, from \cite[(7.17) (ii)]{Macdon92}, we have%
\[
\mathbf{H}\left(  V\right)  =\mathbf{H}\left(  V\sslash U\right)  \cdot
\varphi^{\dim\left(  V/U\right)  }\left(  \mathbf{H}\left(  U\right)  \right)
,
\]
where an algebra morphism such as $\varphi^{\dim\left(  V/U\right)  }$ is
applied to a matrix by applying it to each entry of the matrix. Thus,%
\begin{align}
\mathbf{H}\left(  V\sslash U\right)   &  =\mathbf{H}\left(  V\right)
\cdot\left(  \varphi^{\dim\left(  V/U\right)  }\left(  \mathbf{H}\left(
U\right)  \right)  \right)  ^{-1}\nonumber\\
&  =\mathbf{H}\left(  V\right)  \cdot\varphi^{\dim\left(  V/U\right)  }\left(
\underbrace{\left(  \mathbf{H}\left(  U\right)  \right)  ^{-1}}_{=\mathbf{E}%
\left(  U\right)  }\right) \nonumber\\
&  \ \ \ \ \ \ \ \ \ \ \ \ \ \ \ \ \ \ \ \ \left(  \text{since }\varphi\text{
is an algebra morphism}\right) \nonumber\\
&  =\mathbf{H}\left(  V\right)  \cdot\varphi^{\dim\left(  V/U\right)  }\left(
\mathbf{E}\left(  U\right)  \right)  . \label{pf.lem.Stilde.coproduct.HHE}%
\end{align}

From this equality, we proceed using the Cauchy--Binet theorem, analogously to
\cite[proof of (7.18)]{Macdon92} (see \cite[errata, \textquotedblleft page 19,
proof of (6.13)\textquotedblright]{Macdon92} for the details of that proof),
to prove (\ref{eq.Stilde.coproduct}). See the Appendix (Section
\ref{sec.apx-coprod}) for the details of this argument.
\end{proof}

\subsection{Proof of the recursion}

One last formula -- a sort of Pieri rule -- will be crucial to our proof of
Theorem \ref{thm.skew-V/L}:

\begin{lemma}
\label{lem.SlamVL-ie}Assume that $\mathbf{A}$ is an integral domain such that
the Frobenius morphism $\varphi:\mathbf{A}\rightarrow\mathbf{A}$ is invertible.

Let $n\in\mathbb{N}$. Let $V$ be an $n$-dimensional $F$-vector subspace of
$\mathbf{A}$. Let $\lambda$ and $\mu$ be two partitions of length $<n$. Let
$\ell\in V$ be nonzero, and let $L=\operatorname*{span}\left(  \ell\right)
\subseteq V$. For any vertical strip $\lambda/\nu$, we set%
\[
\mathbf{q}\left(  \lambda,\nu\right)  :=\left(  q-1\right)  \sum
_{\substack{i\geq1;\\\lambda_{i}>\nu_{i}}}q^{\lambda_{i}+n-1-i}.
\]
Then,%
\[
S_{\lambda/\mu}\left(  V\sslash L\right)  =\sum_{\lambda/\nu\text{ is a
vertical strip}}\left(  -1\right)  ^{\left\vert \lambda\right\vert -\left\vert
\nu\right\vert }\ell^{\mathbf{q}\left(  \lambda,\nu\right)  }S_{\nu/\mu
}\left(  V\right)  .
\]
(The sum is understood to range over all partitions $\nu$ such that
$\lambda/\nu$ is a vertical strip.)
\end{lemma}

\begin{proof}
We have $\dim L=1$ (since $L$ is the span of the nonzero vector $\ell$) and
$\dim V=n$, so that
\[
\dim\left(  V/L\right)  =\underbrace{\dim V}_{=n}-\underbrace{\dim L}%
_{=1}=n-1.
\]
Moreover, from $L=\operatorname*{span}\left(  \ell\right)  $, we obtain (using
the notation $\pi\left(  L\right)  $ defined previously)%
\begin{equation}
\pi\left(  L\right)  =\pi\left(  \operatorname*{span}\left(  \ell\right)
\right)  =-\ell^{q-1} \label{pf.lem.SlamVL-ie.piL}%
\end{equation}
(by Lemma \ref{lem.pispanv}, applied to $v=\ell$, since $\ell$ is nonzero).

Applying (\ref{eq.Stilde.coproduct}) to $U=L$, we obtain%
\begin{align}
S_{\lambda/\mu}\left(  V\sslash L\right)   &  =\sum_{\nu}\left(  -1\right)
^{\left\vert \lambda\right\vert -\left\vert \nu\right\vert }S_{\nu/\mu}\left(
V\right)  \cdot\underbrace{\varphi^{\dim\left(  V/L\right)  }}%
_{\substack{=\varphi^{n-1}\\\text{(since }\dim\left(  V/L\right)
=n-1\text{)}}}\widetilde{S}_{\lambda/\nu}\left(  L\right) \nonumber\\
&  =\sum_{\nu}\left(  -1\right)  ^{\left\vert \lambda\right\vert -\left\vert
\nu\right\vert }S_{\nu/\mu}\left(  V\right)  \cdot\varphi^{n-1}\widetilde{S}%
_{\lambda/\nu}\left(  L\right)  . \label{pf.lem.SlamVL-ie.1}%
\end{align}
If $\nu$ is a partition for which $\lambda/\nu$ is \textbf{not} a vertical
strip, then (\ref{eq.Stilde.1-dim.0}) (applied to $\nu$ and $L$ instead of
$\mu$ and $U$) yields $\widetilde{S}_{\lambda/\nu}\left(  L\right)  =0$. Thus,
in the sum on the right hand side of (\ref{pf.lem.SlamVL-ie.1}), all addends
corresponding to such partitions $\nu$ vanish. Consequently, we can simplify
the sum by removing all these addends, and obtain%
\begin{equation}
S_{\lambda/\mu}\left(  V\sslash L\right)  =\sum_{\substack{\lambda/\nu\text{
is a}\\\text{vertical strip}}}\left(  -1\right)  ^{\left\vert \lambda
\right\vert -\left\vert \nu\right\vert }S_{\nu/\mu}\left(  V\right)
\cdot\varphi^{n-1}\widetilde{S}_{\lambda/\nu}\left(  L\right)  .
\label{pf.lem.SlamVL-ie.2}%
\end{equation}

Now, let $\nu$ be a partition such that $\lambda/\nu$ is a vertical strip.
Then, (\ref{eq.Stilde.1-dim}) (applied to $L$ and $\nu$ instead of $U$ and
$\mu$) yields
\begin{align*}
\widetilde{S}_{\lambda/\nu}\left(  L\right)   &  =\prod_{\substack{i\geq
1;\\\lambda_{i}>\nu_{i}}}\varphi^{\lambda_{i}-i}\left(  -\underbrace{\pi
\left(  L\right)  }_{\substack{=-\ell^{q-1}\\\text{(by
(\ref{pf.lem.SlamVL-ie.piL}))}}}\right) \\
&  =\prod_{\substack{i\geq1;\\\lambda_{i}>\nu_{i}}}\varphi^{\lambda_{i}%
-i}\left(  -\left(  -\ell^{q-1}\right)  \right)  =\prod_{\substack{i\geq
1;\\\lambda_{i}>\nu_{i}}}\varphi^{\lambda_{i}-i}\left(  \ell^{q-1}\right)  .
\end{align*}
Applying the algebra morphism $\varphi^{n-1}$ to both sides of this equality,
we find
\begin{align}
\varphi^{n-1}\widetilde{S}_{\lambda/\nu}\left(  L\right)   &  =\prod
_{\substack{i\geq1;\\\lambda_{i}>\nu_{i}}}\underbrace{\varphi^{n-1}%
\varphi^{\lambda_{i}-i}\left(  \ell^{q-1}\right)  }_{\substack{=\varphi
^{n-1+\lambda_{i}-i}\left(  \ell^{q-1}\right)  \\=\left(  \ell^{q-1}\right)
^{q^{n-1+\lambda_{i}-i}}\\\text{(by (\ref{eq.phiia}))}}}=\prod
_{\substack{i\geq1;\\\lambda_{i}>\nu_{i}}}\left(  \ell^{q-1}\right)
^{q^{n-1+\lambda_{i}-i}}=\left(  \ell^{q-1}\right)  ^{\sum_{\substack{i\geq
1;\\\lambda_{i}>\nu_{i}}}q^{n-1+\lambda_{i}-i}}\nonumber\\
&  =\ell^{\left(  q-1\right)  \sum_{\substack{i\geq1;\\\lambda_{i}>\nu_{i}%
}}q^{n-1+\lambda_{i}-i}}=\ell^{\mathbf{q}\left(  \lambda,\nu\right)  }
\label{pf.lem.SlamVL-ie.5}%
\end{align}
(since $\left(  q-1\right)  \sum_{\substack{i\geq1;\\\lambda_{i}>\nu_{i}%
}}\underbrace{q^{n-1+\lambda_{i}-i}}_{=q^{\lambda_{i}+n-1-i}}=\left(
q-1\right)  \sum_{\substack{i\geq1;\\\lambda_{i}>\nu_{i}}}q^{\lambda
_{i}+n-1-i}=\mathbf{q}\left(  \lambda,\nu\right)  $).

Forget that we fixed $\nu$. We thus have shown that (\ref{pf.lem.SlamVL-ie.5})
holds for each partition $\nu$ such that $\lambda/\nu$ is a vertical strip.
Thus, (\ref{pf.lem.SlamVL-ie.2}) becomes%
\begin{align*}
S_{\lambda/\mu}\left(  V\sslash L\right)   &  =\sum_{\substack{\lambda
/\nu\text{ is a}\\\text{vertical strip}}}\left(  -1\right)  ^{\left\vert
\lambda\right\vert -\left\vert \nu\right\vert }S_{\nu/\mu}\left(  V\right)
\cdot\underbrace{\varphi^{n-1}\widetilde{S}_{\lambda/\nu}\left(  L\right)
}_{\substack{=\ell^{\mathbf{q}\left(  \lambda,\nu\right)  }\\\text{(by
(\ref{pf.lem.SlamVL-ie.5}))}}}\\
&  =\sum_{\substack{\lambda/\nu\text{ is a}\\\text{vertical strip}}}\left(
-1\right)  ^{\left\vert \lambda\right\vert -\left\vert \nu\right\vert }%
\ell^{\mathbf{q}\left(  \lambda,\nu\right)  }S_{\nu/\mu}\left(  V\right)  .
\end{align*}
This proves Lemma \ref{lem.SlamVL-ie}.
\end{proof}

\begin{proof}
[Proof of Theorem \ref{thm.skew-V/L}.]Let $n=\dim V$. Thus, $\left\vert
V\right\vert =\left\vert F\right\vert ^{n}=q^{n}$ (since $\left\vert
F\right\vert =q$), so that $\left\vert V\setminus\left\{  0\right\}
\right\vert =q^{n}-1$. The partitions $\lambda$ and $\mu$ have length $<\dim
V=n$; thus, $n>0$, so that $q^{n}\equiv0\operatorname{mod}q$ and therefore
$q^{n}=0$ in $F$. Hence, $q^{n}-1=-1$ in $F$.

For any vertical strip $\lambda/\nu$, define $\mathbf{q}\left(  \lambda
,\nu\right)  $ as in Lemma \ref{lem.SlamVL-ie}. It is clear that
$\mathbf{q}\left(  \lambda,\lambda\right)  =0$, since the sum is empty for
$\nu=\lambda$.

By Lemma \ref{lem.linesum} (applied to $A=\mathbf{A}$ and $b_{L}%
=S_{\lambda/\mu}\left(  V\sslash L\right)  $), we have
\begin{align}
&  \sum_{L\subseteq V\text{ line}}S_{\lambda/\mu}\left(  V\sslash L\right)
\nonumber\\
&  =-\sum_{w\in V\setminus\left\{  0\right\}  }\underbrace{S_{\lambda/\mu
}\left(  V\sslash \operatorname*{span}\left(  w\right)  \right)
}_{\substack{=\sum_{\lambda/\nu\text{ is a vertical strip}}\left(  -1\right)
^{\left\vert \lambda\right\vert -\left\vert \nu\right\vert }w^{\mathbf{q}%
\left(  \lambda,\nu\right)  }S_{\nu/\mu}\left(  V\right)  \\\text{(by Lemma
\ref{lem.SlamVL-ie}, applied to }\ell=w\text{ and }L=\operatorname*{span}%
\left(  w\right)  \text{)}}}\nonumber\\
&  =-\sum_{w\in V\setminus\left\{  0\right\}  }\ \ \sum_{\lambda/\nu\text{ is
a vertical strip}}\left(  -1\right)  ^{\left\vert \lambda\right\vert
-\left\vert \nu\right\vert }w^{\mathbf{q}\left(  \lambda,\nu\right)  }%
S_{\nu/\mu}\left(  V\right) \nonumber\\
&  =-\sum_{\lambda/\nu\text{ is a vertical strip}}\left(  -1\right)
^{\left\vert \lambda\right\vert -\left\vert \nu\right\vert }\left(  \sum_{w\in
V\setminus\left\{  0\right\}  }w^{\mathbf{q}\left(  \lambda,\nu\right)
}\right)  S_{\nu/\mu}\left(  V\right)  \label{pf.thm.skew-V/L.2}%
\end{align}
(here, we have interchanged the summation signs).

Now, let $\nu\neq\lambda$ be a partition such that $\lambda/\nu$ is a vertical
strip. Then, $\nu\subseteq\lambda$, so that the partition $\nu$ has length
$<n$ (since $\lambda$ has length $<n$). Moreover, from $\nu\neq\lambda$, we
conclude that at least one positive integer $i\geq1$ satisfies $\nu_{i}%
\neq\lambda_{i}$ and therefore $\nu_{i}<\lambda_{i}$ (since $\nu
\subseteq\lambda$ entails $\nu_{i}\leq\lambda_{i}$), that is, $\lambda_{i}%
>\nu_{i}$.

Let $i_{1},i_{2},\ldots,i_{k}$ be all the positive integers $i\geq1$
satisfying $\lambda_{i}>\nu_{i}$ (listed without repetitions, so there are $k$
of them). Then, each of these positive integers $i_{1},i_{2},\ldots,i_{k}$
must be $\leq n-1$ (since both $\lambda$ and $\nu$ have length $<n$, so that
each $i\geq n$ must satisfy $\lambda_{i}=0=\nu_{i}$, and therefore each
positive integer $i\geq1$ satisfying $\lambda_{i}>\nu_{i}$ must be $\leq
n-1$). Hence, there are at most $n-1$ of them. In other words, $k\leq n-1<n$.
That is, $n>k$. But we must also have $k>0$ (since we have shown that at least
one positive integer $i\geq1$ satisfies $\lambda_{i}>\nu_{i}$).

For each $j\in\left[  k\right]  $, set $a_{j}:=\lambda_{i_{j}}+n-1-i_{j}$.
This is a nonnegative integer, because we have $i_{j}\leq n-1$ (since each of
the positive integers $i_{1},i_{2},\ldots,i_{k}$ is $\leq n-1$) and thus
$a_{j}=\lambda_{i_{j}}+n-1-\underbrace{i_{j}}_{\leq n-1}\geq\lambda_{i_{j}%
}+n-1-\left(  n-1\right)  =\lambda_{i_{j}}\geq0$. Hence, Lemma
\ref{lem.zerosum} shows that%
\begin{equation}
\sum_{w\in V\setminus\left\{  0\right\}  }w^{\left(  q-1\right)  \left(
q^{a_{1}}+q^{a_{2}}+\cdots+q^{a_{k}}\right)  }=0. \label{pf.thm.skew-V/L.3}%
\end{equation}
Since
\begin{align*}
&  q^{a_{1}}+q^{a_{2}}+\cdots+q^{a_{k}}\\
&  =\sum_{j=1}^{k}q^{a_{j}}=\sum_{j=1}^{k}q^{\lambda_{i_{j}}+n-1-i_{j}%
}\ \ \ \ \ \ \ \ \ \ \left(  \text{since }a_{j}=\lambda_{i_{j}}+n-1-i_{j}%
\right) \\
&  =\sum_{\substack{i\geq1;\\\lambda_{i}>\nu_{i}}}q^{\lambda_{i}%
+n-1-i}\ \ \ \ \ \ \ \ \ \ \left(
\begin{array}
[c]{c}%
\text{since }i_{1},i_{2},\ldots,i_{k}\text{ are precisely the}\\
\text{positive integers }i\geq1\text{ satisfying }\lambda_{i}>\nu_{i}%
\end{array}
\right)  ,
\end{align*}
we have%
\[
\left(  q-1\right)  \left(  q^{a_{1}}+q^{a_{2}}+\cdots+q^{a_{k}}\right)
=\left(  q-1\right)  \sum_{\substack{i\geq1;\\\lambda_{i}>\nu_{i}}%
}q^{\lambda_{i}+n-1-i}=\mathbf{q}\left(  \lambda,\nu\right)
\]
(by the definition of $\mathbf{q}\left(  \lambda,\nu\right)  $). Thus, we can
rewrite (\ref{pf.thm.skew-V/L.3}) as%
\begin{equation}
\sum_{w\in V\setminus\left\{  0\right\}  }w^{\mathbf{q}\left(  \lambda
,\nu\right)  }=0. \label{pf.thm.7.25.7}%
\end{equation}

Forget that we fixed $\nu$. We thus have proved (\ref{pf.thm.7.25.7}) for each
partition $\nu\neq\lambda$ such that $\lambda/\nu$ is a vertical strip. Thus,
on the right hand side of (\ref{pf.thm.skew-V/L.2}), all addends of the outer
sum are $0$ except for the addend for $\nu=\lambda$. Hence,
(\ref{pf.thm.skew-V/L.2}) simplifies to
\begin{align*}
\sum_{L\subseteq V\text{ line}}S_{\lambda/\mu}\left(  V\sslash L\right)   &
=-\underbrace{\left(  -1\right)  ^{\left\vert \lambda\right\vert -\left\vert
\lambda\right\vert }}_{=\left(  -1\right)  ^{0}=1}\left(  \sum_{w\in
V\setminus\left\{  0\right\}  }\underbrace{w^{\mathbf{q}\left(  \lambda
,\lambda\right)  }}_{\substack{=1\\\text{(since }\mathbf{q}\left(
\lambda,\lambda\right)  =0\text{)}}}\right)  S_{\lambda/\mu}\left(  V\right)
\\
&  =-\underbrace{\left(  \sum_{w\in V\setminus\left\{  0\right\}  }1\right)
}_{\substack{=\left\vert V\setminus\left\{  0\right\}  \right\vert
\\=q^{n}-1\\=-1\text{ in }F}}S_{\lambda/\mu}\left(  V\right)  =-\left(
-1\right)  S_{\lambda/\mu}\left(  V\right)  =S_{\lambda/\mu}\left(  V\right)
.
\end{align*}
This proves Theorem \ref{thm.skew-V/L}.
\end{proof}

\section{Macdonald's sum-over-flags formula}

\subsection{The sum-over-flags formula}

Macdonald viewed Theorem \ref{thm.7.25} as a stepping stone towards an
explicit sum-over-flags formula for the 7th Variation Schur functions
(\cite[(7.24 ?)]{Macdon92}):

\begin{theorem}
\label{thm.7.24}Assume that $\mathbf{A}$ is an integral domain. Let
$n\in\mathbb{N}$. Let $V$ be an $n$-dimensional $F$-vector subspace of
$\mathbf{A}$. Let $\lambda=\left(  \lambda_{1},\lambda_{2},\ldots,\lambda
_{n}\right)  $ be a partition of length $\leq n$. Then,%
\[
S_{\lambda}\left(  V\right)  =\sum_{\substack{V=V_{0}>V_{1}>\cdots
>V_{n}=0\\\text{is a complete flag in }V}}\ \ \prod_{i=1}^{n}H_{\lambda_{i}%
}\left(  V_{i-1}\sslash V_{i}\right)  .
\]

\end{theorem}

Since Macdonald does not explain how this can be derived from Theorem
\ref{thm.7.25}, we shall do this here. We do still need some preparations.

\subsection{Lemmas}

If $V$ is a finite-dimensional $F$-vector subspace of $\mathbf{A}$, then
$\pi\left(  V\right)  $ shall denote the product $\prod_{u\in V\setminus
\left\{  0\right\}  }u$ of all nonzero vectors in $V$. This product will play
an important role in the next few lemmas.

\begin{lemma}
\label{lem.pi1}Let $n$ be a positive integer. Assume that $\mathbf{A}$ is an
integral domain. Let $V$ be an $n$-dimensional $F$-vector subspace of
$\mathbf{A}$. Let $V=V_{0}>V_{1}>\cdots>V_{n}=0$ be a complete flag in $V$.
Then,%
\[
\pi\left(  V\right)  =\prod_{i=1}^{n}\pi\left(  V_{i-1}\sslash V_{i}\right)
.
\]

\end{lemma}

\begin{proof}
For each $i\in\left[  n\right]  $, we have $\pi\left(  V_{i-1}\sslash V_{i}%
\right)  =\prod_{u\in V_{i-1}-V_{i}}u$ (by \cite[(7.21)]{Macdon92}, applied to
$U=V_{i-1}$ and $U^{\prime}=V_{i}$). Multiplying these $n$ equalities, we find%
\[
\prod_{i=1}^{n}\pi\left(  V_{i-1}\sslash V_{i}\right)  =\underbrace{\prod
_{i=1}^{n}\ \ \prod_{u\in V_{i-1}-V_{i}}}_{\substack{=\prod_{u\in V_{0}-V_{n}%
}\\\text{(since }V_{0}-V_{n}\text{ is the union of the}\\\text{disjoint sets
}V_{i-1}-V_{i}\text{ for all }i\in\left[  n\right]  \text{)}}}u=\prod_{u\in
V_{0}-V_{n}}u=\prod_{u\in V-\left\{  0\right\}  }u
\]
(since $V_{0}=V$ and $V_{n}=0=\left\{  0\right\}  $). But we also have
$\pi\left(  V\right)  =\prod_{u\in V-\left\{  0\right\}  }u$ by the definition
of $\pi\left(  V\right)  $. Comparing these two equalities, we find
$\pi\left(  V\right)  =\prod_{i=1}^{n}\pi\left(  V_{i-1}\sslash V_{i}\right)
$. This proves Lemma \ref{lem.pi1}.
\end{proof}

\begin{lemma}
\label{lem.Hri1}Let $U$ be a $1$-dimensional $F$-vector subspace of
$\mathbf{A}$. Let $r$ be a positive integer. Then,%
\[
\pi\left(  U\right)  \cdot\varphi\left(  H_{r-1}\left(  U\right)  \right)
=-H_{r}\left(  U\right)  .
\]

\end{lemma}

\begin{proof}
The first equality in \cite[proof of (7.22)]{Macdon92} says that%
\begin{equation}
H_{r}\left(  U\right)  =\left(  -1\right)  ^{r}\prod_{j=1}^{r}\varphi^{j-1}%
\pi\left(  U\right)  . \label{pf.lem.Hri1.1}%
\end{equation}
The same argument (applied to $r-1$ instead of $r$) yields%
\[
H_{r-1}\left(  U\right)  =\left(  -1\right)  ^{r-1}\prod_{j=1}^{r-1}%
\varphi^{j-1}\pi\left(  U\right)  .
\]
Applying the map $\varphi$ to this equality, we find%
\begin{align*}
\varphi\left(  H_{r-1}\left(  U\right)  \right)   &  =\varphi\left(  \left(
-1\right)  ^{r-1}\prod_{j=1}^{r-1}\varphi^{j-1}\pi\left(  U\right)  \right) \\
&  =\left(  -1\right)  ^{r-1}\prod_{j=1}^{r-1}\underbrace{\varphi\left(
\varphi^{j-1}\pi\left(  U\right)  \right)  }_{=\varphi^{j}\pi\left(  U\right)
}\ \ \ \ \ \ \ \ \ \ \left(  \text{since }\varphi\text{ is a ring
morphism}\right) \\
&  =\left(  -1\right)  ^{r-1}\prod_{j=1}^{r-1}\varphi^{j}\pi\left(  U\right)
=\left(  -1\right)  ^{r-1}\prod_{j=2}^{r}\varphi^{j-1}\pi\left(  U\right)
\end{align*}
(here, we have substituted $j-1$ for $j$ in the product). Thus,%
\begin{align*}
\pi\left(  U\right)  \cdot\varphi\left(  H_{r-1}\left(  U\right)  \right)   &
=\underbrace{\pi\left(  U\right)  }_{\substack{=\varphi^{0}\pi\left(
U\right)  \\=\varphi^{1-1}\pi\left(  U\right)  }}\cdot\underbrace{\left(
-1\right)  ^{r-1}}_{=-\left(  -1\right)  ^{r}}\prod_{j=2}^{r}\varphi^{j-1}%
\pi\left(  U\right) \\
&  =-\left(  -1\right)  ^{r}\cdot\underbrace{\varphi^{1-1}\pi\left(  U\right)
\cdot\prod_{j=2}^{r}\varphi^{j-1}\pi\left(  U\right)  }_{=\prod_{j=1}%
^{r}\varphi^{j-1}\pi\left(  U\right)  }\\
&  =-\underbrace{\left(  -1\right)  ^{r}\prod_{j=1}^{r}\varphi^{j-1}\pi\left(
U\right)  }_{\substack{=H_{r}\left(  U\right)  \\\text{(by
(\ref{pf.lem.Hri1.1}))}}}=-H_{r}\left(  U\right)  .
\end{align*}
This proves Lemma \ref{lem.Hri1}.
\end{proof}

\begin{lemma}
\label{lem.fullhouse-1}Let $n$ be a positive integer. Let $V$ be an
$n$-dimensional $F$-vector subspace of $\mathbf{A}$.

Let $\lambda=\left(  \lambda_{1},\lambda_{2},\ldots,\lambda_{n}\right)  $ be a
partition with $\lambda_{n}\geq1$. Let $\lambda\ominus1$ be the partition
$\left(  \lambda_{1}-1,\lambda_{2}-1,\ldots,\lambda_{n}-1\right)  $. Then,%
\[
S_{\lambda}\left(  V\right)  =\left(  -1\right)  ^{n}\pi\left(  V\right)
\cdot\left(  S_{\lambda\ominus1}\left(  V\right)  \right)  ^{q}.
\]

\end{lemma}

\begin{proof}
More generally, for any $n$-tuple $\beta=\left(  \beta_{1},\beta_{2}%
,\ldots,\beta_{n}\right)  \in\mathbb{N}^{n}$, we set $\beta\ominus1:=\left(
\beta_{1}-1,\beta_{2}-1,\ldots,\beta_{n}-1\right)  $. This is an $n$-tuple of
integers, but belongs to $\mathbb{N}^{n}$ if all entries of $\beta$ are
$\geq1$. In particular, $\lambda\ominus1\in\mathbb{N}^{n}$, since each
$j\in\left[  n\right]  $ satisfies $\lambda_{j}\geq\lambda_{n}\geq1$.

Let $\mathbf{B}$ be the polynomial ring $F\left[  x_{1},x_{2},\ldots
,x_{n}\right]  $. Let $W$ be the $F$-vector subspace $\operatorname*{span}%
\left(  x_{1},x_{2},\ldots,x_{n}\right)  $ of $\mathbf{B}$. Then, the
definition of $S_{\lambda}\left(  x_{1},x_{2},\ldots,x_{n}\right)  $ yields
\begin{equation}
S_{\lambda}\left(  x_{1},x_{2},\ldots,x_{n}\right)  =A_{\lambda+\delta
}/A_{\delta}, \label{pf.lem.fullhouse-1.S1}%
\end{equation}
where the $\alpha$-alternants $A_{\alpha}$ for all $\alpha\in\mathbb{N}^{n}$
are defined in (\ref{eq.Aalpha=}). Similarly,%
\begin{equation}
S_{\lambda\ominus1}\left(  x_{1},x_{2},\ldots,x_{n}\right)  =A_{\left(
\lambda\ominus1\right)  +\delta}/A_{\delta}. \label{pf.lem.fullhouse-1.S2}%
\end{equation}

Now, define an $n$-tuple $\beta=\left(  \beta_{1},\beta_{2},\ldots,\beta
_{n}\right)  \in\mathbb{N}^{n}$ by $\beta:=\lambda+\delta$, so that $\beta
_{j}=\lambda_{j}+n-j$ for each $j\in\left[  n\right]  $. From $\lambda
+\delta=\beta$, we obtain%
\begin{equation}
A_{\lambda+\delta}=A_{\beta}=\det\left(  x_{i}^{q^{\beta_{j}}}\right)
_{i,j\in\left[  n\right]  } \label{pf.lem.fullhouse-1.1}%
\end{equation}
by (\ref{eq.Aalpha=}). Moreover, each $j\in\left[  n\right]  $ satisfies
$\beta_{j}=\underbrace{\lambda_{j}}_{\geq1}+\underbrace{n-j}_{\geq0}\geq1$.
Thus, each $i,j\in\left[  n\right]  $ satisfy%
\[
x_{i}^{q^{\beta_{j}}}=x_{i}^{q^{\beta_{j}-1}q}=\left(  x_{i}^{q^{\beta_{j}-1}%
}\right)  ^{q}=\varphi\left(  x_{i}^{q^{\beta_{j}-1}}\right)
\]
(by the definition of $\varphi$). Hence, we can rewrite
(\ref{pf.lem.fullhouse-1.1}) as
\begin{align}
A_{\lambda+\delta}  &  =\det\left(  \varphi\left(  x_{i}^{q^{\beta_{j}-1}%
}\right)  \right)  _{i,j\in\left[  n\right]  }\nonumber\\
&  =\varphi\left(  \det\left(  \left(  x_{i}^{q^{\beta_{j}-1}}\right)
_{i,j\in\left[  n\right]  }\right)  \right)  \label{pf.lem.fullhouse-1.2}%
\end{align}
(since $\varphi$ is a ring morphism and thus commutes with determinants).

However, $\left(  \lambda\ominus1\right)  +\delta=\underbrace{\left(
\lambda+\delta\right)  }_{=\beta}\ominus\,1=\beta\ominus1$. Hence,
\[
A_{\left(  \lambda\ominus1\right)  +\delta}=A_{\beta\ominus1}=\det\left(
\left(  x_{i}^{q^{\beta_{j}-1}}\right)  _{i,j\in\left[  n\right]  }\right)
\ \ \ \ \ \ \ \ \ \ \left(  \text{by (\ref{eq.Aalpha=})}\right)  .
\]
In view of this, we can rewrite (\ref{pf.lem.fullhouse-1.2}) as%
\[
A_{\lambda+\delta}=\varphi\left(  A_{\left(  \lambda\ominus1\right)  +\delta
}\right)  .
\]
Thus, (\ref{pf.lem.fullhouse-1.S1}) becomes%
\begin{align}
S_{\lambda}\left(  x_{1},x_{2},\ldots,x_{n}\right)   &
=\underbrace{A_{\lambda+\delta}}_{=\varphi\left(  A_{\left(  \lambda
\ominus1\right)  +\delta}\right)  }/A_{\delta}=\varphi\left(  A_{\left(
\lambda\ominus1\right)  +\delta}\right)  /A_{\delta}\nonumber\\
&  =\dfrac{\varphi\left(  A_{\left(  \lambda\ominus1\right)  +\delta}\right)
}{\varphi\left(  A_{\delta}\right)  }\cdot\dfrac{\varphi\left(  A_{\delta
}\right)  }{A_{\delta}}. \label{pf.lem.fullhouse-1.4}%
\end{align}

Applying the same argument to the partition $\left(  1^{n}\right)  =\left(
\underbrace{1,1,\ldots,1}_{n\text{ times}}\right)  $ instead of $\lambda$, we
obtain%
\begin{align*}
S_{\left(  1^{n}\right)  }\left(  x_{1},x_{2},\ldots,x_{n}\right)   &
=\dfrac{\varphi\left(  A_{\left(  \left(  1^{n}\right)  \ominus1\right)
+\delta}\right)  }{\varphi\left(  A_{\delta}\right)  }\cdot\dfrac
{\varphi\left(  A_{\delta}\right)  }{A_{\delta}}\\
&  =\underbrace{\dfrac{\varphi\left(  A_{\delta}\right)  }{\varphi\left(
A_{\delta}\right)  }}_{=1}\cdot\,\dfrac{\varphi\left(  A_{\delta}\right)
}{A_{\delta}}\ \ \ \ \ \ \ \ \ \ \left(  \text{since }\underbrace{\left(
\left(  1^{n}\right)  \ominus1\right)  }_{=0}+\,\delta=\delta\right) \\
&  =\dfrac{\varphi\left(  A_{\delta}\right)  }{A_{\delta}}.
\end{align*}
Thus,%
\begin{align*}
\dfrac{\varphi\left(  A_{\delta}\right)  }{A_{\delta}}  &  =S_{\left(
1^{n}\right)  }\left(  x_{1},x_{2},\ldots,x_{n}\right)  =S_{\left(
1^{n}\right)  }\left(  W\right) \\
&  \ \ \ \ \ \ \ \ \ \ \ \ \ \ \ \ \ \ \ \ \left(
\begin{array}
[c]{c}%
\text{by the definition of }S_{\left(  1^{n}\right)  }\left(  W\right)
\text{,}\\
\text{since }\left(  x_{1},x_{2},\ldots,x_{n}\right)  \text{ is a basis of }W
\end{array}
\right) \\
&  =E_{n}\left(  W\right)  \ \ \ \ \ \ \ \ \ \ \left(  \text{by the definition
of }E_{n}\left(  W\right)  \right) \\
&  =\left(  -1\right)  ^{n}\pi\left(  W\right)
\end{align*}
(since \cite[(7.20)]{Macdon92} says that $\pi\left(  W\right)  =\left(
-1\right)  ^{n}E_{n}\left(  W\right)  $). Substituting this into
(\ref{pf.lem.fullhouse-1.4}), we obtain%
\[
S_{\lambda}\left(  x_{1},x_{2},\ldots,x_{n}\right)  =\dfrac{\varphi\left(
A_{\left(  \lambda\ominus1\right)  +\delta}\right)  }{\varphi\left(
A_{\delta}\right)  }\cdot\left(  -1\right)  ^{n}\pi\left(  W\right)  .
\]
In view of%
\begin{align*}
\dfrac{\varphi\left(  A_{\left(  \lambda\ominus1\right)  +\delta}\right)
}{\varphi\left(  A_{\delta}\right)  }  &  =\dfrac{A_{\left(  \lambda
\ominus1\right)  +\delta}^{q}}{A_{\delta}^{q}}\ \ \ \ \ \ \ \ \ \ \left(
\text{by the definition of }\varphi\right) \\
&  =\left(  \underbrace{A_{\left(  \lambda\ominus1\right)  +\delta}/A_{\delta
}}_{\substack{=S_{\lambda\ominus1}\left(  x_{1},x_{2},\ldots,x_{n}\right)
\\\text{(by (\ref{pf.lem.fullhouse-1.S2}))}}}\right)  ^{q}\\
&  =\left(  S_{\lambda\ominus1}\left(  x_{1},x_{2},\ldots,x_{n}\right)
\right)  ^{q},
\end{align*}
we can rewrite this further as%
\begin{align}
S_{\lambda}\left(  x_{1},x_{2},\ldots,x_{n}\right)   &  =\left(
S_{\lambda\ominus1}\left(  x_{1},x_{2},\ldots,x_{n}\right)  \right)  ^{q}%
\cdot\left(  -1\right)  ^{n}\pi\left(  W\right) \nonumber\\
&  =\left(  -1\right)  ^{n}\pi\left(  W\right)  \cdot\left(  S_{\lambda
\ominus1}\left(  x_{1},x_{2},\ldots,x_{n}\right)  \right)  ^{q}.
\label{pf.lem.fullhouse-1.Sgen}%
\end{align}

Now, pick a basis $\left(  v_{1},v_{2},\ldots,v_{n}\right)  $ of the
$n$-dimensional $F$-vector space $V$. Let us substitute $v_{1},v_{2}%
,\ldots,v_{n}$ for $x_{1},x_{2},\ldots,x_{n}$ on both sides of
(\ref{pf.lem.fullhouse-1.Sgen}). This substitution transforms $S_{\lambda
}\left(  x_{1},x_{2},\ldots,x_{n}\right)  $ into $S_{\lambda}\left(  V\right)
$ (since $S_{\lambda}\left(  V\right)  =S_{\lambda}\left(  v_{1},v_{2}%
,\ldots,v_{n}\right)  $ was defined by substituting $v_{1},v_{2},\ldots,v_{n}$
for $x_{1},x_{2},\ldots,x_{n}$ in the polynomial $A_{\lambda+\delta}%
/A_{\delta}=S_{\lambda}\left(  x_{1},x_{2},\ldots,x_{n}\right)  $), and
transforms $S_{\lambda\ominus1}\left(  x_{1},x_{2},\ldots,x_{n}\right)  $ into
$S_{\lambda\ominus1}\left(  V\right)  $ (for analogous reasons); furthermore,
it transforms $\pi\left(  W\right)  $ into $\pi\left(  V\right)  $ (since this
substitution turns the $F$-linear combinations of the $x_{1},x_{2}%
,\ldots,x_{n}$ into the corresponding $F$-linear combinations of $v_{1}%
,v_{2},\ldots,v_{n}$, and is injective\footnote{In more detail: Let
$\psi:\mathbf{B}\rightarrow\mathbf{A}$ be the map that substitutes
$v_{1},v_{2},\ldots,v_{n}$ for $x_{1},x_{2},\ldots,x_{n}$ in any given
polynomial $f\in\mathbf{B}=F\left[  x_{1},x_{2},\ldots,x_{n}\right]  $. Then,
$\psi$ is an $F$-algebra morphism, and sends $x_{1},x_{2},\ldots,x_{n}$ to
$v_{1},v_{2},\ldots,v_{n}$. Hence, $\psi$ restricts to an isomorphism from the
$F$-vector space $W$ to $V$ (since $\psi$ sends the basis $\left(  x_{1}%
,x_{2},\ldots,x_{n}\right)  $ of $W$ to the basis $\left(  v_{1},v_{2}%
,\ldots,v_{n}\right)  $ of $V$). Therefore, $\psi$ also restricts to a
bijection from $W\setminus\left\{  0\right\}  $ to $V\setminus\left\{
0\right\}  $. In other words, the map%
\begin{align}
W\setminus\left\{  0\right\}   &  \rightarrow V\setminus\left\{  0\right\}
,\nonumber\\
u  &  \mapsto\psi\left(  u\right)  \label{pf.lem.fullhouse-1.fnc.1}%
\end{align}
is a bijection.
\par
But the definition of $\pi\left(  V\right)  $ yields $\pi\left(  V\right)
=\prod_{u\in V\setminus\left\{  0\right\}  }u$. Similarly, $\pi\left(
W\right)  =\prod_{u\in W\setminus\left\{  0\right\}  }u$. Applying the map
$\psi$ to the latter equality, we find%
\begin{align*}
\psi\left(  \pi\left(  W\right)  \right)   &  =\psi\left(  \prod_{u\in
W\setminus\left\{  0\right\}  }u\right)  =\prod_{u\in W\setminus\left\{
0\right\}  }\psi\left(  u\right)  \ \ \ \ \ \ \ \ \ \ \left(  \text{since
}\psi\text{ is an }F\text{-algebra morphism}\right) \\
&  =\prod_{u\in V\setminus\left\{  0\right\}  }u
\end{align*}
(here, we have substituted $u$ for $\psi\left(  u\right)  $ in the product,
since the map (\ref{pf.lem.fullhouse-1.fnc.1}) is a bijection). Comparing this
with $\pi\left(  V\right)  =\prod_{u\in V\setminus\left\{  0\right\}  }u$, we
obtain $\psi\left(  \pi\left(  W\right)  \right)  =\pi\left(  V\right)  $. In
other words, the substitution of $v_{1},v_{2},\ldots,v_{n}$ for $x_{1}%
,x_{2},\ldots,x_{n}$ transforms $\pi\left(  W\right)  $ into $\pi\left(
V\right)  $ (since this substitution is precisely $\psi$).}). Hence, the
result of applying this substitution to (\ref{pf.lem.fullhouse-1.Sgen}) is%
\[
S_{\lambda}\left(  V\right)  =\left(  -1\right)  ^{n}\pi\left(  V\right)
\cdot\left(  S_{\lambda\ominus1}\left(  V\right)  \right)  ^{q}.
\]
This proves Lemma \ref{lem.fullhouse-1}.
\end{proof}

\subsection{Proof of the formula}

\begin{proof}
[Proof of Theorem \ref{thm.7.24}.]We use strong induction on $\left\vert
\lambda\right\vert +n$. The \textit{base case} ($\left\vert \lambda\right\vert
+n=0$) is entirely trivial (because in this case, $n=0$, so that $V=0$ and
$\lambda=\varnothing$). For the \textit{induction step}, we fix a positive
integer $N$. We assume (as the induction hypothesis) that Theorem
\ref{thm.7.24} holds whenever $\left\vert \lambda\right\vert +n<N$. We must
now prove Theorem \ref{thm.7.24} whenever $\left\vert \lambda\right\vert +n=N$.

So we fix some $n$ and $\lambda$ satisfying $\left\vert \lambda\right\vert
+n=N$. Then, $\left\vert \lambda\right\vert +n=N>0$, so that $n>0$ (since
$n=0$ would force $\lambda=\varnothing$ and thus $\left\vert \lambda
\right\vert +n=\left\vert \varnothing\right\vert +0=\left\vert \varnothing
\right\vert =0$). Hence, $\lambda_{n}$ is well-defined. We are in one of the
following two cases:

\textit{Case 1:} We have $\lambda_{n}=0$.

\textit{Case 2:} We have $\lambda_{n}\geq1$.

Let us consider Case 1. In this case, we have $\lambda_{n}=0$. Thus, the
partition $\lambda$ has length $<n=\dim V$. Hence, Theorem \ref{thm.7.25}
yields%
\begin{equation}
S_{\lambda}\left(  V\right)  =\sum_{L\subseteq V\text{ line}}S_{\lambda
}\left(  V\sslash L\right)  . \label{pf.thm.7.24.c1.1}%
\end{equation}

Now, let $L$ be any line in $V$. Then, $\dim L=1$. But what we said about
internal quotients long ago\footnote{specifically: the fact that if
$\mathbf{A}$ is an integral domain, and if $U\subseteq V$ are two
finite-dimensional $F$-vector subspaces of $\mathbf{A}$, then the internal
quotient $V\sslash U$ is isomorphic to the actual quotient $V/U$} entails that
$V\sslash L\cong V/L$ as $F$-vector spaces, via the canonical isomorphism%
\begin{align}
V/L  &  \rightarrow V\sslash L,\nonumber\\
\overline{v}  &  \mapsto\widetilde{f}_{L}\left(  v\right)  =f_{L}\left(
v\right)  . \label{pf.thm.7.24.c1.iso}%
\end{align}
Thus, $\dim\left(  V\sslash L\right)  =\dim\left(  V/L\right)
=\underbrace{\dim V}_{=n}-\underbrace{\dim L}_{=1}=n-1$. In other words, the
vector space $V\sslash L$ is $\left(  n-1\right)  $-dimensional. Hence, by the
induction hypothesis, we can apply Theorem \ref{thm.7.24} to $V\sslash L$ and
$n-1$ instead of $V$ and $n$ (since $\lambda_{n}=0$ yields $\lambda=\left(
\lambda_{1},\lambda_{2},\ldots,\lambda_{n-1}\right)  $). We thus obtain%
\begin{equation}
S_{\lambda}\left(  V\sslash L\right)  =\sum_{\substack{V\sslash L=W_{0}%
>W_{1}>\cdots>W_{n-1}=0\\\text{is a complete flag in }V\sslash L}%
}\ \ \prod_{i=1}^{n-1}H_{\lambda_{i}}\left(  W_{i-1}\sslash W_{i}\right)  .
\label{pf.thm.7.24.c1.2}%
\end{equation}

However, it is well-known (the \textquotedblleft lattice isomorphism
theorem\textquotedblright) that there is an inclusion-respecting
bijection\footnote{Here and in the following, the word \textquotedblleft
subspace\textquotedblright\ means \textquotedblleft$F$-vector
subspace\textquotedblright. A map $\phi$ whose domain and target consist of
sets is said to be \emph{inclusion-respecting} if it has the property that if
two sets $A$ and $B$ in its domain satisfy $A\subseteq B$, then $\phi\left(
A\right)  \subseteq\phi\left(  B\right)  $.}%
\begin{align*}
\left\{  \text{subspaces }U\text{ of }V\text{ such that }V\geq U\geq
L\right\}   &  \rightarrow\left\{  \text{subspaces of }V/L\right\}  ,\\
U  &  \mapsto U/L,
\end{align*}
which furthermore reduces the dimension of the subspace by $1$ (meaning that
$\dim\left(  U/L\right)  =\dim U-1$ for any $U$).

We can use the isomorphism (\ref{pf.thm.7.24.c1.iso}) to replace $V/L$ by
$V\sslash L$ here; thus we obtain an inclusion-respecting bijection%
\begin{align*}
\left\{  \text{subspaces }U\text{ of }V\text{ such that }V\geq U\geq
L\right\}   &  \rightarrow\left\{  \text{subspaces of }V\sslash L\right\}  ,\\
U  &  \mapsto\widetilde{f}_{L}\left(  U\right)  =U\sslash L,
\end{align*}
which again reduces the dimension of the subspace by $1$ (since $\dim\left(
U\sslash L\right)  =\dim\left(  U/L\right)  =\dim U-1$ for any $U$).

Since this bijection is inclusion-respecting and reduces dimension by $1$, we
thus obtain a bijection%
\begin{align*}
&  \left\{  \text{complete flags }V=V_{0}>V_{1}>\cdots>V_{n-1}=L>0\text{ in
}V\right\} \\
&  \rightarrow\left\{  \text{complete flags }V\sslash L=W_{0}>W_{1}%
>\cdots>W_{n-1}=0\text{ in }V\sslash L\right\}
\end{align*}
that sends each flag $V=V_{0}>V_{1}>\cdots>V_{n-1}=L>0$ to the flag
$V\sslash
L=W_{0}>W_{1}>\cdots>W_{n-1}=0$ given by $W_{i}=V_{i}\sslash L$. Using this
bijection to reindex the sum in (\ref{pf.thm.7.24.c1.2}), we can rewrite
(\ref{pf.thm.7.24.c1.2}) as%
\begin{equation}
S_{\lambda}\left(  V\sslash L\right)  =\sum_{\substack{V=V_{0}>V_{1}%
>\cdots>V_{n-1}=L>0\\\text{is a complete flag in }V}}\ \ \prod_{i=1}%
^{n-1}H_{\lambda_{i}}\left(  \left(  V_{i-1}\sslash L\right)  \sslash\left(
V_{i}\sslash L\right)  \right)  .\ \ \ \ \ \ \ \ \ \ \label{pf.thm.7.24.c1.3}%
\end{equation}

However, \cite[(7.16)]{Macdon92} says that if $U,V,T$ are three subspaces of
$\mathbf{A}$ satisfying $T\leq U\leq V$, then%
\[
V\sslash U=\left(  V\sslash T\right)  \sslash \left(  U\sslash T\right)  .
\]
Hence, for each $i\in\left[  n-1\right]  $ and each complete flag
$V=V_{0}>V_{1}>\cdots>V_{n-1}=L>0$ in $V$, we obtain%
\begin{equation}
V_{i-1}\sslash V_{i}=\left(  V_{i-1}\sslash L\right)  \sslash \left(
V_{i}\sslash L\right)  \label{pf.thm.7.24.c1.4}%
\end{equation}
(since $L\leq V_{i}\leq V_{i-1}$). Thus, (\ref{pf.thm.7.24.c1.3}) becomes%
\begin{align}
S_{\lambda}\left(  V\sslash L\right)   &  =\sum_{\substack{V=V_{0}%
>V_{1}>\cdots>V_{n-1}=L>0\\\text{is a complete flag in }V}}\ \ \prod
_{i=1}^{n-1}H_{\lambda_{i}}\left(  \underbrace{\left(  V_{i-1}%
\sslash L\right)  \sslash \left(  V_{i}\sslash L\right)  }_{\substack{=V_{i-1}%
\sslash V_{i}\\\text{(by (\ref{pf.thm.7.24.c1.4}))}}}\right) \nonumber\\
&  =\sum_{\substack{V=V_{0}>V_{1}>\cdots>V_{n-1}=L>0\\\text{is a complete flag
in }V}}\ \ \prod_{i=1}^{n-1}H_{\lambda_{i}}\left(  V_{i-1}\sslash V_{i}%
\right)  . \label{pf.thm.7.24.c1.5}%
\end{align}

Forget that we fixed $L$. We thus have proved (\ref{pf.thm.7.24.c1.5}) for
each line $L$ in $V$.

Substituting (\ref{pf.thm.7.24.c1.5}) into (\ref{pf.thm.7.24.c1.1}), we find%
\begin{align}
S_{\lambda}\left(  V\right)   &  =\underbrace{\sum_{L\subseteq V\text{ line}%
}\ \ \sum_{\substack{V=V_{0}>V_{1}>\cdots>V_{n-1}=L>0\\\text{is a complete
flag in }V}}}_{\substack{=\sum_{\substack{V=V_{0}>V_{1}>\cdots>V_{n-1}%
>0\\\text{is a complete flag in }V}}\\=\sum_{\substack{V=V_{0}>V_{1}%
>\cdots>V_{n}=0\\\text{is a complete flag in }V}}\\\text{(here, we have set
}V_{n}=0\text{ for the sake of uniformity)}}}\ \ \prod_{i=1}^{n-1}%
H_{\lambda_{i}}\left(  V_{i-1}\sslash V_{i}\right) \nonumber\\
&  =\sum_{\substack{V=V_{0}>V_{1}>\cdots>V_{n}=0\\\text{is a complete flag in
}V}}\ \ \prod_{i=1}^{n-1}H_{\lambda_{i}}\left(  V_{i-1}\sslash V_{i}\right)  .
\label{pf.thm.7.24.c1.6}%
\end{align}

However, $\lambda_{n}=0$, and thus $H_{\lambda_{n}}\left(  V_{n-1}%
\sslash V_{n}\right)  =H_{0}\left(  V_{n-1}\sslash V_{n}\right)  =1$ (since
$H_{0}\left(  U\right)  =1$ for any subspace $U$ of $\mathbf{A}$). Hence, the
product $\prod_{i=1}^{n-1}H_{\lambda_{i}}\left(  V_{i-1}\sslash V_{i}\right)
$ does not change if we extend it to encompass $i=n$ (since the extra factor
that it thus gains is $H_{\lambda_{n}}\left(  V_{n-1}\sslash V_{n}\right)
=1$). Thus, by extending this product in this way, we can rewrite
(\ref{pf.thm.7.24.c1.6}) as%
\[
S_{\lambda}\left(  V\right)  =\sum_{\substack{V=V_{0}>V_{1}>\cdots
>V_{n}=0\\\text{is a complete flag in }V}}\ \ \prod_{i=1}^{n}H_{\lambda_{i}%
}\left(  V_{i-1}\sslash V_{i}\right)  .
\]
In other words, Theorem \ref{thm.7.24} holds for our $n$ and $\lambda$. This
completes the induction step in Case 1.

Let us now consider Case 2. In this case, we have $\lambda_{n}\geq1$. Hence,
Lemma \ref{lem.fullhouse-1} yields
\begin{equation}
S_{\lambda}\left(  V\right)  =\left(  -1\right)  ^{n}\pi\left(  V\right)
\cdot\left(  S_{\lambda\ominus1}\left(  V\right)  \right)  ^{q},
\label{pf.pf.thm.7.24.c2.4}%
\end{equation}
where $\lambda\ominus1$ has been defined in Lemma \ref{lem.fullhouse-1}. Since
$\left\vert \lambda\ominus1\right\vert =\left\vert \lambda\right\vert
-n<\left\vert \lambda\right\vert $ (because $n>0$), we have $\left\vert
\lambda\ominus1\right\vert +n<\left\vert \lambda\right\vert +n$. Thus, the
induction hypothesis shows that Theorem \ref{thm.7.24} holds for
$\lambda\ominus1$ instead of $\lambda$. That is, we have%
\begin{equation}
S_{\lambda\ominus1}\left(  V\right)  =\sum_{\substack{V=V_{0}>V_{1}%
>\cdots>V_{n}=0\\\text{is a complete flag in }V}}\ \ \prod_{i=1}^{n}%
H_{\lambda_{i}-1}\left(  V_{i-1}\sslash V_{i}\right)
\label{pf.pf.thm.7.24.c2.5}%
\end{equation}
(since $\lambda\ominus1=\left(  \lambda_{1}-1,\lambda_{2}-1,\ldots,\lambda
_{n}-1\right)  $). Taking this equality to the $q$-th power, we find%
\begin{align*}
\left(  S_{\lambda\ominus1}\left(  V\right)  \right)  ^{q}  &  =\left(
\sum_{\substack{V=V_{0}>V_{1}>\cdots>V_{n}=0\\\text{is a complete flag in }%
V}}\ \ \prod_{i=1}^{n}H_{\lambda_{i}-1}\left(  V_{i-1}\sslash V_{i}\right)
\right)  ^{q}\\
&  =\varphi\left(  \sum_{\substack{V=V_{0}>V_{1}>\cdots>V_{n}=0\\\text{is a
complete flag in }V}}\ \ \prod_{i=1}^{n}H_{\lambda_{i}-1}\left(
V_{i-1}\sslash V_{i}\right)  \right) \\
&  \ \ \ \ \ \ \ \ \ \ \ \ \ \ \ \ \ \ \ \ \left(  \text{by the definition of
}\varphi\right) \\
&  =\sum_{\substack{V=V_{0}>V_{1}>\cdots>V_{n}=0\\\text{is a complete flag in
}V}}\ \ \prod_{i=1}^{n}\varphi\left(  H_{\lambda_{i}-1}\left(  V_{i-1}%
\sslash V_{i}\right)  \right) \\
&  \ \ \ \ \ \ \ \ \ \ \ \ \ \ \ \ \ \ \ \ \left(  \text{since }\varphi\text{
is an }F\text{-algebra morphism}\right)  .
\end{align*}
Substituting this into (\ref{pf.pf.thm.7.24.c2.4}), we obtain%
\begin{align*}
S_{\lambda}\left(  V\right)   &  =\left(  -1\right)  ^{n}\pi\left(  V\right)
\cdot\sum_{\substack{V=V_{0}>V_{1}>\cdots>V_{n}=0\\\text{is a complete flag in
}V}}\ \ \prod_{i=1}^{n}\varphi\left(  H_{\lambda_{i}-1}\left(  V_{i-1}%
\sslash V_{i}\right)  \right) \\
&  =\sum_{\substack{V=V_{0}>V_{1}>\cdots>V_{n}=0\\\text{is a complete flag in
}V}}\ \ \underbrace{\left(  -1\right)  ^{n}}_{=\prod_{i=1}^{n}\left(
-1\right)  }\ \ \underbrace{\pi\left(  V\right)  }_{\substack{=\prod_{i=1}%
^{n}\pi\left(  V_{i-1}\sslash V_{i}\right)  \\\text{(by Lemma \ref{lem.pi1})}%
}}\prod_{i=1}^{n}\varphi\left(  H_{\lambda_{i}-1}\left(  V_{i-1}%
\sslash V_{i}\right)  \right) \\
&  =\sum_{\substack{V=V_{0}>V_{1}>\cdots>V_{n}=0\\\text{is a complete flag in
}V}}\ \ \underbrace{\prod_{i=1}^{n}\left(  -1\right)  \cdot\prod_{i=1}^{n}%
\pi\left(  V_{i-1}\sslash V_{i}\right)  \cdot\prod_{i=1}^{n}\varphi\left(
H_{\lambda_{i}-1}\left(  V_{i-1}\sslash V_{i}\right)  \right)  }_{=\prod
_{i=1}^{n}\left(  -\pi\left(  V_{i-1}\sslash V_{i}\right)  \cdot\varphi\left(
H_{\lambda_{i}-1}\left(  V_{i-1}\sslash V_{i}\right)  \right)  \right)  }\\
&  =\sum_{\substack{V=V_{0}>V_{1}>\cdots>V_{n}=0\\\text{is a complete flag in
}V}}\ \ \prod_{i=1}^{n}\left(  -\pi\left(  V_{i-1}\sslash V_{i}\right)
\cdot\varphi\left(  H_{\lambda_{i}-1}\left(  V_{i-1}\sslash V_{i}\right)
\right)  \right)  .
\end{align*}

However, if $V=V_{0}>V_{1}>\cdots>V_{n}=0$ is a complete flag in $V$, and if
$i\in\left[  n\right]  $ is arbitrary, then the $F$-vector subspace
$V_{i-1}\sslash V_{i}$ of $\mathbf{A}$ is $1$-dimensional\footnote{Indeed, the
same reasoning that gave us $\dim\left(  V\sslash L\right)  =\dim\left(
V/L\right)  $ above can be used to show that $\dim\left(  V_{i-1}\sslash
V_{i}\right)  =\dim\left(  V_{i-1}/V_{i}\right)  $. But $\dim\left(
V_{i-1}/V_{i}\right)  =1$ since $V=V_{0}>V_{1}>\cdots>V_{n}=0$ is a complete
flag. Hence, $\dim\left(  V_{i-1}\sslash V_{i}\right)  =\dim\left(
V_{i-1}/V_{i}\right)  =1$, so that $V_{i-1}\sslash V_{i}$ is $1$%
-dimensional.}, and thus satisfies
\[
\pi\left(  V_{i-1}\sslash V_{i}\right)  \cdot\varphi\left(  H_{\lambda_{i}%
-1}\left(  V_{i-1}\sslash V_{i}\right)  \right)  =-H_{\lambda_{i}}\left(
V_{i-1}\sslash V_{i}\right)
\]
(by Lemma \ref{lem.Hri1}, applied to $U=V_{i-1}\sslash V_{i}$ and
$r=\lambda_{i}$), so that%
\begin{equation}
-\pi\left(  V_{i-1}\sslash V_{i}\right)  \cdot\varphi\left(  H_{\lambda_{i}%
-1}\left(  V_{i-1}\sslash V_{i}\right)  \right)  =H_{\lambda_{i}}\left(
V_{i-1}\sslash V_{i}\right)  . \label{pf.pf.thm.7.24.c2.7}%
\end{equation}
Hence, we can continue our computation of $S_{\lambda}\left(  V\right)  $ as
follows:%
\begin{align*}
S_{\lambda}\left(  V\right)   &  =\sum_{\substack{V=V_{0}>V_{1}>\cdots
>V_{n}=0\\\text{is a complete flag in }V}}\ \ \prod_{i=1}^{n}%
\underbrace{\left(  -\pi\left(  V_{i-1}\sslash V_{i}\right)  \cdot
\varphi\left(  H_{\lambda_{i}-1}\left(  V_{i-1}\sslash V_{i}\right)  \right)
\right)  }_{\substack{=H_{\lambda_{i}}\left(  V_{i-1}\sslash V_{i}\right)
\\\text{(by (\ref{pf.pf.thm.7.24.c2.7}))}}}\\
&  =\sum_{\substack{V=V_{0}>V_{1}>\cdots>V_{n}=0\\\text{is a complete flag in
}V}}\ \ \prod_{i=1}^{n}H_{\lambda_{i}}\left(  V_{i-1}\sslash V_{i}\right)  .
\end{align*}
In other words, Theorem \ref{thm.7.24} holds for our $n$ and $\lambda$. This
completes the induction step in Case 2.

Hence, the induction step is completed in both cases. Thus, the inductive
proof of Theorem \ref{thm.7.24} is complete.
\end{proof}

\appendix

\section{Appendix: Folklore proofs}

\subsection{\label{sec.apx-coprod}Appendix: Details for the proof of Lemma
\ref{lem.Stilde.coproduct}}

In our above proof of Lemma \ref{lem.Stilde.coproduct}, we have said that
(\ref{eq.Stilde.coproduct}) can be derived from
(\ref{pf.lem.Stilde.coproduct.HHE}) using the Cauchy--Binet theorem. Let us
explain in detail how this derivation proceeds. We will need a tailored
variant of the Cauchy--Binet theorem, which we shall first derive from a more
classical version.

For the rest of this section, we fix a commutative ring $R$.

We shall use matrices whose rows and columns can be indexed by arbitrary
objects, not just numbers. An $I\times J$\emph{-matrix} (where $I$ and $J$ are
two sets) is a matrix whose rows are indexed by the elements of $I$ and whose
columns are indexed by the elements of $J$. The space of all $I\times
J$-matrices over $R$ will be called $R^{I\times J}$. If $A=\left(
a_{i,j}\right)  _{i\in I,\ j\in J}$ is any $I\times J$-matrix, and if $\left(
i_{1},i_{2},\ldots,i_{k}\right)  \in I^{k}$ and $\left(  j_{1},j_{2}%
,\ldots,j_{\ell}\right)  \in J^{\ell}$ are any finite lists of elements of $I$
and $J$, respectively (for some $k,\ell\in\mathbb{N}$), then
$\operatorname*{sub}\nolimits_{i_{1},i_{2},\ldots,i_{k}}^{j_{1},j_{2}%
,\ldots,j_{\ell}}A$ shall denote the $k\times\ell$-matrix $\left(
a_{i_{x},j_{y}}\right)  _{x\in\left[  k\right]  ,\ y\in\left[  \ell\right]  }$.

One of the forms of the \emph{Cauchy--Binet theorem} (see, e.g.,
\cite[Corollary 7.182]{detnotes}) says the following:

\begin{proposition}
\label{prop.CB1}Let $A\in R^{n\times p}$ and $B\in R^{p\times m}$ be two
matrices over $R$. Let $u\in\mathbb{N}$. Let $\left(  i_{1},i_{2},\ldots
,i_{u}\right)  \in\left[  n\right]  ^{u}$ and $\left(  j_{1},j_{2}%
,\ldots,j_{u}\right)  \in\left[  m\right]  ^{u}$ be two $u$-tuples of elements
of $\left[  n\right]  $ and $\left[  m\right]  $, respectively. Then,%
\begin{align}
&  \det\left(  \operatorname*{sub}\nolimits_{i_{1},i_{2},\ldots,i_{u}}%
^{j_{1},j_{2},\ldots,j_{u}}\left(  AB\right)  \right) \nonumber\\
&  =\sum_{g_{1}<g_{2}<\cdots<g_{u}}\det\left(  \operatorname*{sub}%
\nolimits_{i_{1},i_{2},\ldots,i_{u}}^{g_{1},g_{2},\ldots,g_{u}}A\right)
\cdot\det\left(  \operatorname*{sub}\nolimits_{g_{1},g_{2},\ldots,g_{u}%
}^{j_{1},j_{2},\ldots,j_{u}}B\right)
,\ \ \ \ \ \ \ \ \ \ \label{pf.lem.Stilde.coproduct.CB1}%
\end{align}
where the sum ranges over all strictly increasing $u$-tuples $\left(
g_{1},g_{2},\ldots,g_{u}\right)  \in\left[  p\right]  ^{u}$.
\end{proposition}

By relabelling the rows and columns of both matrices $A$ and $B$ here, we can
rewrite this result in the following equivalent form:

\begin{proposition}
\label{prop.CB2}Let $N$ and $M$ be two finite sets, and let $P$ be a finite
totally ordered set. Let $A\in R^{N\times P}$ and $B\in R^{P\times M}$ be two
matrices over $R$. Let $u\in\mathbb{N}$. Let $\left(  i_{1},i_{2},\ldots
,i_{u}\right)  \in N^{u}$ and $\left(  j_{1},j_{2},\ldots,j_{u}\right)  \in
M^{u}$ be two $u$-tuples of elements of $N$ and $M$, respectively. Then,%
\begin{align}
&  \det\left(  \operatorname*{sub}\nolimits_{i_{1},i_{2},\ldots,i_{u}}%
^{j_{1},j_{2},\ldots,j_{u}}\left(  AB\right)  \right) \nonumber\\
&  =\sum_{g_{1}<g_{2}<\cdots<g_{u}}\det\left(  \operatorname*{sub}%
\nolimits_{i_{1},i_{2},\ldots,i_{u}}^{g_{1},g_{2},\ldots,g_{u}}A\right)
\cdot\det\left(  \operatorname*{sub}\nolimits_{g_{1},g_{2},\ldots,g_{u}%
}^{j_{1},j_{2},\ldots,j_{u}}B\right)
,\ \ \ \ \ \ \ \ \ \ \label{pf.lem.Stilde.coproduct.CB2a}%
\end{align}
where the sum ranges over all strictly increasing $u$-tuples $\left(
g_{1},g_{2},\ldots,g_{u}\right)  \in P^{u}$.
\end{proposition}

\begin{proof}
Just rename the elements of $N$ as $1,2,\ldots,n$, rename the elements of $M$
as $1,2,\ldots,m$, and rename the elements of $P$ as $1,2,\ldots,p$ in
increasing order. Then, the claimed equality
(\ref{pf.lem.Stilde.coproduct.CB2a}) becomes precisely
(\ref{pf.lem.Stilde.coproduct.CB1}).
\end{proof}

Reversing the order of the totally ordered set $P$, we can replace the
condition $g_{1}<g_{2}<\cdots<g_{u}$ under the summation sign in
(\ref{pf.lem.Stilde.coproduct.CB2a}) by the opposite condition $g_{1}%
>g_{2}>\cdots>g_{u}$ (that is, we can sum over the strictly decreasing
$u$-tuples instead of the strictly increasing ones). Thus, we transform
Proposition \ref{prop.CB2} into the following proposition:

\begin{proposition}
\label{prop.CB3}Let $N$ and $M$ be two finite sets, and let $P$ be a finite
totally ordered set. Let $A\in R^{N\times P}$ and $B\in R^{P\times M}$ be two
matrices over $R$. Let $u\in\mathbb{N}$. Let $\left(  i_{1},i_{2},\ldots
,i_{u}\right)  \in N^{u}$ and $\left(  j_{1},j_{2},\ldots,j_{u}\right)  \in
M^{u}$ be two $u$-tuples of elements of $N$ and $M$, respectively. Then,%
\begin{align}
&  \det\left(  \operatorname*{sub}\nolimits_{i_{1},i_{2},\ldots,i_{u}}%
^{j_{1},j_{2},\ldots,j_{u}}\left(  AB\right)  \right) \nonumber\\
&  =\sum_{g_{1}>g_{2}>\cdots>g_{u}}\det\left(  \operatorname*{sub}%
\nolimits_{i_{1},i_{2},\ldots,i_{u}}^{g_{1},g_{2},\ldots,g_{u}}A\right)
\cdot\det\left(  \operatorname*{sub}\nolimits_{g_{1},g_{2},\ldots,g_{u}%
}^{j_{1},j_{2},\ldots,j_{u}}B\right)
,\ \ \ \ \ \ \ \ \ \ \label{pf.lem.Stilde.coproduct.CB3}%
\end{align}
where the sum ranges over all strictly decreasing $u$-tuples $\left(
g_{1},g_{2},\ldots,g_{u}\right)  \in P^{u}$.
\end{proposition}

Now, let us restrict this result to upper-triangular matrices.

First we recall how they are defined: If $P$ is a totally ordered set, then a
$P\times P$-matrix $C=\left(  c_{i,j}\right)  _{i\in P,\ j\in P}\in R^{P\times
P}$ is said to be \emph{upper-triangular} if its entries satisfy $c_{i,j}=0$
whenever $i>j$.

Now, if the matrices $A$ and $B$ in Proposition \ref{prop.CB3} are
upper-triangular, then the sum on the right hand side of
(\ref{pf.lem.Stilde.coproduct.CB3}) can be made significantly shorter by
removing many vanishing addends:

\begin{proposition}
\label{prop.CB4}Let $P$ be a finite totally ordered set. Let $A\in R^{P\times
P}$ and $B\in R^{P\times P}$ be two upper-triangular matrices over $R$. Let
$u\in\mathbb{N}$. Let $\left(  i_{1},i_{2},\ldots,i_{u}\right)  \in P^{u}$ and
$\left(  j_{1},j_{2},\ldots,j_{u}\right)  \in P^{u}$ be two $u$-tuples of
elements of $P$. Assume that $i_{1}>i_{2}>\cdots>i_{u}$ and $j_{1}%
>j_{2}>\cdots>j_{u}$. Then,%
\begin{align}
&  \det\left(  \operatorname*{sub}\nolimits_{i_{1},i_{2},\ldots,i_{u}}%
^{j_{1},j_{2},\ldots,j_{u}}\left(  AB\right)  \right) \nonumber\\
&  =\sum_{\substack{g_{1}>g_{2}>\cdots>g_{u};\\i_{k}\leq g_{k}\leq j_{k}\text{
for all }k\in\left[  u\right]  }}\det\left(  \operatorname*{sub}%
\nolimits_{i_{1},i_{2},\ldots,i_{u}}^{g_{1},g_{2},\ldots,g_{u}}A\right)
\cdot\det\left(  \operatorname*{sub}\nolimits_{g_{1},g_{2},\ldots,g_{u}%
}^{j_{1},j_{2},\ldots,j_{u}}B\right)
.\ \ \ \ \ \ \ \ \ \ \label{pf.lem.Stilde.coproduct.CB4}%
\end{align}

\end{proposition}

\begin{proof}
Write the $P\times P$-matrices $A$ and $B$ as $A=\left(  a_{i,j}\right)
_{i\in P,\ j\in P}$ and $B=\left(  b_{i,j}\right)  _{i\in P,\ j\in P}$.

Let $\left(  g_{1},g_{2},\ldots,g_{u}\right)  \in P^{u}$ be a strictly
decreasing $u$-tuple of elements of $P$. Thus, $g_{1}>g_{2}>\cdots>g_{u}$.

We will use the following classical and easy fact (see, e.g., \cite[Exercise
6.47 \textbf{(a)}]{detnotes}): If $C=\left(  c_{x,y}\right)  _{x,y\in\left[
u\right]  }\in R^{u\times u}$ is a $u\times u$-matrix, and if $X$ and $Y$ are
two subsets of $\left[  u\right]  $ satisfying $\left\vert X\right\vert
+\left\vert Y\right\vert >u$ and $\left(  c_{x,y}=0\text{ for all }x\in
X\text{ and }y\in Y\right)  $, then $\det C=0$. We will refer to this fact as
the \emph{too-many-zeroes lemma}. This lemma gives us two consequences in our
specific situation:

\begin{itemize}
\item If some $k\in\left[  u\right]  $ satisfies $i_{k}>g_{k}$, then
\begin{equation}
\det\left(  \operatorname*{sub}\nolimits_{i_{1},i_{2},\ldots,i_{u}}%
^{g_{1},g_{2},\ldots,g_{u}}A\right)  =0.
\label{pf.lem.Stilde.coproduct.CB4.pf.3}%
\end{equation}

[\textit{Proof:} Assume that some $k\in\left[  u\right]  $ satisfies
$i_{k}>g_{k}$. Consider this $k$.

We have $\operatorname*{sub}\nolimits_{i_{1},i_{2},\ldots,i_{u}}^{g_{1}%
,g_{2},\ldots,g_{u}}A=\left(  a_{i_{x},g_{y}}\right)  _{x,y\in\left[
u\right]  }$ (since $A=\left(  a_{i,j}\right)  _{i\in P,\ j\in P}$).

Now, let $x\in\left\{  1,2,\ldots,k\right\}  $ and $y\in\left\{
k,k+1,\ldots,u\right\}  $ be arbitrary. We shall show that $a_{i_{x},g_{y}}=0$.

Indeed, from $x\in\left\{  1,2,\ldots,k\right\}  $, we obtain $x\leq k$ and
thus $i_{x}\geq i_{k}$ (since $i_{1}>i_{2}>\cdots>i_{u}$). Moreover, from
$y\in\left\{  k,k+1,\ldots,u\right\}  $, we obtain $y\geq k$, hence $k\leq y$
and thus $g_{k}\geq g_{y}$ (since $g_{1}>g_{2}>\cdots>g_{u}$). Thus,
$i_{x}\geq i_{k}>g_{k}\geq g_{y}$, so that $a_{i_{x},g_{y}}=0$ (since the
matrix $A=\left(  a_{i,j}\right)  _{i\in P,\ j\in P}$ is upper-triangular).

Forget that we fixed $x$ and $y$. We thus have shown that $a_{i_{x},g_{y}}=0$
for all $x\in\left\{  1,2,\ldots,k\right\}  $ and $y\in\left\{  k,k+1,\ldots
,u\right\}  $. Hence, the too-many-zeroes lemma (applied to the matrix
$\operatorname*{sub}\nolimits_{i_{1},i_{2},\ldots,i_{u}}^{g_{1},g_{2}%
,\ldots,g_{u}}A=\left(  a_{i_{x},g_{y}}\right)  _{x,y\in\left[  u\right]  }$
instead of $C=\left(  c_{x,y}\right)  _{x,y\in\left[  u\right]  }$, and to the
sets $\left\{  1,2,\ldots,k\right\}  $ and $\left\{  k,k+1,\ldots,u\right\}  $
instead of $X$ and $Y$) yields $\det\left(  \operatorname*{sub}%
\nolimits_{i_{1},i_{2},\ldots,i_{u}}^{g_{1},g_{2},\ldots,g_{u}}A\right)  =0$
(since $\underbrace{\left\vert \left\{  1,2,\ldots,k\right\}  \right\vert
}_{=k}+\underbrace{\left\vert \left\{  k,k+1,\ldots,u\right\}  \right\vert
}_{=u-k+1}=k+\left(  u-k+1\right)  =u+1>u$). This proves
(\ref{pf.lem.Stilde.coproduct.CB4.pf.3}).]

\item If some $k\in\left[  u\right]  $ satisfies $g_{k}>j_{k}$, then
\begin{equation}
\det\left(  \operatorname*{sub}\nolimits_{g_{1},g_{2},\ldots,g_{u}}%
^{j_{1},j_{2},\ldots,j_{u}}B\right)  =0.
\label{pf.lem.Stilde.coproduct.CB4.pf.4}%
\end{equation}

[\textit{Proof:} This is analogous to the proof of
(\ref{pf.lem.Stilde.coproduct.CB4.pf.3}) (but now using $B$, $b_{i,j}$,
$g_{x}$ and $j_{y}$ instead of $A$, $a_{i,j}$, $i_{x}$ and $g_{y}$).]

\item Thus, if we don't have $\left(  i_{k}\leq g_{k}\leq j_{k}\text{ for all
}k\in\left[  u\right]  \right)  $, then%
\begin{equation}
\det\left(  \operatorname*{sub}\nolimits_{i_{1},i_{2},\ldots,i_{u}}%
^{g_{1},g_{2},\ldots,g_{u}}A\right)  \cdot\det\left(  \operatorname*{sub}%
\nolimits_{g_{1},g_{2},\ldots,g_{u}}^{j_{1},j_{2},\ldots,j_{u}}B\right)  =0.
\label{pf.lem.Stilde.coproduct.CB4.pf.5}%
\end{equation}

[\textit{Proof:} Assume that we don't have $\left(  i_{k}\leq g_{k}\leq
j_{k}\text{ for all }k\in\left[  u\right]  \right)  $. Hence, there exists
some $k\in\left[  u\right]  $ such that we don't have $i_{k}\leq g_{k}\leq
j_{k}$. Consider this $k$. Thus, we have $i_{k}>g_{k}$ or $g_{k}>j_{k}$ (since
we don't have $i_{k}\leq g_{k}\leq j_{k}$). In the former case, we have%
\[
\underbrace{\det\left(  \operatorname*{sub}\nolimits_{i_{1},i_{2},\ldots
,i_{u}}^{g_{1},g_{2},\ldots,g_{u}}A\right)  }_{\substack{=0\\\text{(by
(\ref{pf.lem.Stilde.coproduct.CB4.pf.3}), since }i_{k}>g_{k}\text{)}}%
}\cdot\det\left(  \operatorname*{sub}\nolimits_{g_{1},g_{2},\ldots,g_{u}%
}^{j_{1},j_{2},\ldots,j_{u}}B\right)  =0.
\]
In the latter case, we have%
\[
\det\left(  \operatorname*{sub}\nolimits_{i_{1},i_{2},\ldots,i_{u}}%
^{g_{1},g_{2},\ldots,g_{u}}A\right)  \cdot\underbrace{\det\left(
\operatorname*{sub}\nolimits_{g_{1},g_{2},\ldots,g_{u}}^{j_{1},j_{2}%
,\ldots,j_{u}}B\right)  }_{\substack{=0\\\text{(by
(\ref{pf.lem.Stilde.coproduct.CB4.pf.4}), since }g_{k}>j_{k}\text{)}}}=0.
\]
Thus, $\det\left(  \operatorname*{sub}\nolimits_{i_{1},i_{2},\ldots,i_{u}%
}^{g_{1},g_{2},\ldots,g_{u}}A\right)  \cdot\det\left(  \operatorname*{sub}%
\nolimits_{g_{1},g_{2},\ldots,g_{u}}^{j_{1},j_{2},\ldots,j_{u}}B\right)  =0$
is proved in both cases. This completes the proof of
(\ref{pf.lem.Stilde.coproduct.CB4.pf.5}).]
\end{itemize}

Forget that we fixed $\left(  g_{1},g_{2},\ldots,g_{u}\right)  $. We thus have
proved (\ref{pf.lem.Stilde.coproduct.CB4.pf.5}) for each strictly decreasing
$u$-tuple $\left(  g_{1},g_{2},\ldots,g_{u}\right)  \in P^{u}$ that does not
satisfy $\left(  i_{k}\leq g_{k}\leq j_{k}\text{ for all }k\in\left[
u\right]  \right)  $.

Now, (\ref{pf.lem.Stilde.coproduct.CB3}) (applied to $N=P$ and $M=P$) becomes
\begin{align*}
&  \det\left(  \operatorname*{sub}\nolimits_{i_{1},i_{2},\ldots,i_{u}}%
^{j_{1},j_{2},\ldots,j_{u}}\left(  AB\right)  \right) \\
&  =\sum_{g_{1}>g_{2}>\cdots>g_{u}}\det\left(  \operatorname*{sub}%
\nolimits_{i_{1},i_{2},\ldots,i_{u}}^{g_{1},g_{2},\ldots,g_{u}}A\right)
\cdot\det\left(  \operatorname*{sub}\nolimits_{g_{1},g_{2},\ldots,g_{u}%
}^{j_{1},j_{2},\ldots,j_{u}}B\right) \\
&  =\sum_{\substack{g_{1}>g_{2}>\cdots>g_{u};\\i_{k}\leq g_{k}\leq j_{k}\text{
for all }k\in\left[  u\right]  }}\det\left(  \operatorname*{sub}%
\nolimits_{i_{1},i_{2},\ldots,i_{u}}^{g_{1},g_{2},\ldots,g_{u}}A\right)
\cdot\det\left(  \operatorname*{sub}\nolimits_{g_{1},g_{2},\ldots,g_{u}%
}^{j_{1},j_{2},\ldots,j_{u}}B\right)
\end{align*}
(here, we have removed all addends that don't satisfy $\left(  i_{k}\leq
g_{k}\leq j_{k}\text{ for all }k\in\left[  u\right]  \right)  $ from our sum,
since (\ref{pf.lem.Stilde.coproduct.CB4.pf.5}) shows that all these addends
are $0$). This proves (\ref{pf.lem.Stilde.coproduct.CB4}).
\end{proof}

We shall now extend (\ref{pf.lem.Stilde.coproduct.CB4}) to infinite matrices.
It is not always possible to multiply two $P\times P$-matrices $A,B\in
R^{P\times P}$ when the set $P$ is infinite; for example, the product $\left(
1\right)  _{i,j\in\mathbb{Z}}\cdot\left(  1\right)  _{i,j\in\mathbb{Z}}$ makes
no sense because its entries would be the divergent infinite sums $\sum
_{k\in\mathbb{Z}}1$. However, in some cases, such products are defined. One
such case is when the matrices are upper-triangular and the set $P$ is
equipped with an interval-finite total order. We recall the definition:

If $P$ is a totally ordered set, and if $i,j\in P$ are two elements, then the
\emph{interval} $\left[  i,j\right]  _{P}$ is defined to be the set $\left\{
k\in P\ \mid\ i\leq k\leq j\right\}  $. A totally ordered set $P$ is said to
be \emph{interval-finite} if for any two elements $i,j\in P$, the interval
$\left[  i,j\right]  _{P}=\left\{  k\in P\ \mid\ i\leq k\leq j\right\}  $ is
finite. If $P$ is an interval-finite totally ordered set, and if $A=\left(
a_{i,j}\right)  _{i\in P,\ j\in P}\in R^{P\times P}$ and $B=\left(
b_{i,j}\right)  _{i\in P,\ j\in P}\in R^{P\times P}$ are two upper-triangular
$P\times P$-matrices, then the product $AB$ is well-defined, since its
$\left(  i,j\right)  $-th entry (for all $i,j\in P$) is
\[
\sum_{k\in P}\underbrace{a_{i,k}b_{k,j}}_{\substack{=0\text{ unless }%
k\in\left[  i,j\right]  _{P}\\\text{(because }a_{i,k}=0\text{ when
}k<i\text{,}\\\text{whereas }b_{k,j}=0\text{ when }k>j\text{)}}%
}=\underbrace{\sum_{k\in\left[  i,j\right]  _{P}}a_{i,k}b_{k,j}}%
_{\substack{\text{a finite sum,}\\\text{since }\left[  i,j\right]  _{P}\text{
is finite}}}.
\]
Moreover, this product $AB$ is itself upper-triangular (since $\left[
i,j\right]  _{P}=\varnothing$ unless $i\leq j$).

We can now generalize (\ref{pf.lem.Stilde.coproduct.CB4}) to infinite matrices:

\begin{proposition}
\label{prop.CB5}Let $P$ be an interval-finite totally ordered set. Let $A\in
R^{P\times P}$ and $B\in R^{P\times P}$ be two upper-triangular matrices over
$R$. Let $u\in\mathbb{N}$. Let $\left(  i_{1},i_{2},\ldots,i_{u}\right)  \in
P^{u}$ and $\left(  j_{1},j_{2},\ldots,j_{u}\right)  \in P^{u}$ be two
$u$-tuples of elements of $P$. Assume that $i_{1}>i_{2}>\cdots>i_{u}$ and
$j_{1}>j_{2}>\cdots>j_{u}$. Then,%
\begin{align}
&  \det\left(  \operatorname*{sub}\nolimits_{i_{1},i_{2},\ldots,i_{u}}%
^{j_{1},j_{2},\ldots,j_{u}}\left(  AB\right)  \right) \nonumber\\
&  =\sum_{\substack{g_{1}>g_{2}>\cdots>g_{u};\\i_{k}\leq g_{k}\leq j_{k}\text{
for all }k\in\left[  u\right]  }}\det\left(  \operatorname*{sub}%
\nolimits_{i_{1},i_{2},\ldots,i_{u}}^{g_{1},g_{2},\ldots,g_{u}}A\right)
\cdot\det\left(  \operatorname*{sub}\nolimits_{g_{1},g_{2},\ldots,g_{u}%
}^{j_{1},j_{2},\ldots,j_{u}}B\right)
.\ \ \ \ \ \ \ \ \ \ \label{pf.lem.Stilde.coproduct.CB5}%
\end{align}

\end{proposition}

\begin{proof}
If $u=0$, then the claim is trivial (since the determinant of a $0\times
0$-matrix is $1$). Thus, we WLOG assume that $u\geq1$.

Pick two elements $m_{-}\in P$ and $m_{+}\in P$ such that all the $2u$
elements $i_{1},i_{2},\ldots,i_{u},j_{1},j_{2},\ldots,j_{u}$ belong to the
interval $\left[  m_{-},m_{+}\right]  _{P}$ (this is always possible: just let
$m_{-}$ be the smallest of these $2u$ elements, and $m_{+}$ be the largest of
them). This interval $\left[  m_{-},m_{+}\right]  _{P}$ is finite (since $P$
is interval-finite). Moreover, the definition of $m_{-}$ and $m_{+}$ yields
$\left(  i_{1},i_{2},\ldots,i_{u}\right)  \in\left[  m_{-},m_{+}\right]
_{P}^{u}$ and $\left(  j_{1},j_{2},\ldots,j_{u}\right)  \in\left[  m_{-}%
,m_{+}\right]  _{P}^{u}$.

Write the $P\times P$-matrices $A$ and $B$ as $A=\left(  a_{i,j}\right)
_{i\in P,\ j\in P}$ and $B=\left(  b_{i,j}\right)  _{i\in P,\ j\in P}$.

For each upper-triangular $P\times P$-matrix $C=\left(  c_{i,j}\right)  _{i\in
P,\ j\in P}\in R^{P\times P}$, we let $\widetilde{C}$ be the $\left[
m_{-},m_{+}\right]  _{P}\times\left[  m_{-},m_{+}\right]  _{P}$-matrix
$\left(  c_{i,j}\right)  _{i,j\in\left[  m_{-},m_{+}\right]  _{P}}\in
R^{\left[  m_{-},m_{+}\right]  _{P}\times\left[  m_{-},m_{+}\right]  _{P}}$.
(This is simply the submatrix of $C$ in which only the rows and the columns
indexed by the elements of $\left[  m_{-},m_{+}\right]  _{P}$ have been kept.)
Thus, $\widetilde{A}=\left(  a_{i,j}\right)  _{i,j\in\left[  m_{-}%
,m_{+}\right]  _{P}}$ and $\widetilde{B}=\left(  b_{i,j}\right)
_{i,j\in\left[  m_{-},m_{+}\right]  _{P}}$. Clearly, these matrices
$\widetilde{A}$ and $\widetilde{B}$ are upper-triangular (since $A$ and $B$ are).

It is easy to see that%
\begin{equation}
\widetilde{AB}=\widetilde{A}\cdot\widetilde{B}.
\label{pf.lem.Stilde.coproduct.CB5.1}%
\end{equation}
\footnote{\textit{Proof:} Let $i,j\in\left[  m_{-},m_{+}\right]  _{P}$. Then,
the $\left(  i,j\right)  $-th entry of $\widetilde{AB}$ is the sum $\sum_{k\in
P}a_{i,k}b_{k,j}$, whereas the $\left(  i,j\right)  $-th entry of
$\widetilde{A}\cdot\widetilde{B}$ is the sum $\sum_{k\in\left[  m_{-}%
,m_{+}\right]  _{P}}a_{i,k}b_{k,j}$. But these two sums are equal, since
\begin{align*}
\sum_{k\in P}a_{i,k}b_{k,j}  &  =\sum_{\substack{k\in P;\\k<m_{-}%
}}\underbrace{a_{i,k}}_{\substack{=0\\\text{(since }k<m_{-}\leq
i\\\text{(because }i\in\left[  m_{-},m_{+}\right]  _{P}\text{),}\\\text{hence
}i>k\text{,}\\\text{but }A\text{ is upper-triangular)}}}b_{k,j}%
+\underbrace{\sum_{\substack{k\in P;\\m_{-}\leq k\leq m_{+}}}}_{=\sum
_{k\in\left[  m_{-},m_{+}\right]  _{P}}}a_{i,k}b_{k,j}+\sum_{\substack{k\in
P;\\k>m_{+}}}a_{i,k}\underbrace{b_{k,j}}_{\substack{=0\\\text{(since }%
k>m_{+}\geq j\\\text{(because }j\in\left[  m_{-},m_{+}\right]  _{P}%
\text{),}\\\text{but }B\text{ is upper-triangular)}}}\\
&  \ \ \ \ \ \ \ \ \ \ \ \ \ \ \ \ \ \ \ \ \left(
\begin{array}
[c]{c}%
\text{because each }k\in P\text{ satisfies}\\
\text{either }k<m_{-}\text{ or }m_{-}\leq k\leq m_{+}\text{ or }k>m_{+}%
\end{array}
\right) \\
&  =\underbrace{\sum_{\substack{k\in P;\\k<m_{-}}}0b_{k,j}}_{=0}+\sum
_{k\in\left[  m_{-},m_{+}\right]  _{P}}a_{i,k}b_{k,j}+\underbrace{\sum
_{\substack{k\in P;\\k>m_{+}}}a_{i,k}0}_{=0}=\sum_{k\in\left[  m_{-}%
,m_{+}\right]  _{P}}a_{i,k}b_{k,j}.
\end{align*}
So we have showed that the $\left(  i,j\right)  $-th entries of the matrices
$\widetilde{AB}$ and $\widetilde{A}\cdot\widetilde{B}$ are equal. Since we
have showed this for all $i,j\in\left[  m_{-},m_{+}\right]  _{P}$, we thus
conclude that $\widetilde{AB}=\widetilde{A}\cdot\widetilde{B}$.}.

But we can apply (\ref{pf.lem.Stilde.coproduct.CB4}) to $\left[  m_{-}%
,m_{+}\right]  _{P}$, $\widetilde{A}$ and $\widetilde{B}$ instead of $P$, $A$
and $B$ (since $\left[  m_{-},m_{+}\right]  _{P}$ is finite). Thus we obtain%
\begin{align}
&  \det\left(  \operatorname*{sub}\nolimits_{i_{1},i_{2},\ldots,i_{u}}%
^{j_{1},j_{2},\ldots,j_{u}}\left(  \widetilde{A}\cdot\widetilde{B}\right)
\right) \nonumber\\
&  =\sum_{\substack{g_{1}>g_{2}>\cdots>g_{u};\\i_{k}\leq g_{k}\leq j_{k}\text{
for all }k\in\left[  u\right]  }}\det\left(  \operatorname*{sub}%
\nolimits_{i_{1},i_{2},\ldots,i_{u}}^{g_{1},g_{2},\ldots,g_{u}}\widetilde{A}%
\right)  \cdot\det\left(  \operatorname*{sub}\nolimits_{g_{1},g_{2}%
,\ldots,g_{u}}^{j_{1},j_{2},\ldots,j_{u}}\widetilde{B}\right)  .
\label{pf.lem.Stilde.coproduct.CB5.3}%
\end{align}
Here, the indices $g_{1},g_{2},\ldots,g_{u}$ in the sum are supposed to belong
to $\left[  m_{-},m_{+}\right]  _{P}$, but we could just as well relax this
requirement and instead demand them to belong to $P$, because the second
condition \textquotedblleft$i_{k}\leq g_{k}\leq j_{k}$ for all $k\in\left[
u\right]  $\textquotedblright\ would force them to belong to $\left[
m_{-},m_{+}\right]  _{P}$ anyway (indeed, if $i_{k}\leq g_{k}\leq j_{k}$ for
all $k\in\left[  u\right]  $, then each $k\in\left[  u\right]  $ satisfies
$m_{-}\leq i_{k}\leq g_{k}\leq j_{k}\leq m_{+}$ and thus $g_{k}\in\left[
m_{-},m_{+}\right]  _{P}$). Thus, the sum on the right hand side of
(\ref{pf.lem.Stilde.coproduct.CB5.3}) ranges over the exact same $u$-tuples
$\left(  g_{1},g_{2},\ldots,g_{u}\right)  $ as the sum on the right hand side
of (\ref{pf.lem.Stilde.coproduct.CB5}). Moreover, each such $u$-tuple $\left(
g_{1},g_{2},\ldots,g_{u}\right)  $ satisfies
\[
\operatorname*{sub}\nolimits_{i_{1},i_{2},\ldots,i_{u}}^{g_{1},g_{2}%
,\ldots,g_{u}}\widetilde{A}=\operatorname*{sub}\nolimits_{i_{1},i_{2}%
,\ldots,i_{u}}^{g_{1},g_{2},\ldots,g_{u}}A
\]
(since $\widetilde{A}$ is a submatrix of $A$ that preserves the same indexing
as $A$: i.e., the $\left(  i,j\right)  $-th entry of $\widetilde{A}$ equals
the $\left(  i,j\right)  $-th entry of $A$ whenever the former is defined)
and
\[
\operatorname*{sub}\nolimits_{g_{1},g_{2},\ldots,g_{u}}^{j_{1},j_{2}%
,\ldots,j_{u}}\widetilde{B}=\operatorname*{sub}\nolimits_{g_{1},g_{2}%
,\ldots,g_{u}}^{j_{1},j_{2},\ldots,j_{u}}B
\]
(similarly). Hence, we can rewrite (\ref{pf.lem.Stilde.coproduct.CB5.3}) as%
\begin{align*}
&  \det\left(  \operatorname*{sub}\nolimits_{i_{1},i_{2},\ldots,i_{u}}%
^{j_{1},j_{2},\ldots,j_{u}}\left(  \widetilde{A}\cdot\widetilde{B}\right)
\right) \\
&  =\sum_{\substack{g_{1}>g_{2}>\cdots>g_{u};\\i_{k}\leq g_{k}\leq j_{k}\text{
for all }k\in\left[  u\right]  }}\det\underbrace{\left(  \operatorname*{sub}%
\nolimits_{i_{1},i_{2},\ldots,i_{u}}^{g_{1},g_{2},\ldots,g_{u}}\widetilde{A}%
\right)  }_{=\operatorname*{sub}\nolimits_{i_{1},i_{2},\ldots,i_{u}}%
^{g_{1},g_{2},\ldots,g_{u}}A}\cdot\det\underbrace{\left(  \operatorname*{sub}%
\nolimits_{g_{1},g_{2},\ldots,g_{u}}^{j_{1},j_{2},\ldots,j_{u}}\widetilde{B}%
\right)  }_{=\operatorname*{sub}\nolimits_{g_{1},g_{2},\ldots,g_{u}}%
^{j_{1},j_{2},\ldots,j_{u}}B}\\
&  =\sum_{\substack{g_{1}>g_{2}>\cdots>g_{u};\\i_{k}\leq g_{k}\leq j_{k}\text{
for all }k\in\left[  u\right]  }}\det\left(  \operatorname*{sub}%
\nolimits_{i_{1},i_{2},\ldots,i_{u}}^{g_{1},g_{2},\ldots,g_{u}}A\right)
\cdot\det\left(  \operatorname*{sub}\nolimits_{g_{1},g_{2},\ldots,g_{u}%
}^{j_{1},j_{2},\ldots,j_{u}}B\right)  .
\end{align*}

In order to prove (\ref{pf.lem.Stilde.coproduct.CB5}), it thus remains to show
that%
\begin{equation}
\operatorname*{sub}\nolimits_{i_{1},i_{2},\ldots,i_{u}}^{j_{1},j_{2}%
,\ldots,j_{u}}\left(  AB\right)  =\operatorname*{sub}\nolimits_{i_{1}%
,i_{2},\ldots,i_{u}}^{j_{1},j_{2},\ldots,j_{u}}\left(  \widetilde{A}%
\cdot\widetilde{B}\right)  . \label{pf.lem.Stilde.coproduct.CB5.7}%
\end{equation}
But this is easy: We have%
\[
\operatorname*{sub}\nolimits_{i_{1},i_{2},\ldots,i_{u}}^{j_{1},j_{2}%
,\ldots,j_{u}}\left(  \widetilde{AB}\right)  =\operatorname*{sub}%
\nolimits_{i_{1},i_{2},\ldots,i_{u}}^{j_{1},j_{2},\ldots,j_{u}}\left(
AB\right)
\]
(since $\widetilde{AB}$ is a submatrix of $AB$ that preserves the same
indexing as $AB$: i.e., the $\left(  i,j\right)  $-th entry of $\widetilde{AB}%
$ equals the $\left(  i,j\right)  $-th entry of $AB$ whenever the former is
defined) and thus%
\[
\operatorname*{sub}\nolimits_{i_{1},i_{2},\ldots,i_{u}}^{j_{1},j_{2}%
,\ldots,j_{u}}\left(  AB\right)  =\operatorname*{sub}\nolimits_{i_{1}%
,i_{2},\ldots,i_{u}}^{j_{1},j_{2},\ldots,j_{u}}\underbrace{\left(
\widetilde{AB}\right)  }_{\substack{=\widetilde{A}\cdot\widetilde{B}%
\\\text{(by (\ref{pf.lem.Stilde.coproduct.CB5.1}))}}}=\operatorname*{sub}%
\nolimits_{i_{1},i_{2},\ldots,i_{u}}^{j_{1},j_{2},\ldots,j_{u}}\left(
\widetilde{A}\cdot\widetilde{B}\right)  .
\]
This proves (\ref{pf.lem.Stilde.coproduct.CB5.7}). Thus, our proof of
(\ref{pf.lem.Stilde.coproduct.CB5}) is complete.
\end{proof}

Let us also record a simple property of determinants:

\begin{lemma}
\label{lem.detscalesign}Let $u\in\mathbb{N}$. Let $\lambda=\left(  \lambda
_{1},\lambda_{2},\ldots,\lambda_{u}\right)  $ and $\nu=\left(  \nu_{1},\nu
_{2},\ldots,\nu_{u}\right)  $ be two partitions of length $\leq u$. Let
$\left(  c_{i,j}\right)  _{i,j\in\left[  u\right]  }\in R^{u\times u}$ be any
$u\times u$-matrix. Then,%
\[
\det\left(  \left(  \left(  -1\right)  ^{\lambda_{i}-\nu_{j}-i+j}%
c_{i,j}\right)  _{i,j\in\left[  u\right]  }\right)  =\left(  -1\right)
^{\left\vert \lambda\right\vert -\left\vert \nu\right\vert }\det\left(
\left(  c_{i,j}\right)  _{i,j\in\left[  u\right]  }\right)  .
\]

\end{lemma}

\begin{proof}
For any $i,j\in\left[  u\right]  $, we have
\begin{equation}
\left(  -1\right)  ^{\lambda_{i}-\nu_{j}-i+j}=\left(  -1\right)  ^{\left(
\lambda_{i}-i\right)  +\left(  j-\nu_{j}\right)  }=\left(  -1\right)
^{\lambda_{i}-i}\left(  -1\right)  ^{j-\nu_{j}}. \label{pf.lem.detscalesign.1}%
\end{equation}

But the matrix $\left(  \left(  -1\right)  ^{\lambda_{i}-i}\left(  -1\right)
^{j-\nu_{j}}c_{i,j}\right)  _{i,j\in\left[  u\right]  }$ is obviously obtained
from the matrix $\left(  c_{i,j}\right)  _{i,j\in\left[  u\right]  }$ by the
following operations:

\begin{enumerate}
\item Scale the $i$-th row by the factor $\left(  -1\right)  ^{\lambda_{i}-i}$
for each $i\in\left[  u\right]  $.

\item Scale the $j$-th column by the factor $\left(  -1\right)  ^{j-\nu_{j}}$
for each $j\in\left[  u\right]  $.
\end{enumerate}

The effect of these operations on the determinant of the matrix is that the
determinant gets multiplied by $\left(  \prod_{i=1}^{u}\left(  -1\right)
^{\lambda_{i}-i}\right)  \left(  \prod_{j=1}^{u}\left(  -1\right)  ^{j-\nu
_{j}}\right)  $ (because when we scale a row or a column of a matrix by a
factor $\lambda$, the determinant of this matrix gets multiplied by $\lambda
$). Thus,%
\[
\det\left(  \left(  \left(  -1\right)  ^{\lambda_{i}-i}\left(  -1\right)
^{j-\nu_{j}}c_{i,j}\right)  _{i,j\in\left[  u\right]  }\right)  =\left(
\prod_{i=1}^{u}\left(  -1\right)  ^{\lambda_{i}-i}\right)  \left(  \prod
_{j=1}^{u}\left(  -1\right)  ^{j-\nu_{j}}\right)  \det\left(  \left(
c_{i,j}\right)  _{i,j\in\left[  u\right]  }\right)  .
\]
In view of%
\begin{align*}
&  \underbrace{\left(  \prod_{i=1}^{u}\left(  -1\right)  ^{\lambda_{i}%
-i}\right)  }_{\substack{=\left(  -1\right)  ^{\left(  \lambda_{1}-1\right)
+\left(  \lambda_{2}-2\right)  +\cdots+\left(  \lambda_{u}-u\right)
}\\=\left(  -1\right)  ^{\left(  \lambda_{1}+\lambda_{2}+\cdots+\lambda
_{u}\right)  -\left(  1+2+\cdots+u\right)  }}}\underbrace{\left(  \prod
_{j=1}^{u}\left(  -1\right)  ^{j-\nu_{j}}\right)  }_{\substack{=\left(
-1\right)  ^{\left(  1-\nu_{1}\right)  +\left(  2-\nu_{2}\right)
+\cdots+\left(  u-\nu_{u}\right)  }\\=\left(  -1\right)  ^{\left(
1+2+\cdots+u\right)  -\left(  \nu_{1}+\nu_{2}+\cdots+\nu_{u}\right)  }}}\\
&  =\left(  -1\right)  ^{\left(  \lambda_{1}+\lambda_{2}+\cdots+\lambda
_{u}\right)  -\left(  1+2+\cdots+u\right)  }\left(  -1\right)  ^{\left(
1+2+\cdots+u\right)  -\left(  \nu_{1}+\nu_{2}+\cdots+\nu_{u}\right)  }\\
&  =\left(  -1\right)  ^{\left(  \left(  \lambda_{1}+\lambda_{2}%
+\cdots+\lambda_{u}\right)  -\left(  1+2+\cdots+u\right)  \right)  +\left(
\left(  1+2+\cdots+u\right)  -\left(  \nu_{1}+\nu_{2}+\cdots+\nu_{u}\right)
\right)  }\\
&  =\left(  -1\right)  ^{\left(  \lambda_{1}+\lambda_{2}+\cdots+\lambda
_{u}\right)  -\left(  \nu_{1}+\nu_{2}+\cdots+\nu_{u}\right)  }=\left(
-1\right)  ^{\left\vert \lambda\right\vert -\left\vert \nu\right\vert }%
\end{align*}
(because $\lambda_{1}+\lambda_{2}+\cdots+\lambda_{u}=\left\vert \lambda
\right\vert $ and $\nu_{1}+\nu_{2}+\cdots+\nu_{u}=\left\vert \nu\right\vert
$), we can rewrite this as%
\[
\det\left(  \left(  \left(  -1\right)  ^{\lambda_{i}-i}\left(  -1\right)
^{j-\nu_{j}}c_{i,j}\right)  _{i,j\in\left[  u\right]  }\right)  =\left(
-1\right)  ^{\left\vert \lambda\right\vert -\left\vert \nu\right\vert }%
\det\left(  \left(  c_{i,j}\right)  _{i,j\in\left[  u\right]  }\right)  .
\]
Using (\ref{pf.lem.detscalesign.1}), we can furthermore rewrite this as%
\[
\det\left(  \left(  \left(  -1\right)  ^{\lambda_{i}-\nu_{j}-i+j}%
c_{i,j}\right)  _{i,j\in\left[  u\right]  }\right)  =\left(  -1\right)
^{\left\vert \lambda\right\vert -\left\vert \nu\right\vert }\det\left(
\left(  c_{i,j}\right)  _{i,j\in\left[  u\right]  }\right)  .
\]
Thus, Lemma \ref{lem.detscalesign} is proved.
\end{proof}

We can now finish the proof of Lemma \ref{lem.Stilde.coproduct}:

\begin{proof}
[Details for the proof of Lemma \ref{lem.Stilde.coproduct}.]We must derive
(\ref{eq.Stilde.coproduct}) from (\ref{pf.lem.Stilde.coproduct.HHE}).

Pick $u\in\mathbb{N}$ such that both $\ell\left(  \lambda\right)  $ and
$\ell\left(  \mu\right)  $ are $\leq u$. Thus, $\lambda=\left(  \lambda
_{1},\lambda_{2},\ldots,\lambda_{u}\right)  $ and $\mu=\left(  \mu_{1},\mu
_{2},\ldots,\mu_{u}\right)  $.

Let $R$ be the commutative ring $\mathbf{A}$. Clearly, the totally ordered set
$\mathbb{Z}$ is interval-finite. Recall that both matrices $\mathbf{H}\left(
V\right)  $ and $\mathbf{E}\left(  U\right)  $ are upper-triangular matrices
in $\mathbf{A}^{\mathbb{Z}\times\mathbb{Z}}=R^{\mathbb{Z}\times\mathbb{Z}}$.
Hence, the matrix $\varphi^{\dim\left(  V/U\right)  }\left(  \mathbf{E}\left(
U\right)  \right)  $ is upper-triangular as well (since $\varphi^{\dim\left(
V/U\right)  }$ is a ring endomorphism). Moreover, from $\mu_{1}\geq\mu_{2}%
\geq\cdots\geq\mu_{u}$ (since $\mu$ is a partition), we obtain $\mu_{1}%
-1>\mu_{2}-2>\cdots>\mu_{u}-u$. Similarly, $\lambda_{1}-1>\lambda_{2}%
-2>\cdots>\lambda_{u}-u$.

Thus, we can apply (\ref{pf.lem.Stilde.coproduct.CB5}) to $P=\mathbb{Z}$ and
$A=\mathbf{H}\left(  V\right)  $ and $B=\varphi^{\dim\left(  V/U\right)
}\left(  \mathbf{E}\left(  U\right)  \right)  $ and $i_{x}=\mu_{x}-x$ and
$j_{y}=\lambda_{y}-y$. As a result, we obtain%
\begin{align}
&  \det\left(  \operatorname*{sub}\nolimits_{\mu_{1}-1,\mu_{2}-2,\ldots
,\mu_{u}-u}^{\lambda_{1}-1,\lambda_{2}-2,\ldots,\lambda_{u}-u}\left(
\mathbf{H}\left(  V\right)  \cdot\varphi^{\dim\left(  V/U\right)  }\left(
\mathbf{E}\left(  U\right)  \right)  \right)  \right) \nonumber\\
&  =\sum_{\substack{g_{1}>g_{2}>\cdots>g_{u};\\\mu_{k}-k\leq g_{k}\leq
\lambda_{k}-k\text{ for all }k\in\left[  u\right]  }}\det\left(
\operatorname*{sub}\nolimits_{\mu_{1}-1,\mu_{2}-2,\ldots,\mu_{u}-u}%
^{g_{1},g_{2},\ldots,g_{u}}\left(  \mathbf{H}\left(  V\right)  \right)
\right) \nonumber\\
&  \ \ \ \ \ \ \ \ \ \ \cdot\det\left(  \operatorname*{sub}\nolimits_{g_{1}%
,g_{2},\ldots,g_{u}}^{\lambda_{1}-1,\lambda_{2}-2,\ldots,\lambda_{u}-u}\left(
\varphi^{\dim\left(  V/U\right)  }\left(  \mathbf{E}\left(  U\right)  \right)
\right)  \right) \nonumber\\
&  =\sum_{\substack{\nu_{1}\geq\nu_{2}\geq\cdots\geq\nu_{u};\\\mu_{k}-k\leq
\nu_{k}-k\leq\lambda_{k}-k\text{ for all }k\in\left[  u\right]  }}\det\left(
\operatorname*{sub}\nolimits_{\mu_{1}-1,\mu_{2}-2,\ldots,\mu_{u}-u}^{\nu
_{1}-1,\nu_{2}-2,\ldots,\nu_{u}-u}\left(  \mathbf{H}\left(  V\right)  \right)
\right) \nonumber\\
&  \ \ \ \ \ \ \ \ \ \ \cdot\det\left(  \operatorname*{sub}\nolimits_{\nu
_{1}-1,\nu_{2}-2,\ldots,\nu_{u}-u}^{\lambda_{1}-1,\lambda_{2}-2,\ldots
,\lambda_{u}-u}\left(  \varphi^{\dim\left(  V/U\right)  }\left(
\mathbf{E}\left(  U\right)  \right)  \right)  \right)
\label{pf.lem.Stilde.coproduct.5}%
\end{align}
(here, we have substituted $\nu_{k}-k$ for $g_{k}$ in the sum, noting that
this substitution transforms the chain of inequalities $g_{1}>g_{2}%
>\cdots>g_{u}$ into $\nu_{1}-1>\nu_{2}-2>\cdots>\nu_{u}-u$, which is
equivalent to $\nu_{1}\geq\nu_{2}\geq\cdots\geq\nu_{u}$ because $\nu_{1}%
,\nu_{2},\ldots,\nu_{u}$ are integers). Of course, the condition
\textquotedblleft$\mu_{k}-k\leq\nu_{k}-k\leq\lambda_{k}-k$ for all
$k\in\left[  u\right]  $\textquotedblright\ under the summation sign in
(\ref{pf.lem.Stilde.coproduct.5}) is equivalent to \textquotedblleft$\mu
_{k}\leq\nu_{k}\leq\lambda_{k}$ for all $k\in\left[  u\right]  $%
\textquotedblright; thus the summation sign can be rewritten as follows:%
\[
\sum_{\substack{\nu_{1}\geq\nu_{2}\geq\cdots\geq\nu_{u};\\\mu_{k}-k\leq\nu
_{k}-k\leq\lambda_{k}-k\text{ for all }k\in\left[  u\right]  }}=\sum
_{\substack{\nu_{1}\geq\nu_{2}\geq\cdots\geq\nu_{u};\\\mu_{k}\leq\nu_{k}%
\leq\lambda_{k}\text{ for all }k\in\left[  u\right]  }}=\sum
_{\substack{\left(  \nu_{1},\nu_{2},\ldots,\nu_{u}\right)  \\\text{is a
partition of length }\leq u;\\\mu_{k}\leq\nu_{k}\leq\lambda_{k}\text{ for all
}k\in\left[  u\right]  }}
\]
(because the condition \textquotedblleft$\mu_{k}\leq\nu_{k}\leq\lambda_{k}$
for all $k\in\left[  u\right]  $\textquotedblright\ forces all $\nu_{k}$ to be
nonnegative\footnote{since it entails $\nu_{k}\geq\mu_{k}\geq0$ for each
$k\in\left[  u\right]  $}, and then the condition \textquotedblleft$\nu
_{1}\geq\nu_{2}\geq\cdots\geq\nu_{u}$\textquotedblright\ is simply saying that
$\left(  \nu_{1},\nu_{2},\ldots,\nu_{u}\right)  $ is a partition of length
$\leq u$). Thus, we can rewrite (\ref{pf.lem.Stilde.coproduct.5}) as%
\begin{align*}
&  \det\left(  \operatorname*{sub}\nolimits_{\mu_{1}-1,\mu_{2}-2,\ldots
,\mu_{u}-u}^{\lambda_{1}-1,\lambda_{2}-2,\ldots,\lambda_{u}-u}\left(
\mathbf{H}\left(  V\right)  \cdot\varphi^{\dim\left(  V/U\right)  }\left(
\mathbf{E}\left(  U\right)  \right)  \right)  \right) \\
&  =\sum_{\substack{\left(  \nu_{1},\nu_{2},\ldots,\nu_{u}\right)  \\\text{is
a partition of length }\leq u;\\\mu_{k}\leq\nu_{k}\leq\lambda_{k}\text{ for
all }k\in\left[  u\right]  }}\det\left(  \operatorname*{sub}\nolimits_{\mu
_{1}-1,\mu_{2}-2,\ldots,\mu_{u}-u}^{\nu_{1}-1,\nu_{2}-2,\ldots,\nu_{u}%
-u}\left(  \mathbf{H}\left(  V\right)  \right)  \right) \\
&  \ \ \ \ \ \ \ \ \ \ \cdot\det\left(  \operatorname*{sub}\nolimits_{\nu
_{1}-1,\nu_{2}-2,\ldots,\nu_{u}-u}^{\lambda_{1}-1,\lambda_{2}-2,\ldots
,\lambda_{u}-u}\left(  \varphi^{\dim\left(  V/U\right)  }\left(
\mathbf{E}\left(  U\right)  \right)  \right)  \right)  .
\end{align*}
In view of (\ref{pf.lem.Stilde.coproduct.HHE}), we can furthermore rewrite
this as%
\begin{align}
&  \det\left(  \operatorname*{sub}\nolimits_{\mu_{1}-1,\mu_{2}-2,\ldots
,\mu_{u}-u}^{\lambda_{1}-1,\lambda_{2}-2,\ldots,\lambda_{u}-u}\left(
\mathbf{H}\left(  V\sslash U\right)  \right)  \right) \nonumber\\
&  =\underbrace{\sum_{\substack{\left(  \nu_{1},\nu_{2},\ldots,\nu_{u}\right)
\\\text{is a partition of length }\leq u;\\\mu_{k}\leq\nu_{k}\leq\lambda
_{k}\text{ for all }k\in\left[  u\right]  }}}_{\substack{=\sum_{\substack{\nu
=\left(  \nu_{1},\nu_{2},\ldots,\nu_{u}\right)  \\\text{is a partition of
length }\leq u;\\\mu\subseteq\nu\subseteq\lambda}}\\\text{(since the condition
\textquotedblleft}\mu_{k}\leq\nu_{k}\leq\lambda_{k}\text{ for all }k\in\left[
u\right]  \text{\textquotedblright}\\\text{is equivalent to \textquotedblleft%
}\mu\subseteq\nu\subseteq\lambda\text{\textquotedblright\ when }\nu=\left(
\nu_{1},\nu_{2},\ldots,\nu_{u}\right)  \text{)}}}\det\left(
\operatorname*{sub}\nolimits_{\mu_{1}-1,\mu_{2}-2,\ldots,\mu_{u}-u}^{\nu
_{1}-1,\nu_{2}-2,\ldots,\nu_{u}-u}\left(  \mathbf{H}\left(  V\right)  \right)
\right) \nonumber\\
&  \ \ \ \ \ \ \ \ \ \ \cdot\underbrace{\det\left(  \operatorname*{sub}%
\nolimits_{\nu_{1}-1,\nu_{2}-2,\ldots,\nu_{u}-u}^{\lambda_{1}-1,\lambda
_{2}-2,\ldots,\lambda_{u}-u}\left(  \varphi^{\dim\left(  V/U\right)  }\left(
\mathbf{E}\left(  U\right)  \right)  \right)  \right)  }_{\substack{=\varphi
^{\dim\left(  V/U\right)  }\left(  \det\left(  \operatorname*{sub}%
\nolimits_{\nu_{1}-1,\nu_{2}-2,\ldots,\nu_{u}-u}^{\lambda_{1}-1,\lambda
_{2}-2,\ldots,\lambda_{u}-u}\left(  \mathbf{E}\left(  U\right)  \right)
\right)  \right)  \\\text{(since }\varphi^{\dim\left(  V/U\right)  }\text{ is
a ring morphism, and thus commutes}\\\text{with taking determinants and
submatrices)}}}\nonumber\\
&  =\sum_{\substack{\nu=\left(  \nu_{1},\nu_{2},\ldots,\nu_{u}\right)
\\\text{is a partition of length }\leq u;\\\mu\subseteq\nu\subseteq\lambda
}}\det\left(  \operatorname*{sub}\nolimits_{\mu_{1}-1,\mu_{2}-2,\ldots,\mu
_{u}-u}^{\nu_{1}-1,\nu_{2}-2,\ldots,\nu_{u}-u}\left(  \mathbf{H}\left(
V\right)  \right)  \right) \nonumber\\
&  \ \ \ \ \ \ \ \ \ \ \cdot\varphi^{\dim\left(  V/U\right)  }\left(
\det\left(  \operatorname*{sub}\nolimits_{\nu_{1}-1,\nu_{2}-2,\ldots,\nu
_{u}-u}^{\lambda_{1}-1,\lambda_{2}-2,\ldots,\lambda_{u}-u}\left(
\mathbf{E}\left(  U\right)  \right)  \right)  \right)  .
\label{pf.lem.Stilde.coproduct.7}%
\end{align}

Furthermore, if $\nu=\left(  \nu_{1},\nu_{2},\ldots,\nu_{u}\right)  $ is any
partition of length $\leq u$, then%
\begin{align*}
\operatorname*{sub}\nolimits_{\mu_{1}-1,\mu_{2}-2,\ldots,\mu_{u}-u}^{\nu
_{1}-1,\nu_{2}-2,\ldots,\nu_{u}-u}\left(  \mathbf{H}\left(  V\right)  \right)
&  =\left(  \varphi^{\mu_{x}-x+1}H_{\left(  \nu_{y}-y\right)  -\left(  \mu
_{x}-x\right)  }\left(  V\right)  \right)  _{x,y\in\left[  u\right]  }\\
&  \ \ \ \ \ \ \ \ \ \ \ \ \ \ \ \ \ \ \ \ \left(  \text{since }%
\mathbf{H}\left(  V\right)  =\left(  \varphi^{i+1}H_{j-i}\left(  V\right)
\right)  _{i,j\in\mathbb{Z}}\right) \\
&  =\left(  \varphi^{\mu_{x}-x+1}H_{\nu_{y}-\mu_{x}-y+x}\left(  V\right)
\right)  _{x,y\in\left[  u\right]  }%
\end{align*}
and thus%
\begin{align}
&  \det\left(  \operatorname*{sub}\nolimits_{\mu_{1}-1,\mu_{2}-2,\ldots
,\mu_{u}-u}^{\nu_{1}-1,\nu_{2}-2,\ldots,\nu_{u}-u}\left(  \mathbf{H}\left(
V\right)  \right)  \right) \nonumber\\
&  =\det\left(  \left(  \varphi^{\mu_{x}-x+1}H_{\nu_{y}-\mu_{x}-y+x}\left(
V\right)  \right)  _{x,y\in\left[  u\right]  }\right) \nonumber\\
&  =\det\left(  \left(  \varphi^{\mu_{j}-j+1}H_{\nu_{i}-\mu_{j}-i+j}\left(
V\right)  \right)  _{j,i\in\left[  u\right]  }\right) \nonumber\\
&  =\det\left(  \left(  \varphi^{\mu_{j}-j+1}H_{\nu_{i}-\mu_{j}-i+j}\left(
V\right)  \right)  _{i,j\in\left[  u\right]  }\right) \nonumber\\
&  \ \ \ \ \ \ \ \ \ \ \ \ \ \ \ \ \ \ \ \ \left(
\begin{array}
[c]{c}%
\text{since the determinant of a matrix}\\
\text{equals the determinant of its transpose}%
\end{array}
\right) \nonumber\\
&  =S_{\nu/\mu}\left(  V\right)  \ \ \ \ \ \ \ \ \ \ \left(  \text{by
(\ref{eq.Slammu})}\right)  . \label{pf.lem.Stilde.coproduct.8a}%
\end{align}
The same argument (but with $V$ and $\nu$ replaced by $V\sslash U$ and
$\lambda$) shows that%
\begin{align}
&  \det\left(  \operatorname*{sub}\nolimits_{\mu_{1}-1,\mu_{2}-2,\ldots
,\mu_{u}-u}^{\lambda_{1}-1,\lambda_{2}-2,\ldots,\lambda_{u}-u}\left(
\mathbf{H}\left(  V\sslash U\right)  \right)  \right) \nonumber\\
&  =S_{\lambda/\mu}\left(  V\sslash U\right)  .
\label{pf.lem.Stilde.coproduct.8b}%
\end{align}

Furthermore, if $\nu=\left(  \nu_{1},\nu_{2},\ldots,\nu_{u}\right)  $ is any
partition of length $\leq u$, then%
\begin{align*}
&  \operatorname*{sub}\nolimits_{\nu_{1}-1,\nu_{2}-2,\ldots,\nu_{u}%
-u}^{\lambda_{1}-1,\lambda_{2}-2,\ldots,\lambda_{u}-u}\left(  \mathbf{E}%
\left(  U\right)  \right) \\
&  =\left(  \left(  -1\right)  ^{\left(  \lambda_{y}-y\right)  -\left(
\nu_{x}-x\right)  }\varphi^{\lambda_{y}-y}E_{\left(  \lambda_{y}-y\right)
-\left(  \nu_{x}-x\right)  }\left(  U\right)  \right)  _{x,y\in\left[
u\right]  }\\
&  \ \ \ \ \ \ \ \ \ \ \ \ \ \ \ \ \ \ \ \ \left(  \text{since }%
\mathbf{E}\left(  U\right)  =\left(  \left(  -1\right)  ^{j-i}\varphi
^{j}E_{j-i}\left(  U\right)  \right)  _{i,j\in\mathbb{Z}}\right) \\
&  =\left(  \left(  -1\right)  ^{\lambda_{y}-\nu_{x}-y+x}\varphi^{\lambda
_{y}-y}E_{\lambda_{y}-\nu_{x}-y+x}\left(  U\right)  \right)  _{x,y\in\left[
u\right]  }%
\end{align*}
and thus%
\begin{align}
&  \det\left(  \operatorname*{sub}\nolimits_{\nu_{1}-1,\nu_{2}-2,\ldots
,\nu_{u}-u}^{\lambda_{1}-1,\lambda_{2}-2,\ldots,\lambda_{u}-u}\left(
\mathbf{E}\left(  U\right)  \right)  \right) \nonumber\\
&  =\det\left(  \left(  \left(  -1\right)  ^{\lambda_{y}-\nu_{x}-y+x}%
\varphi^{\lambda_{y}-y}E_{\lambda_{y}-\nu_{x}-y+x}\left(  U\right)  \right)
_{x,y\in\left[  u\right]  }\right) \nonumber\\
&  =\det\left(  \left(  \left(  -1\right)  ^{\lambda_{i}-\nu_{j}-i+j}%
\varphi^{\lambda_{i}-i}E_{\lambda_{i}-\nu_{j}-i+j}\left(  U\right)  \right)
_{j,i\in\left[  u\right]  }\right) \nonumber\\
&  =\det\left(  \left(  \left(  -1\right)  ^{\lambda_{i}-\nu_{j}-i+j}%
\varphi^{\lambda_{i}-i}E_{\lambda_{i}-\nu_{j}-i+j}\left(  U\right)  \right)
_{i,j\in\left[  u\right]  }\right) \nonumber\\
&  \ \ \ \ \ \ \ \ \ \ \ \ \ \ \ \ \ \ \ \ \left(
\begin{array}
[c]{c}%
\text{since the determinant of a matrix}\\
\text{equals the determinant of its transpose}%
\end{array}
\right) \nonumber\\
&  =\left(  -1\right)  ^{\left\vert \lambda\right\vert -\left\vert
\nu\right\vert }\underbrace{\det\left(  \left(  \varphi^{\lambda_{i}%
-i}E_{\lambda_{i}-\nu_{j}-i+j}\left(  U\right)  \right)  _{i,j\in\left[
u\right]  }\right)  }_{\substack{=\widetilde{S}_{\lambda/\nu}\left(  U\right)
\\\text{(by (\ref{eq.Stil=}))}}}\nonumber\\
&  \ \ \ \ \ \ \ \ \ \ \ \ \ \ \ \ \ \ \ \ \left(  \text{by Lemma
\ref{lem.detscalesign}, applied to }c_{i,j}=\varphi^{\lambda_{i}-i}%
E_{\lambda_{i}-\nu_{j}-i+j}\left(  U\right)  \right) \nonumber\\
&  =\left(  -1\right)  ^{\left\vert \lambda\right\vert -\left\vert
\nu\right\vert }\widetilde{S}_{\lambda/\nu}\left(  U\right)  .
\label{pf.lem.Stilde.coproduct.8c}%
\end{align}

Using (\ref{pf.lem.Stilde.coproduct.8a}), (\ref{pf.lem.Stilde.coproduct.8b})
and (\ref{pf.lem.Stilde.coproduct.8c}), we can rewrite the equality
(\ref{pf.lem.Stilde.coproduct.7}) as%
\begin{align}
S_{\lambda/\mu}\left(  V\sslash U\right)   &  =\sum_{\substack{\nu=\left(
\nu_{1},\nu_{2},\ldots,\nu_{u}\right)  \\\text{is a partition of length }\leq
u;\\\mu\subseteq\nu\subseteq\lambda}}S_{\nu/\mu}\left(  V\right)
\cdot\underbrace{\varphi^{\dim\left(  V/U\right)  }\left(  \left(  -1\right)
^{\left\vert \lambda\right\vert -\left\vert \nu\right\vert }\widetilde{S}%
_{\lambda/\nu}\left(  U\right)  \right)  }_{\substack{=\left(  -1\right)
^{\left\vert \lambda\right\vert -\left\vert \nu\right\vert }\varphi
^{\dim\left(  V/U\right)  }\widetilde{S}_{\lambda/\nu}\left(  U\right)
\\\text{(since }\varphi^{\dim\left(  V/U\right)  }\text{ is a ring morphism)}%
}}\nonumber\\
&  =\sum_{\substack{\nu=\left(  \nu_{1},\nu_{2},\ldots,\nu_{u}\right)
\\\text{is a partition of length }\leq u;\\\mu\subseteq\nu\subseteq\lambda
}}\left(  -1\right)  ^{\left\vert \lambda\right\vert -\left\vert
\nu\right\vert }S_{\nu/\mu}\left(  V\right)  \cdot\varphi^{\dim\left(
V/U\right)  }\widetilde{S}_{\lambda/\nu}\left(  U\right) \nonumber\\
&  =\sum_{\substack{\nu\text{ is a partition;}\\\mu\subseteq\nu\subseteq
\lambda}}\left(  -1\right)  ^{\left\vert \lambda\right\vert -\left\vert
\nu\right\vert }S_{\nu/\mu}\left(  V\right)  \cdot\varphi^{\dim\left(
V/U\right)  }\widetilde{S}_{\lambda/\nu}\left(  U\right)
\label{pf.lem.Stilde.coproduct.9}%
\end{align}
(here, we have removed the condition \textquotedblleft$\nu$ has length $\leq
u$\textquotedblright\ from the summation sign, since this condition is
automatically implied by the condition \textquotedblleft$\nu\subseteq\lambda
$\textquotedblright). This is almost the desired formula
(\ref{eq.Stilde.coproduct}). The only difference is that the right hand side
of (\ref{eq.Stilde.coproduct}) contains more addends than the right hand side
of (\ref{pf.lem.Stilde.coproduct.9}), since the partition $\nu$ is not
required to satisfy $\mu\subseteq\nu\subseteq\lambda$ in
(\ref{eq.Stilde.coproduct}). However, this difference is immaterial: If a
partition $\nu$ does not satisfy $\mu\subseteq\nu\subseteq\lambda$, then

\begin{itemize}
\item it either fails to satisfy $\mu\subseteq\nu$, in which case we have
$S_{\nu/\mu}\left(  V\right)  =0$ by (\ref{eq.Slammu=0ifnotsub}) and therefore%
\[
\left(  -1\right)  ^{\left\vert \lambda\right\vert -\left\vert \nu\right\vert
}\underbrace{S_{\nu/\mu}\left(  V\right)  }_{=0}\cdot\,\varphi^{\dim\left(
V/U\right)  }\widetilde{S}_{\lambda/\nu}\left(  U\right)  =0;
\]

\item or it fails to satisfy $\nu\subseteq\lambda$, in which case we have
$\widetilde{S}_{\lambda/\nu}\left(  U\right)  =0$ by
(\ref{eq.Stilde.0ifnotsub}) and therefore%
\[
\left(  -1\right)  ^{\left\vert \lambda\right\vert -\left\vert \nu\right\vert
}S_{\nu/\mu}\left(  V\right)  \cdot\varphi^{\dim\left(  V/U\right)
}\underbrace{\widetilde{S}_{\lambda/\nu}\left(  U\right)  }_{=0}=0.
\]

\end{itemize}

\medskip In both cases, we obtain $\left(  -1\right)  ^{\left\vert
\lambda\right\vert -\left\vert \nu\right\vert }S_{\nu/\mu}\left(  V\right)
\cdot\varphi^{\dim\left(  V/U\right)  }\widetilde{S}_{\lambda/\nu}\left(
U\right)  =0$. Thus, all addends on the right hand side of
(\ref{eq.Stilde.coproduct}) that don't appear on the right hand side of
(\ref{pf.lem.Stilde.coproduct.9}) are $0$, and therefore do not affect the
sum. Consequently, the two right hand sides are equal. Thus,
(\ref{eq.Stilde.coproduct}) follows from (\ref{pf.lem.Stilde.coproduct.9}), so
that the proof of Lemma \ref{lem.Stilde.coproduct} is complete.
\end{proof}

\subsection{\label{sec.apx-orig-7.25}Appendix: Proof of Theorem \ref{thm.7.25}%
}

To prove Theorem \ref{thm.7.25}, we need two simple functoriality lemmas. The
first is a functoriality for the $S_{\lambda}$:

\begin{lemma}
\label{lem.funcS}Let $\mathbf{A}$ and $\mathbf{B}$ be two commutative
$F$-algebras. Let $\psi:\mathbf{B}\rightarrow\mathbf{A}$ be an $F$-algebra
morphism, and let $V$ be a finite-dimensional $F$-vector subspace of
$\mathbf{B}$. Assume that the restriction $\psi\mid_{V}$ is injective. Let
$\lambda$ be a partition. Then,%
\[
\psi\left(  S_{\lambda}\left(  V\right)  \right)  =S_{\lambda}\left(
\psi\left(  V\right)  \right)  .
\]

\end{lemma}

\begin{proof}
Pick a basis $\left(  v_{1},v_{2},\ldots,v_{n}\right)  $ of the $F$-vector
space $V$. Then, the $n$-tuple $\left(  \psi\left(  v_{1}\right)  ,\psi\left(
v_{2}\right)  ,\ldots,\psi\left(  v_{n}\right)  \right)  $ is a basis of
$\psi\left(  V\right)  $ (since the restriction $\psi\mid_{V}$ is injective).
Hence, $\dim\left(  \psi\left(  V\right)  \right)  =n=\dim V$.

We are in one of the following two cases:

\textit{Case 1:} We have $\ell\left(  \lambda\right)  >n$.

\textit{Case 2:} We have $\ell\left(  \lambda\right)  \leq n$.

Consider Case 1 first. In this case, $\ell\left(  \lambda\right)  >n$. Hence,
$\lambda$ has length $\ell\left(  \lambda\right)  >n=\dim V$. Thus,
(\ref{eq.Slam.SlamV0}) shows that $S_{\lambda}\left(  V\right)  =0$. But
$\lambda$ also has length $\ell\left(  \lambda\right)  >n=\dim\left(
\psi\left(  V\right)  \right)  $, and therefore (\ref{eq.Slam.SlamV0}) shows
that $S_{\lambda}\left(  \psi\left(  V\right)  \right)  =0$. Hence, Lemma
\ref{lem.funcS} holds in Case 1 (since both $S_{\lambda}\left(  V\right)  $
and $S_{\lambda}\left(  \psi\left(  V\right)  \right)  $ are $0$).

Let us now consider Case 2. In this case, $\ell\left(  \lambda\right)  \leq
n$. Now, (\ref{eq.Slam.SlamV1}) yields $S_{\lambda}\left(  V\right)
=S_{\lambda}\left(  v_{1},v_{2},\ldots,v_{n}\right)  $, and similarly%
\[
S_{\lambda}\left(  \psi\left(  V\right)  \right)  =S_{\lambda}\left(
\psi\left(  v_{1}\right)  ,\psi\left(  v_{2}\right)  ,\ldots,\psi\left(
v_{n}\right)  \right)
\]
(since $\left(  \psi\left(  v_{1}\right)  ,\psi\left(  v_{2}\right)
,\ldots,\psi\left(  v_{n}\right)  \right)  $ is a basis of $\psi\left(
V\right)  $).

However, $S_{\lambda}\left(  v_{1},v_{2},\ldots,v_{n}\right)  $ is defined by
substituting $v_{1},v_{2},\ldots,v_{n}$ for $x_{1},x_{2},\ldots,x_{n}$ into a
certain polynomial $A_{\lambda+\delta}/A_{\delta}\in F\left[  x_{1}%
,x_{2},\ldots,x_{n}\right]  $; likewise, \newline$S_{\lambda}\left(
\psi\left(  v_{1}\right)  ,\psi\left(  v_{2}\right)  ,\ldots,\psi\left(
v_{n}\right)  \right)  $ is defined by substituting $\psi\left(  v_{1}\right)
,\psi\left(  v_{2}\right)  ,\ldots,\psi\left(  v_{n}\right)  $ into the same
polynomial. Hence,
\[
S_{\lambda}\left(  \psi\left(  v_{1}\right)  ,\psi\left(  v_{2}\right)
,\ldots,\psi\left(  v_{n}\right)  \right)  =\psi\left(  S_{\lambda}\left(
v_{1},v_{2},\ldots,v_{n}\right)  \right)
\]
(since $\psi$ is an $F$-algebra morphism and thus commutes with polynomials).
In other words, $S_{\lambda}\left(  \psi\left(  V\right)  \right)
=\psi\left(  S_{\lambda}\left(  V\right)  \right)  $ (since $S_{\lambda
}\left(  \psi\left(  V\right)  \right)  =S_{\lambda}\left(  \psi\left(
v_{1}\right)  ,\psi\left(  v_{2}\right)  ,\ldots,\psi\left(  v_{n}\right)
\right)  $ and $S_{\lambda}\left(  V\right)  =S_{\lambda}\left(  v_{1}%
,v_{2},\ldots,v_{n}\right)  $). Thus, Lemma \ref{lem.funcS} is proved in Case 2.

We have now proved Lemma \ref{lem.funcS} in both Cases 1 and 2. Hence, Lemma
\ref{lem.funcS} always holds.
\end{proof}

Our next functoriality lemma says that internal quotients are functorial with
respect to $F$-algebra morphisms that are injective on the relevant subspaces:

\begin{lemma}
\label{lem.func//}Let $\mathbf{A}$ and $\mathbf{B}$ be two commutative
$F$-algebras that are integral domains. Let $\psi:\mathbf{B}\rightarrow
\mathbf{A}$ be an $F$-algebra morphism, and let $U\subseteq W$ be two
finite-dimensional $F$-vector subspaces of $\mathbf{B}$. Assume that the
restriction $\psi\mid_{W}$ is injective. Then, $\psi\mid_{W\sslash U}$ is an
$F$-vector space isomorphism from $W\sslash U$ to $\psi\left(  W\right)
\sslash\psi\left(  U\right)  $.
\end{lemma}

\begin{proof}
We assumed that the restriction $\psi\mid_{W}$ is injective. Hence, the
restriction $\psi\mid_{U}$ is injective as well (since $U\subseteq W$). Thus,
the map $U\rightarrow\psi\left(  U\right)  ,\ u\mapsto\psi\left(  u\right)  $
is a bijection.

By definition of internal quotients, we have
\begin{align}
W\sslash U  &  =\widetilde{f}_{U}\left(  W\right)  =\left\{  \widetilde{f}%
_{U}\left(  w\right)  \ \mid\ w\in W\right\} \nonumber\\
&  =\left\{  \prod_{u\in U}\left(  w+u\right)  \ \mid\ w\in W\right\}
\label{pf.lem.func//.1}%
\end{align}
(since each $w\in W$ satisfies $\widetilde{f}_{U}\left(  w\right)
=f_{U}\left(  w\right)  =\prod_{u\in U}\left(  w+u\right)  $ by the definition
of $f_{U}$) and similarly%
\begin{equation}
\psi\left(  W\right)  \sslash \psi\left(  U\right)  =\left\{  \prod_{u\in
\psi\left(  U\right)  }\left(  w+u\right)  \ \mid\ w\in\psi\left(  W\right)
\right\}  . \label{pf.lem.func//.2}%
\end{equation}
Applying the map $\psi$ to both sides of (\ref{pf.lem.func//.1}), we find%
\begin{align}
\psi\left(  W\sslash U\right)   &  =\psi\left(  \left\{  \prod_{u\in U}\left(
w+u\right)  \ \mid\ w\in W\right\}  \right) \nonumber\\
&  =\left\{  \psi\left(  \prod_{u\in U}\left(  w+u\right)  \right)
\ \mid\ w\in W\right\}  . \label{pf.lem.func//.3}%
\end{align}

However, each $w\in W$ satisfies
\begin{align*}
&  \psi\left(  \prod_{u\in U}\left(  w+u\right)  \right) \\
&  =\prod_{u\in U}\left(  \psi\left(  w\right)  +\psi\left(  u\right)
\right)  \ \ \ \ \ \ \ \ \ \ \left(  \text{since }\psi\text{ is an
}F\text{-algebra morphism}\right) \\
&  =\prod_{u\in\psi\left(  U\right)  }\left(  \psi\left(  w\right)  +u\right)
\end{align*}
(here, we have substituted $u$ for $\psi\left(  u\right)  $ in the product,
since the map $U\rightarrow\psi\left(  U\right)  ,\ u\mapsto\psi\left(
u\right)  $ is a bijection). Thus, we can rewrite (\ref{pf.lem.func//.3}) as%
\begin{align*}
\psi\left(  W\sslash U\right)   &  =\left\{  \prod_{u\in\psi\left(  U\right)
}\left(  \psi\left(  w\right)  +u\right)  \ \mid\ w\in W\right\} \\
&  =\left\{  \prod_{u\in\psi\left(  U\right)  }\left(  w+u\right)
\ \mid\ w\in\psi\left(  W\right)  \right\}
\end{align*}
(here, we have substituted $w$ for $\psi\left(  w\right)  $ in the set, since
$\psi\left(  W\right)  $ is the set of all $\psi\left(  w\right)  $ with $w\in
W$). Comparing this with (\ref{pf.lem.func//.2}), we obtain%
\[
\psi\left(  W\sslash U\right)  =\psi\left(  W\right)  \sslash \psi\left(
U\right)  .
\]
Thus, the restriction $\psi\mid_{W\sslash U}$ is a well-defined and surjective
map from $W\sslash U$ to $\psi\left(  W\right)  \sslash \psi\left(  U\right)
$. Of course, this restriction is furthermore $F$-linear (since $\psi$ is
$F$-linear). It remains to prove that it is injective; then (combined with
$F$-linearity and surjectivity) it will automatically follow that $\psi
\mid_{W\sslash U}$ is an isomorphism.

So let us prove that $\psi\mid_{W\sslash U}$ is injective. Since $\psi
\mid_{W\sslash U}$ is $F$-linear, it suffices to show that
$\operatorname*{Ker}\left(  \psi\mid_{W\sslash U}\right)  =0$.

So let $r\in\operatorname*{Ker}\left(  \psi\mid_{W\sslash U}\right)  $. Thus,
$r\in W\sslash U$ and $\psi\left(  r\right)  =0$. Since $r\in
W\sslash U=\widetilde{f}_{U}\left(  W\right)  $, we can write $r$ as
$r=\widetilde{f}_{U}\left(  w\right)  $ for some $w\in W$. Consider this $w$.
Then,
\[
r=\widetilde{f}_{U}\left(  w\right)  =f_{U}\left(  w\right)  =\prod_{u\in
U}\left(  w+u\right)
\]
(by the definition of $f_{U}$), and thus%
\[
\psi\left(  r\right)  =\psi\left(  \prod_{u\in U}\left(  w+u\right)  \right)
=\prod_{u\in U}\left(  \psi\left(  w\right)  +\psi\left(  u\right)  \right)
\]
(since $\psi$ is an $F$-algebra morphism), so that $\prod_{u\in U}\left(
\psi\left(  w\right)  +\psi\left(  u\right)  \right)  =\psi\left(  r\right)
=0$. Since $\mathbf{A}$ is an integral domain, this entails that $\psi\left(
w\right)  +\psi\left(  u\right)  =0$ for some $u\in U$ (because a product in
an integral domain can only be $0$ if one of its factors is $0$). Consider
this $u$. Thus, $\psi\left(  w\right)  +\psi\left(  u\right)  =0$, so that
$\psi\left(  w\right)  =-\psi\left(  u\right)  =\psi\left(  -u\right)  $
(since $\psi$ is an $F$-algebra morphism). Since $\psi\mid_{W}$ is injective
(and since $w\in W$ and $-\underbrace{u}_{\in U}\in-U\subseteq U\subseteq W$),
we thus conclude that $w=-\underbrace{u}_{\in U}\in-U\subseteq U\subseteq
\operatorname*{Ker}\widetilde{f}_{U}$ (since we know that $\widetilde{f}_{U}$
always contains $U$ in its kernel). Thus, $\widetilde{f}_{U}\left(  w\right)
=0$, so that $r=\widetilde{f}_{U}\left(  w\right)  =0$.

Forget that we fixed $r$. We thus have shown that $r=0$ for each
$r\in\operatorname*{Ker}\left(  \psi\mid_{W\sslash U}\right)  $. Hence,
$\operatorname*{Ker}\left(  \psi\mid_{W\sslash U}\right)  =0$, so that
$\psi\mid_{W\sslash U}$ is injective. As we explained above, this completes
the proof of Lemma \ref{lem.func//}.
\end{proof}

We can now derive Theorem \ref{thm.7.25} from Theorem \ref{thm.skew-V/L}:

\begin{proof}
[Proof of Theorem \ref{thm.7.25}.]We cannot directly apply Theorem
\ref{thm.skew-V/L}, since we have not assumed that the Frobenius morphism
$\varphi:\mathbf{A}\rightarrow\mathbf{A}$ is invertible. Instead, we shall
apply Theorem \ref{thm.skew-V/L} to a new $F$-algebra $\widehat{\mathbf{B}}$
whose Frobenius morphism is invertible, and then transport the result to
$\mathbf{A}$ via an $F$-algebra morphism.

Here are the details: Let $n=\dim V$, and pick a basis $\left(  v_{1}%
,v_{2},\ldots,v_{n}\right)  $ of the $F$-vector space $V$. Let $\mathbf{B}$ be
the polynomial ring $F\left[  x_{1},x_{2},\ldots,x_{n}\right]  $ in $n$
indeterminates. By its universal property, there is an $F$-algebra morphism
$\psi:\mathbf{B}\rightarrow\mathbf{A}$ that sends each $x_{i}$ to $v_{i}$.
Consider this $\psi$. Let $W$ be the $F$-vector subspace of $\mathbf{B}$
spanned by $x_{1},x_{2},\ldots,x_{n}$ (that is, the space of all homogeneous
polynomials of degree $1$ in $\mathbf{B}$). Then, the $F$-linear map $\psi$
sends the basis $\left(  x_{1},x_{2},\ldots,x_{n}\right)  $ of $W$ to the
basis $\left(  v_{1},v_{2},\ldots,v_{n}\right)  $ of $V$. Thus, $\psi$
restricts to a vector space isomorphism from $W$ to $V$. In particular, the
restriction $\psi\mid_{W}$ is injective, and we have $\psi\left(  W\right)
=V$. Moreover, the lines $L\subseteq V$ are exactly the images of the lines
$M\subseteq W$ under this isomorphism from $W$ to $V$, and this yields a
1-to-1 correspondence $M\mapsto\psi\left(  M\right)  $ between the lines
$M\subseteq W$ and the lines $L\subseteq V$.

However, Remark \ref{rmk.phi-bij.1} \textbf{(a)} (applied to $\mathbf{B}$
instead of $\mathbf{A}$) shows that $\mathbf{B}$ can be embedded into a larger
commutative $F$-algebra $\widehat{\mathbf{B}}$ whose Frobenius morphism
$\varphi:\widehat{\mathbf{B}}\rightarrow\widehat{\mathbf{B}}$ is invertible.
Consider this larger algebra $\widehat{\mathbf{B}}$. Moreover,
$\widehat{\mathbf{B}}$ is an integral domain (again by Remark
\ref{rmk.phi-bij.1} \textbf{(a)}). Hence, we can apply Theorem
\ref{thm.skew-V/L} to $\widehat{\mathbf{B}}$, $W$ and $\varnothing$ instead of
$\mathbf{A}$, $V$ and $\mu$. This yields
\[
S_{\lambda/\varnothing}\left(  W\right)  =\sum_{L\subseteq W\text{ line}%
}S_{\lambda/\varnothing}\left(  W\sslash L\right)  =\sum_{M\subseteq W\text{
line}}S_{\lambda/\varnothing}\left(  W\sslash M\right)  .
\]
In view of (\ref{eq.Slam0}), we can rewrite this as
\begin{equation}
S_{\lambda}\left(  W\right)  =\sum_{M\subseteq W\text{ line}}S_{\lambda
}\left(  W\sslash M\right)  . \label{pf.thm.7.25.4}%
\end{equation}
This is an equality in $\widehat{\mathbf{B}}$, thus an equality in
$\mathbf{B}$ (since both of its sides lie in the subring $\mathbf{B}$ of
$\widehat{\mathbf{B}}$).

Let us now apply the $F$-algebra morphism $\psi:\mathbf{B}\rightarrow
\mathbf{A}$ to both sides of this equality. Thus, we find%
\begin{align}
\psi\left(  S_{\lambda}\left(  W\right)  \right)   &  =\psi\left(
\sum_{M\subseteq W\text{ line}}S_{\lambda}\left(  W\sslash M\right)  \right)
\nonumber\\
&  =\sum_{M\subseteq W\text{ line}}\psi\left(  S_{\lambda}\left(
W\sslash M\right)  \right)  \label{pf.thm.7.25.5}%
\end{align}
(since $\psi$ is $F$-linear). However, Lemma \ref{lem.funcS} (applied to $W$
instead of $V$) yields $\psi\left(  S_{\lambda}\left(  W\right)  \right)
=S_{\lambda}\left(  \psi\left(  W\right)  \right)  $ (since the restriction
$\psi\mid_{W}$ is injective). In view of $\psi\left(  W\right)  =V$, we
rewrite this as $\psi\left(  S_{\lambda}\left(  W\right)  \right)
=S_{\lambda}\left(  V\right)  $. Comparing this with (\ref{pf.thm.7.25.5}), we
obtain%
\begin{equation}
S_{\lambda}\left(  V\right)  =\sum_{M\subseteq W\text{ line}}\psi\left(
S_{\lambda}\left(  W\sslash M\right)  \right)  . \label{pf.thm.7.25.6}%
\end{equation}

Now, let $M\subseteq W$ be a line. Then, Lemma \ref{lem.func//} (applied to
$U=M$) yields that $\psi\mid_{W\sslash M}$ is an $F$-vector space isomorphism
from $W\sslash M$ to $\psi\left(  W\right)  \sslash\psi\left(  M\right)  $.
Hence, in particular, the restriction $\psi\mid_{W\sslash M}$ is injective,
and we have $\psi\left(  W\sslash M\right)  =\underbrace{\psi\left(  W\right)
}_{=V}\sslash\psi\left(  M\right)  =V\sslash\psi\left(  M\right)  $.
Therefore, Lemma \ref{lem.funcS} (applied to $W\sslash M$ instead of $V$)
yields
\[
\psi\left(  S_{\lambda}\left(  W\sslash M\right)  \right)  =S_{\lambda}\left(
\underbrace{\psi\left(  W\sslash M\right)  }_{=V\sslash\psi\left(  M\right)
}\right)  =S_{\lambda}\left(  V\sslash\psi\left(  M\right)  \right)  .
\]

Forget that we fixed $M$. We thus have proved that
\[
\psi\left(  S_{\lambda}\left(  W\sslash M\right)  \right)  =S_{\lambda}\left(
V\sslash\psi\left(  M\right)  \right)  \ \ \ \ \ \ \ \ \ \ \text{for each line
}M\subseteq W.
\]
Thus, we can rewrite (\ref{pf.thm.7.25.6}) as%
\[
S_{\lambda}\left(  V\right)  =\sum_{M\subseteq W\text{ line}}S_{\lambda
}\left(  V\sslash\psi\left(  M\right)  \right)  =\sum_{L\subseteq V\text{
line}}S_{\lambda}\left(  V\sslash L\right)
\]
(here, we have substituted $L$ for $\psi\left(  M\right)  $ in the sum, since
we have a 1-to-1 correspondence $M\mapsto\psi\left(  M\right)  $ between the
lines $M\subseteq W$ and the lines $L\subseteq V$). Thus, Theorem
\ref{thm.7.25} is proved.
\end{proof}

\end{document}